\documentclass[10pt,reqno]{amsart}

\usepackage{color}
\usepackage{amsmath, amsthm, amssymb}
\usepackage{amsfonts}
\usepackage[ansinew]{inputenc}
\usepackage[dvips]{epsfig}
\usepackage{graphicx}
\usepackage[english]{babel}
\usepackage{hyperref}

\usepackage{cite}
\usepackage{graphicx}
\usepackage{amscd}
\usepackage{color}
\usepackage{bm}
\usepackage{enumerate}

\usepackage{verbatim}
\usepackage{hyperref}
\usepackage{amstext}
\usepackage{latexsym}
\theoremstyle{plain}
\newtheorem{theorem}{Theorem}[section]
\newtheorem{definition}[theorem]{Definition}
\newtheorem{corollary}[theorem]{Corollary}
\newtheorem{proposition}[theorem]{Proposition}
\newtheorem{lemma}[theorem]{Lemma}
\newtheorem{remark}[theorem]{Remark}

\numberwithin{theorem}{section}
\numberwithin{equation}{section}

\newcommand{\average}{{\mathchoice {\kern1ex\vcenter{\hrule height.4pt
width 6pt depth0pt} \kern-9.7pt} {\kern1ex\vcenter{\hrule
height.4pt width 4.3pt depth0pt} \kern-7pt} {} {} }}

\def\R{\mathbb{R}}

\renewcommand{\a }{\alpha }
\renewcommand{\b }{\beta }
\renewcommand{\d}{\delta }

\newcommand{\D }{\Delta }

\newcommand{\e }{\varepsilon }
\newcommand{\g }{\gamma}

\newcommand{\G }{\Gamma}
\renewcommand{\l }{\lambda }

\newcommand{\n }{\nabla }
\newcommand{\vp }{\varphi }

\newcommand{\s }{\sigma }
\newcommand{\Sig }{\Sigma}
\renewcommand{\t }{\tau }
\newcommand{\z }{\zeta}
\renewcommand{\th }{\theta }

\renewcommand{\O }{\Omega }

\newcommand{\ov}{\overline}

\newcommand{\be}{\begin{equation}}
\newcommand{\ee}{\end{equation}}

\newcommand{\de}{\partial}

\newcommand{\ti}{\widetilde}

\newcommand{\ra}{{\rangle}}
\newcommand{\la}{{\langle}}

\renewcommand{\k}{\kappa}

 \usepackage{mathrsfs}

\newcommand{\scrK }{\mathscr{K}}

\newcommand{\scrP }{\mathscr{P}}

\newcommand{\calD }{\mathcal{D}}

\newcommand{\mbL }{\mathbb{L}}

\newcommand{\N}{\mathbb{N}}

\newcommand{\cA}{{\mathcal A}}
\newcommand{\cB}{{\mathcal B}}
\newcommand{\cC}{{\mathcal C}}
\newcommand{\cD}{{\mathcal D}}
\newcommand{\cE}{{\mathcal E}}
\newcommand{\cF}{{\mathcal F}}

\newcommand{\cJ}{{\mathcal J}}
\newcommand{\cK}{{\mathcal K}}
\newcommand{\cL}{{\mathcal L}}
\newcommand{\cM}{{\mathcal M}}
\newcommand{\cN}{{\mathcal N}}
\newcommand{\cO}{{\mathcal O}}

\renewcommand{\epsilon}{\varepsilon}

\newcommand{\Ds}{ (-\D)^s}

\begin{document}

\title 
{Regularity results   for nonlocal   equations and applications}

\author[Mouhamed  Moustapha Fall]
{Mouhamed  Moustapha Fall}
\address{M.M.F.: African Institute for Mathematical Sciences in Senegal, 
KM 2, Route de Joal, B.P. 14 18. Mbour, S\'en\'egal}
\email{mouhamed.m.fall@aims-senegal.org, mouhamed.m.fall@gmail.com}

\thanks{ 
 The  author's work is supported by the Alexander 
von Humboldt foundation.  He thanks Joaquim Serra, for his interest in this work and with  whom he  had stimulating discussions that help  to improve the first section of this paper. He also thanks the anonymous referee for useful comments.
}


 \begin{abstract}
   \noindent
  We introduce the concept of $C^{m,\a}$-nonlocal  operators, extending the notion of second order  elliptic  operator in divergence form with $C^{m,\a}$-coefficients.  We then derive the nonlocal  analogue of the  key existing results for elliptic equations in divergence form,  notably the H\"older continuity of the gradient of the solutions in the case of   $C^{0,\a}$-coefficients and the classical Schauder estimates for    $C^{m+1,\a}$-coefficients. 
        We further apply the regularity  results for $C^{m,\a}$-nonlocal  operators    to derive optimal  higher order    regularity estimates of Lipschitz graphs with prescribed Nonlocal Mean Curvature.  Applications to nonlocal equation on manifolds are also provided. 
 \end{abstract}

\maketitle


%
 \section{Introduction}
 We are concerned with a class of   (not necessarily translation invariant) elliptic  equations driven by nonlocal operators of fractional order. We extend in the nonlocal setting  some key existing results for elliptic equations in divergence form with $C^{m,\a}$-coefficients.  For a better description   of   how far the results   in this paper extend to the fractional setting those available in the classical case, we start by recalling    some    main results of the classical local theory.  
 We  consider a weak solution $u\in H^1(\O)$ to the equation 
\be\label{eq:u-solveloc-PDE}
\sum_{i,j=1}^N \de_i(a_{ij}(x)\de_j u)=f \qquad\textrm{ in $\O$},
\ee
where, $\O$ is a bounded open subset of $\R^N$, $f\in L^p_{loc}(\O)$, $p>N/2$,  and  the matrix coefficients   $a_{ij} $ are measurable functions and satisfy, for every $x\in \O$, the following properties:
 \be \label{eq:a-satisf-elliptic}
 \begin{aligned}
(i)\,& a_{ij}(x)=a_{ji}(x) \qquad &\textrm{ for all $i,j=1,\dots,N$,}\\
(ii)\,& \k \d_{ij}\leq a_{ij}(x)\leq \frac{1}{\k} \d_{ij} &\qquad\textrm{ for all $i,j=1,\dots,N$.}
 %
 \end{aligned}
 \ee
 In the  regularity   theory  for elliptic equations in  {divergence form} with  {measurable coefficients}, the   De Giorgi-Nash-Moser  theory provides a priori    $C^{0,\a_0} (\O)$ estimates for weak solutions to \eqref{eq:u-solveloc-PDE}, for some $\a_0=\a_0(N,p,\k)$, see e.g. \cite{GT}. The range or value of the largest H\"older exponent $\a_0$ is known in general  once the   coefficients are sufficiently regular. For instance,  if $a_{ij}\in C(\O)$ then     $u\in C ^{0,\b}_{loc}(\O)$ for all $\b<\min(2-N/p,1)$.  Now H\"older continuous  coefficients $a_{ij}$ yield H\"older continuity of the gradient   of $u$. Namely, if $a_{ij}\in C^{0,\a}(\O)$, for some $\a\in (0,1)$, then $u\in C^{1,\min( 1-\frac{N}{p},\a)} (\O)$, provided $2-N/p>1$. Moreover the Schauder theory states that if $a_{ij}\in C^{m+1,\a}(\O)$ and  $f\in C^{m,\a}(\O)$,  then $u\in C^{m+2,\a} (\O)$ for $m\in\N$. We refer the reader to \cite{GT, Texi}. Notable applications are the smoothness character of  variational solutions, including   the regularity of    critical points of the integral  functional  
\be \label{eq:integ-functinal}
  \cJ(u):=\int_{\O}G(\n u (x))\, dx,
\ee
 for some twice differentiable  function $G$.   As a matter of fact, the above regularity results provides a systematic proof of   the Hilbert's $19^{\textrm{th}}$ problem stating that if $G\in C^\infty(\R^N)$,    then the minimizer of $\cJ$ is of class $C^\infty$ as well. This was   solved by de Giorgi  in \cite{DeGiorgi}.  Indeed, given a critical point $u\in H^1(\O)$ to $\cJ$, we have that $\frac{u(x+h)-u(x)}{|h|}$ solves an equation as in \eqref{eq:u-solveloc-PDE} with coefficients $a_{ij}(x)=\int_0^1\de^2_{ij }G(\varrho \n u(x+h)+(1-\varrho)\n u(x))d\varrho$ satisfying \eqref{eq:a-satisf-elliptic} as soon as $D^2G$ is   uniformly bounded from above and below on $\R^N$. Therefore the  fact that $a_{ij}$ is  as smooth as $\n u$ immediately implies that $u$   smooth, thanks to above regularity results for divergence type operators. On the other if $G(\z)=\sqrt{1+|\z|^2}$, then \eqref{eq:integ-functinal} becomes the area functional, and in this case $D^2G$ is not bounded from below. However, this gap can be filled by assuming that  $u$    Lipschitz.  \\
 The aim of this paper is to extend all  the above regularity   results to  equations driven by  $C^{m,\a}$-nonlocal operators of fractional order   which we describe below.  Our  notion of   $C^{m,\a}$-nonlocal operators can be seen as a nonlocal version of second order partial differential equations in divergence form.   On the other hand, as in the local case,  since our notion of $C^{m,\a}$-nonlocal operators   is stable under $C^{m,1}$ local change of coordinates,   our results  apply to nonlocal equations on  manifolds  nonlocal geometric problems such as the prescribed nonlocal  mean curvature problems.   \\

We start introducing the class of kernels (defining nonlocal operators) we will use in the remaining of the paper.  We consider $s\in (0,1)$, $N\geq 1$  and $K:\R^N\times \R^N \to [-\infty,+\infty] $ such that
 \be \label{eq:Kernel-satisf}
 \begin{aligned}
(i)\,& K(x,y)=K(y,x) \qquad &\textrm{ for all $(x,y)\in\R^N\times \R^N$,}\\
  (ii) \,&    |K(x,y)|  \leq \frac{1}{\k}    |x-y|^{-N-2s} \qquad&\textrm{ for all $(x,y)\in\R^N\times \R^N$,}\\ 
(ii')\,& \k |x-y|^{-N-2s}\leq K(x,y)  \qquad &\textrm{ for all $(x,y)\in B_\d\times B_\d$,} 
 %
 %
 \end{aligned}
 \ee 
 for some constants $\k, \d>0$. We call ${L_s(\R^N)}$  the space  of function   $u\in L^1_{loc}(\R^N)$ such that 
 $$ \|u\|_{{L_s(\R^N)}}:=\int_{\R^N}|u(y)|(1+|y|)^{-N-2s}\, dy<\infty.$$   A kernel $K$ satisfying \eqref{eq:Kernel-satisf}$(i)$-$(ii)$ induces a linear  nonlocal operator  $ \cL_K: H^s (\O)\cap {L_s(\R^N)}\to \calD'(\O)$ given by 
 $$
\la   \cL_K u, \psi \ra:=  \frac{1}{2}\int_{\R^N \times \R^N}(u(x)-u(y))(\psi(x)-\psi(y))K(x,y)dxdy \qquad\textrm{ for all $\psi\in C^\infty_c(\O)$.}
 $$
 The weight in the definition of the  space ${L_s(\R^N)}$   is determined by \eqref{eq:Kernel-satisf}-$(ii)$ and can be modified accordingly.  
Given  $f\in L^1_{loc}(\R^N)$, 
 we say that $u\in H^s (\O)\cap {L_s(\R^N)} $  is a (weak) solution   to the equation 
\be  \label{eq:Main-problem}
\cL_K u  = f \qquad\textrm{ in $ \O$} ,
\ee 
if  $\cL_K u =f$ in $\calD'(\O)$.  

 The class of operators $\cL_K$  induced by the kernels $K$ satisfying  \eqref{eq:Kernel-satisf}  are  the nonlocal analogue of  second order elliptic operators in divergence form with measurable coefficients on $B_\d$.   In this case the de Giorgi-Nash-Moser  a priori  H\"older estimates is well developed, see \cite{Dicastro, KMS, KRS,KS,Cozzi, DK,Mosconi,CCV,Min-JEMS}. In particular,  it follows from  \cite{DK} that, if $f\in L^p(B_\d)$, for some   $p>N/(2s)$, then $u\in C^{0,\a_0}_{loc}(B_\d)$, for some $\a_0=\a_0(N,s,p)>0$.\\
The study of nonlocal variational  equations involving general kernels,  satisfying e.g. \eqref{eq:Kernel-satisf}, is currently an intensive research area.  In particular several papers deals with existence and, a part in some specific cases e.g. fractional Lapalcian,  anisotropic fractional Lapalcian, regional fractional Laplacian,   the "smoothness" properties of the nonlocal operators leading to higher order regularity   of weak solutions  (e.g. $C^{1,\a}$  or  $C^{2s+\a}$-regularity) remains an open questions.   In the case of nonlocal and non-translation invariant operators (say in non-divergence form),  different assumptions on the kernels yielding higher order regularity are present in the literature,   starting from the work of Caffarelli-Silvestre \cite{CS2}, followed by many others e.g. \cite{Kriventsov,Serra,Jin-Xiong}.   A   first difficulty to address this question in the variational framework   is the singular character of the kernel $K$ (satisfying  \eqref{eq:Kernel-satisf}) at the diagonal points $x=y$ which  encodes also the order of regularity   of the solution, regardless the behaviour of the tail.
 In \cite{Fall-Reg}, we attempted to answer this question and introduced a notion of nonlocal operators with "continuous" coefficient, and we proved some  optimal interior and boundary regularities  of solutions to \eqref{eq:Main-problem}. In that paper, we also proved that $u$ is a classical solution provided the operator $\cL_K$ is smooth enough together with $C^{1,\a}$ estimates for translation invariant problems. These results will be sharpened and generalized in the present work.\\
Following \cite{Fall-Reg},   we now introduce the notion of $C^{m,\a}$-nonlocal (or fractional order) operators which, in particular,  are the object of study in the present paper.
\begin{definition}\label{def:nonloc-very-reg}
 For $\d>0$, we define   $Q_\d:=B_\d \times [0,\d)$. Let  $\a\in [0,1)$, $m\in \N$ and $K$ satisfy  \eqref{eq:Kernel-satisf}.  
\begin{itemize}

\item We say that  the kernel $K$ defines a $C^{m,\a}$-nonlocal operator  in   $Q_\d$,  if the function
 $$
B_\d \times (0,\d)  \times S^{N-1}\to \R, \qquad (x,r,\th)\mapsto r^{N+2s}K(x,x+r\th)
 $$  
 extends to a map $ {\cA}_K : Q_\d \times S^{N-1}\to \R$ satisfying,  for some $\k>0$,  the following properties:
 \be \label{eq:NL-Holder00}
  \begin{aligned}
 %
(iii)\,& \|{\cA}_K \|_{C^{m,\a} (Q_\d\times S^{N-1})} \leq \frac{1}{\k},   \\
 (iv)\,&  {\cA}_K(x,0,\th)=   {\cA}_K(x,0,-\th)  & \qquad \textrm {for all $ (x,\th) \in  B_\d  \times S^{N-1}$.}  
 %
 \end{aligned}
 \ee
\item The class of kernels $K$ satisfying    \eqref{eq:Kernel-satisf} and \eqref{eq:NL-Holder00} is denoted by $ \scrK^s(\k,m+\a, Q_{\d})$. 
\end{itemize} 
\end{definition}
A simple example in the class of kernels in Definition \ref{def:nonloc-very-reg}
  is the the one of the fractional Laplacian, where the kernel is given by $K(x,y)=  |x-y|^{-N-2s}$. 
We remark that  the  class of operators induced by the kernels in Definition \ref{def:nonloc-very-reg}     provides a naturally extension of second order elliptic operators with $C^{m,\a}$-coefficients. Indeed,     the computations in \cite[Section 5]{Buc-Sq} show,   for all $\psi\in C^1_c(B_\d)$, that
\be\label{eq:limit-tolocal}
 {(1-s)} \int_{\R^N \times \R^N} (\psi(x)-\psi(y))^2K(x,y)dxdy\to\frac{1}{2} \sum_{i,j=1}^N \int_{\R^N}a_{ij}^K(x)\de_i \psi(x) \de_j \psi(x)\, dx \qquad\textrm{as $s\to 1$,}
\ee
 where $ a_{ij}^K(x)=\int_{S^{N-1}} \cA_{K}(x,0,\th) \th_i\th_j\, d\th$. Hence \eqref{eq:NL-Holder00}-$(iv)$ implies the symmetry of the matrix $( a_{ij}^K)_{1\leq i,j\leq N}$.
 \begin{remark}
 In \eqref{eq:NL-Holder00}-$(iii)$,  we impose   the regularity of $\cA_K$ in the angular variable $\th$. However,  this is typically not necessary to derive the accurate  local behavior of   solutions to \eqref{eq:Main-problem} which parallels   those solving \eqref{eq:u-solveloc-PDE} as  stated above. In fact,  nonlocal operators   provide a wider framework than their local counterpart, since translation invariant nonlocal operators are those given by  kernels $K$ of the form $K(x,y)=J(x-y)$, for some even function $J$. In addition, only in this translation invariant   setting, regularity theory is already rich enough to include fully nonlinear problems, \cite{ CS1,CS2,CS3,Kriventsov,Serra,Jin-Xiong }. 
This  issue on the possible \textit{anisotropic} regularity of $\cA_K$ in its variables    will be taken into account in our main results     stated in Section \ref{ss:General-nonlocaloperator} below.
 \end{remark}
We now  start by  stating the main results concerning $C^{m,\a}$-nonlocal operators. Their generalizations   are contained in Section \ref{ss:General-nonlocaloperator} below.
Our first main result is the following. 

 \begin{theorem}\label{th:main-th10}
Let $s\in (0,1)$, $N\geq 1$, $\k>0$  and $\a\in (0,1)$.   Let $K\in  \scrK^s(\k,\a, Q_2)$, $u\in H^s(B_2)\cap {L_s(\R^N)}$ and $V, f\in L^p(B_2)$, for some $p>N/(2s)$, satisfy
 $$
 \cL_K u+ Vu = f\qquad \textrm{in $B_2$}.
 $$
  \begin{itemize}
 \item[$(i)$] If $2s\leq 1$, then there exists $C=C(s,N,\k, \a,p, \|V\|_{L^p(B_2)})>0$ such that  
\be \label{eq:thm-E10pp}
\|u\|_{C^{0,2s-\frac{N}{p}}(B_1)}\leq C(\|u\|_{L^2(B_2)}+\|u\|_{{L_s(\R^N)}}+ \|f\|_{L^p(B_2)}).
\ee
 \item[$(ii)$] If $2s-1>\max(\frac{N}{p},\a)$, then   there exists $C=C(s,N,\k, \a,p, \|V\|_{L^p(B_2)})>0$ such that
\be  \label{eq:thm-E20pp}
\|u\|_{C^{1,\min (2s-\frac{N}{p}-1,\a)}(B_1)}\leq C(\|u\|_{L^2(B_2)}+\|u\|_{{L_s(\R^N)}}+ \|f\|_{L^p(B_2)}).
\ee
 \end{itemize}
 \end{theorem}
  The H\"older continuity of the gradient   in \eqref{eq:thm-E20pp} is the main novelty in the above result. Theorem \ref{th:main-th10} was only known in the translation invariant case, i.e. $K(x,y)=J(x-y)$, see \cite{Fall-Reg}.  We mention that    the regularity estimate in \eqref{eq:thm-E10pp} remains valid if $\a=0$, see  \cite{Fall-Reg}, where it was proven that if  $K\in \scrK^s(\k,0, Q_2)$ (and for all $s\in (0,1)$), then  $u\in  C^{0,\b}(B_1)$ for all $\b<\min(2s-\frac{N}{p},1)$.  In view of \eqref{eq:limit-tolocal}, it will be apparent from the proof that the estimates in Theorem \ref{th:main-th10}  remain stable as $s\to1$ once we replace $\cL_K$ by $(1-s)\cL_K$ and provided $p>\frac{N}{2s_0}$, with $s_0\in (0,1)$.\\
   We recall that H\"older continuity of the gradient of solutions to fully nonlinear and non translation invariant integro-differential equations, in the spirit of Cordes and Nirenberg for elliptic equations in nondivergence form,  has been first established by Caffarelli and Silvestre in \cite{CS2}, see also \cite{Kriventsov,Serra,Jin-Xiong } for higher order regularity estimates in nonlocal problems corresponding to elliptic equations in non-divergence form. \\ 
 Our next results is concerned with  $C^{m+2s+\a}$ regularity estimates for solutions to equations driven by  $C^{m+(2s-1)_++\a}$-nonlocal operators, provided $2s+\a\not\in \N$.  Here and in the following, we put  $\ell_+=\max(\ell,0)$ for $\ell\in \R$.    
  \begin{theorem}\label{th:Schauder-0-intro}
Let $N\geq 1$,  $s\in (0,1)$ and $\k>0$. Let $m\in \N$ and  $\a\in (0,1)$, with $2s+\a\not\in \N$.  
  Let $K\in \scrK^s(\k, m+\a+ (2s-1)_+, Q_2)$,  $u\in H^s(B_2)\cap L^{\infty}(\R^N)$ and $f\in C^{m,\a}(B_2)$ such that 
 $$
 \cL_K u= f\qquad \textrm{in $B_2$}.
 $$
   \begin{itemize}
 \item[$(i)$] If   $2s+\a<1$, then
$$
\|u\|_{C^{m,2s+\a}(B_1)}\leq C(\|u\|_{L^\infty(\R^N)}+ \|f\|_{C^{m,\a}(B_2)}).
$$
 \item[$(ii)$] If $1<2s+\a<2$ and  $2s\not=1$,  then
$$
\|u\|_{C^{m+1,2s+\a-1}(B_1)}\leq C(\|u\|_{L^\infty(\R^N)}+ \|f\|_{C^{m,\a}(B_2)}).
$$
 \item[$(iii)$] If $2<2s+\a$,   then
$$
\|u\|_{C^{m+2,2s+\a-2}(B_1)}\leq C(\|u\|_{L^\infty(\R^N)}+ \|f\|_{C^{m,\a}(B_2)}).
$$
 \item[$(iv)$] If   $2s=1$,  then for all $\b\in (0,\a)$, 
$$
\|u\|_{C^{m+1,\b}(B_1)}\leq C(\|u\|_{L^\infty(\R^N)}+ \|f\|_{C^{m,\b}(B_2)}).
$$
 \end{itemize}
 Here $C=C(N,s,\k,\a,\b, m).$
 \end{theorem}
It is clear  that Theorem \ref{th:Schauder-0-intro} includes  the   fractional Laplacian $\cL_K=\Ds$, for which it  was  proven in  \cite{DSV,Sil,Grubb1,RS2}. \\


Theorem \ref{th:main-th10} and Theorem \ref{th:Schauder-0-intro}  provide regularity of minimizers of integral energy functional e.g.  of the  form   
\be \label{eq:cJs}
\cJ_{s}(u):= (1-s) \int_{\R^{2N}\setminus (\R^N\setminus \O)^2} F\left(\frac{u(x)-u(y)}{|x-y|} \right)  |x-y|^{-N-2s+2}\, dxdy ,
\ee
for some twice differentiable function  $F$. The   case $F(t)=t^2$ is the well known localized (in $\O$) Dirichlet energy for equations involving the  fractional Laplacian.    The minimization of this  energy should be subject to exterior boundary data on $\R^N\setminus \O$, and posses a minimizer on $H^s(\R^N)$ under some quadratic and convexity  assumption on $F$.  To see how  $\cJ_s$ is related with  \eqref{eq:integ-functinal}, we assume that  $   |F(z)|\leq  |z|$. Then, for all $u\in C^1_{c}(\R^N)$, 
$$
\lim_{s\to   1 }\cJ_{s}(u)\to \int_{\O} G(\n u(x))\, dx, 
$$   
where $G(\z)=\frac{1}{2}\int_{S^{N-1}}F_e(\z\cdot \th)\, d\th$ and $F_e$ is the even part of $F$ i.e., $F_e(t)=\frac{F(t)+F(-t)}{2}$. The nonlocal    de Giorgi-Nash-Moser  provides a priori estimates for minimizers of $\cJ_s$. Indeed,  a  critical point $u\in H^s(\R^N)$ to $\cJ_s$  satisfies
\be \label{eq:EL-eq-gen}
 \int_{\R^N\times\R^N}  \left\{F'(p_u(x,y))  - F'(-p_u(x,y)) \right\}(\psi(x)-\psi(y))  |x-y|^{-N-2s+1}dxdy=0 \qquad\textrm{for all $\psi\in C^\infty_c(\O)$,}
\ee
  where  $p_u(x,y):= \frac{u(y)-u(x)}{|y-x|}$.
Therefore, following the classical de Giorgi's trick and using  the fundamental theorem of calculus, we find that   the difference quotient $u_h(x)=\frac{u(x+h)-u(x)}{|h|}$ solves  the equation 
$$
\cL_{K_{F,u,h}}u_h=0  \qquad\textrm{ in $\O$,}
$$
where
\be \label{eq:K-Q-u-h}
 K_{F,u,h}(x,y)=  |x-y|^{-N-2s}\int_0^1F''_e\left(  \varrho p_{u(\cdot+h)}(x,y) +(1-\varrho)p_u(x,y)  \right) d\varrho,
\ee
and   $F_e$ is the  even part of $F$. We then immediately see that    $K_{F,u,h}$   satisfies \eqref{eq:Kernel-satisf} when $F''$ is bounded from above and below on $\R$. Consequently, the nonlocal de Giorgi-Nash-Moser theory implies that $u\in C^{1,\a_0}(\O)$, for some $\a_0>0$. 
Now an iterative application of  Theorem \ref{th:main-th10} and Theorem \ref{th:Schauder-0-intro} shows, as for the solution to the Hilbert's problem,  that  if $F\in C^\infty(\R) $  then  $u\in C^\infty(\O)$.
It is worth to mention that  in the  nonlocal mean curvature problem, nonlocal minimal graphs satisfy an equation as in   \eqref{eq:EL-eq-gen}, with $F(t)= \int_{t}^{\infty}(1+\t^2)^{-\frac{N+2s}{2}}d\t$ (see Section \ref{ss:NMC} below for a more precise statement). Here,  $F''$ is not uniformly bounded from below and thus $ K_{F,u,h}$ does not satisfy \eqref{eq:Kernel-satisf}-$(ii')$.   However, as in the classical case, this lack of ellipticity, is recovered  once we know that    $u$ is Lipschitz.  \\

%
%

Beyond   their appearances in the  mathematical modeling of real-world  phenomenon,  $C^{m,\a}$-nonlocal operators appear naturally in  geometric problems. Indeed, we are  naturally  confronted with nonlocal equation resulting from an initial one after a change coordinates. For instance, consider     $K(x,y)=|x-y|^{-N-2s}$ (the kernel of the fractional Laplacian) and   $K_\Phi(x,y)=|\Phi(x)-\Phi(y)|^{-N-2s}$, for some  diffeomorphism $\Phi\in C^{m+1,\a}(\R^N;\R^N)$ with $D\Phi$ close to the identity matrix, so that  \eqref{eq:Kernel-satisf} holds.  In this case,  apart in dimension $N=1$,  we may not have  any  regularity of   $z\mapsto  |z|^{N+2s}K_{\Phi}(x,x+z)$ at $z=0$. However, using polar coordinates, we easily see that the map 
$$
(x,r,\th)\mapsto r^{N+2s}K_{\Phi}(x,x+r\th)=\left| \int_0^1D\Phi(x+t r\th)\th \, dt\right|^{-N-2s}
$$
 extends to a  $C^{m,\a}$ map  on $\R^N\times [0,\infty)\times S^{N-1}$ satisfying \eqref{eq:NL-Holder00}-$(ii)$, so that $K_{\Phi}$ defines a $C^{m,\a}$-nonlocal operator. This is also the case  for the kernel in \eqref{eq:K-Q-u-h} with $F\in C^\infty(\R)$ and $u\in C^{m+1,\a}(\O)$.   These facts, among others,  motivate  the splitting in polar coordinates in our definition of $C^{m,\a}$-nonlocal operators.  Moreover, it turns out to be  useful in the study of  prescribed nonlocal  mean curvature   problems and    nonlocal equations on hypersurfaces, see Section \ref{ss:NMC} and Section \ref{ss:NonlocManifold}, respectively.   On the other hand, we remark that in some interesting non-translation invariant  cases,  the map $ z\mapsto  |z|^{N+2s}K (x,x+z)$ can be smooth at $z=0$, and  a first nontrivial example is given by the \textit{censored fractional Laplacian} or the  $\O$-regional fractional Laplacian,  where the kernel is given by $K(x,y)=1_{\O}(x)1_{\O}(y) |x-y|^{-N-2s}$, see e.g. Mou and Yi \cite{Mou}. An other example arises in problems from image processing, see e.g. Gilboa Osher \cite{GO} and  Caffarelli, Chan and Vasseur \cite{CCV}, where the kernel depends on the solution $u\in C^{1,\a_0}$ and,  for simplicity, reads as $K(x,y)=1_{\O}(x)1_{\O}(y)\phi''(u(x)-u(y)) |x-y|^{-N-2s}$, for some even and  convex function $\phi$.   
Here one looks at minimizer $u\in H^s(\O)$ of the energy functional 
$$
\cJ_{s,\O}(u):= (1-s) \int_{\O\times \O} \phi\left({u(x)-u(y)} \right)  |x-y|^{-N-2s}\, dxdy. 
 $$
 This is also the case for (possibly) sign-changing kernels e.g.  $K(x,y)=|x-y|^{-N-2s_1}\pm|x-y|^{-N-2s_2}$, with $s_1\in (0,1)$ and  $ s_2<s_1$.   However the  conditions \eqref{eq:Kernel-satisf} and  \eqref{eq:NL-Holder00} are flexible enough to include such cases.\\

%

 
 The following two paragraphs are devoted to the application of the above regularity estimates   in some nonlocal geometric problems.
 \subsection{Application I: Graphs with prescribed nonlocal mean curvature}\label{ss:NMC}
 In this section, we  assume that   $s\in (1/2,1)$. Recall that  for a  set $E\subset\R^{N+1}$  of class $C^{1,2s-1+\a}  $, with $\a>0$, near a point $X\in \de E $,   the nonlocal (or fractional) mean curvature of the set $E$ (or the hypersurface $\de E$)   at the point $X\in \de E $ is defined as
\be  \label{eq:NMC-PV}
H_s(\de E;X):=   PV\int_{\R^{N+1}}\frac{1_{E^c}(Y)-1_{E}(Y)}{|Y-X|^{N+2s}}\, dY,
\ee
where  $E^c:=\R^{N+1}\setminus {E}$  and    $1_D$ denotes the characteristic function of a set $D \subset \R^{N+1}$.  
Recall that the notion of nonlocal mean curvature appeared first in the work  of Caffarelli and Souganidis in \cite{Caff-Soug2010} and first studied  by Caffarelli, Roquejoffre, and Savin in~\cite{Caffarelli2010}.  As first discovered in \cite{Caffarelli2010} (see also \cite{Davila2014B,FFMMM}), the nonlocal mean curvature arises as the first variation of the fractional perimeter.  For  the convergence of fractional curvature to the classical one as $s\to 1$, see \cite{Ab-Val,Davila2014B}. \\
Suppose that  $\de E$ is the graph of a function  $u\in  C^{1, 2s-1+\a}(\O)\cap C^{0, 1}_{loc}(\R^{N})$,     then see e.g. \cite{Fall-CNMC}, by a change of variable,  for all $x\in \O $, we have 
\begin{align}
 H_s(\de E;(x,u(x)))&=PV\int_{\R^{N}}\frac{\cF_s(p_u({x},{y}))- \cF_s(p_u({y},{x})) }{|{x}-{y}|^{N+2s-1}} d{y}, \label{eq:NMC_curve-E1}
   \end{align} 
   where
   \be \label{eq:def-of-F}
\cF_s(p):=\int_p^{+\infty} {(1+\t^2)^{\frac{-(N+2s)}{2}}}{d\t}
\ee
and  for a measurable function $w:\R^N\to \R$, we put
 \be \label{eq:P_u}
  p_w({x},{y})= \frac{w({y})-w({x})}{|{x}-{y}|} .
 \ee
By the fundamental theorem of calculus and \eqref{eq:NMC_curve-E1}  and noting that $\cF_s(p_u({y},{x})) =\cF_s(-p_u({x},{y})) $, we have 
\begin{align}
H_s(\de E;(x,u(x))) & =   PV\int_{\R^{N}}\frac{u({x})-u({y})}{|{x}-{y}|^{N+2s}}\,q_u(x,y) \, d {y},
 \label{eq:NMC_curve-E3}
   \end{align} 
where  for a measurable function $w:\R^N\to \R$,
$$
q_w(x,y):=-\int_{-1}^1\cF_s'(tp_w(x,y)  )\,dt
= \int_{-1}^1 {\left( 1+t^2 p_w(x,y)^2 \right)^{\frac{-(N+2s)}{2}}}dt.
$$
For the following,   we define the \textit{  the nonlocal mean curvature kernel}  by  
$$
\cK_w(x,y):= \frac{1}{|{x}-{y}|^{N+2s}}\,q_w(x,y)  \qquad\textrm{ for all $x\not=y\in \R^N$.}
$$
 Letting  $\O$ be an open set of $\R^{N}$ and $f\in L^1_{loc}(\O)$, we are interested in  the regularity of measurable functions $u:\R^N\to \R$   satisfying 
 \be\label{eq:weak-NMC-natur}
 \cL_{\cK_u} u=f \qquad\textrm{ in $\O$,} 
\ee
 or equivalently,
 \be\label{eq:decom-NMC-intro}
\frac{1}{2}\int_{\R^N\times \R^N}\frac{\cF_s(p_u({x},{y}))- \cF_s(p_u({y},{x})) }{|{x}-{y}|^{N+2s-1}} (\psi(x)-\psi(y)) \, dx dy= \int_{\R^N}f(x)\psi(x)\, dx \qquad\textrm{for all  $\psi\in C^\infty_c(\O)$.} 
 \ee
Note that, since $\cF_s\in L^\infty(\R)$ and $2s>1$, the right hand side in  \eqref{eq:decom-NMC-intro} is well defined. 
We observe that if $u\in C^{1, 2s-1+\a}(\O)\cap L^1_{loc}(\R^N) $   solves \eqref{eq:weak-NMC-natur},  then the set  $E_u:=\{(x,t)\in \R^N\times \R\,:\,u(x)<t\}$, satisfies  $H_s( \de E_u;(x,u(x)) )=f(x)$, for all $x\in \O$, provided the $(N+1)$-dimension Lebesgue measure of $\de E_u$ is equal to zero. This follows by approximating $u$ by a sequence of smooth functions.   \\
We consider next locally Lipschitz  graphs   with prescribed nonlocal mean curvature in the weak sense of \eqref{eq:decom-NMC-intro}, and we prove that they are of class $C^\infty$ in $\O$ as long as $f$ is $C^\infty$ in $\O$, with quantitative estimates. In the classical case, this   is a consequence of the de Giorgi-Nash theorem and   the Schauder
theory for uniformly elliptic equations in divergence form with $C^{m,\a}$-coefficients. See e.g.  Figalli and Valdinoci  \cite{Figalli-Valdinoci }, it was hardly believed that the same strategy could be carried out in the nonlocal setting. 
   In \cite{Figalli-Valdinoci }, the authors used geometric arguments to prove that Lipschitz sets, locally minimizing fractional perimeter are of class $C^\infty$. However their argument does   not provide quantitative estimates. 
Here, we shall show that it is indeed possible to proceed as in the prescribed mean curvature problem, thanks to our regularity estimates for $C^{m,\a}$-nonlocal operators.   
 It is important to note, in the theorem below, that we do not require any integrability of $u$ in $\R^N\setminus B_2$. We have the following result.
\begin{theorem}\label{eq:thm-nmc-reg}
Let $f\in L^{1}_{loc}(B_2)$ and  $u:\R^N\to \R$ be a measurable function,  with $\|u\|_{  C^{0,1}(B_2) }\leq c_0$, such that
$$
\cL_{\cK_u} u=f \qquad\textrm{ in $B_2$,}
$$
in the sense of \eqref{eq:decom-NMC-intro}. Then the following  statements hold.
\begin{itemize}
\item[$(i)$] If $f\in   C^{0,1}(B_2) $, then   
\be \label{eq:estimu-NMC-first}
\| u\|_{C^{1,\a_0}(B_1)}\leq C (1+\|f\|_{  C^{0,1}(B_2)  }),
\ee
for some constants $\a_0,C>0$, only depending on $  N,s$ and $c_0$.
Moreover, for all $\b\in (0,2s-1)$, 
$$
\| u\|_{C^{ 2,\b  }(B_{1})}\leq C ,
$$
for some  constant $C$, only depending on $ N,s,\b, c_0$ and     $\|f\|_{  C^{0,1}(B_2) }$.
\item[$(ii)$] If $f\in C^{m,\a}(B_2)$, for some $\a\in (0,1)$  and $m\geq 1$, then 
$$
\| u\|_{C^{ m+1,2s+\a-1  }(B_{1})}\leq C  \qquad\textrm{ if $2s+\a<2$,}
$$
$$
\| u\|_{C^{ m+2,2s+\a -2 }(B_{1})}\leq C  \qquad\textrm{ if $2s+\a>2$,}
$$
for some  constant $C$, only depending on  $ N,s,\a,m,c_0$ and     $\|f\|_{C^{m,\a}(B_2)}$.
\end{itemize}
\end{theorem}
The first quantitative estimates for nonlocal minimal graphs was found recently by Cabr\'e and Cozzi   in \cite{CC}. Indeed, they provide, in \cite{CC}, quantitative gradient estimates for global graphs that locally minimize the fractional area functional in a cylinder $B_R\times \R$, in the spirit of  Finn \cite{Finn} and  Bombieri, de Giorgi and Miranda \cite{Bomb}. In this case $f\equiv 0$. Therefore combining their result and Theorem \ref{eq:thm-nmc-reg}, we get  quantitative estimates of all partial derivatives of such  graphs in terms of the oscillation of $u$ in $B_R$. \\
%
Recall that the smoothness character for  fractional perimeter minimizing sets was known, but without quantitative bounds. Indeed, the seminal paper \cite{Caffarelli2010} established the first
existence and $C^{1,\g}$ (except a closed set of zero $(N-3)$-Hausdorff measure) regularity for fractional perimeter minimizing sets.  
In \cite{Barrios}, Barrios, Figalli and Valdinoci,    proved that       fractional perimeter minimizing sets which are of class  $C^{1,(2s-1)/2-\e}$ are of class $C^\infty$. On the other hand Caffarelli and Valdinoci showed, in \cite{Caffarelli2011B}, that, for $s$ close to 1, these sets possess the smoothness property of the classical perimeter minimizing regions.  It is proven  in \cite{DSV-NMC}, by Dipierro, Savin and Valdinoci, that the boundary of a    fractional perimeter minimizing set, in a reference smooth set $\O$,  which coincides with   a continuous graph  $\R^N\setminus \ov \Omega$ is in fact a  global graphs  that is continuous in $\O$.   \\ 

The fact that we do not require any integrability of $u$ in $\R^N\setminus B_2$ makes the proof of Theorem \ref{eq:thm-nmc-reg} particularly nontrivial. In view of the decomposition in \eqref{eq:decom-NMC-intro}, we split further the double integral in the left hand side to get
\begin{align}
\la \cL_{\cK_u}u,\psi \ra &=\frac{1}{2}\int_{\O\times \O} {(u(x)- u(y))(\psi(x)-\psi(y))  } \cK_u(x,y) \, dx dy \nonumber\\
&\quad+ \int_{\O} \psi(x) \int_{ \R^N\setminus \O}\frac{\cF_s(p_u({x},{y}))- \cF_s(p_u({y},{x})) }{|{x}-{y}|^{N+2s-1}}  \, dy dx.
\label{eq:decom-NMC-intro0}
\end{align} 
Now the proof of Theorem \ref{eq:thm-nmc-reg} resides on the  regularity of the map   
$$
\O'\to \R,\qquad  x\mapsto   \int_{ \R^N\setminus \O}\frac{\cF_s(p_u({x},{y}))- \cF_s(p_u({y},{x})) }{|{x}-{y}|^{N+2s-1}}  \, dy ,
$$ 
for $\O'\subset\subset\O$. Surprisingly, the local behavior of this map is completely determined by the one of $u$ only in $\O'$. In fact we will show, in Lemma \ref{lem:Lem-Gamma-u-nmc} below,  that this function is indeed as smooth as $u$ in $\O'$. Once this is proved, the above function is sent in the right hand side, so that  we can use the argument as in the classical case. Indeed,  we apply first    the nonlocal de Giorgi-Nash  a priori  H\"older estimate    to the function $\frac{u(x+h)-u(x)}{|h|}$ which satisfies a nonlocal equation of the form \eqref{eq:Main-problem}, driven by a kernel $K^u_h$ satisfying \eqref{eq:Kernel-satisf}, to deduce that     $\n u\in C^{0,\a_0}$. This will  imply  that $ K^u_h\in \scrK^s(\k,\a_0, Q_\d)$, for some $\k,\d>0$.   Now  Theorem \ref{th:main-th10}$(ii)$ and   Theorem \ref{th:Schauder-0-intro}$(ii)$   kick in and yield the result, since $\cA_{ K^u_h}$ will be, locally,  as regular as $\n u$.

\subsection{Application II: Nonlocal equations on  manifolds}\label{ss:NonlocManifold}
 Let   $\Sig $ be a Lipschitz   hypersurface of $\R^{N+1}$, with $0\in \Sig$. We define  the space  ${L_s(\Sig)}$ given by   the set of functions $u\in L^1_{loc}(\Sig)$ such that
  $$
\|u\|_{{L_s(\Sig)}}:=\int_{\Sig}|u(\ov y)| (1+|\ov y|)^{-N-2s}\, d\s(\ov y)<\infty, 
  $$  
   where $d\s$ denote the volume element on $\Sig$. 
  We assume that 
\be \label{eq:integ-hypersurface}
 \|1\|_{{L_s(\Sig)}}=\int_{\Sig}  (1+|\ov y|)^{-N-2s}\, d\s(\ov y)<\infty.
\ee 
 We  note  that this  condition    always holds when $\Sigma$ has finite diameter.
 %
 In this section we are interested in the regularity estimates of functions $u\in H^s_{loc} (\Sigma )\cap {L_s(\Sig)}$ satisfying, for all $\Psi\in C^\infty_c({\Sig})$,  
\be\label{eq:-u-weak-non-flat}
\frac{1}{2}\int_{ \Sig} \int_{ \Sig} \frac{(u(\ov x)-u(\ov y))(\Psi(\ov x)-\Psi(\ov y)) }{|\ov x-\ov y|^{N+2s}}\, d\s(\ov x)d\s(\ov y)+\int_{\Sig} V(\ov x) u(\ov x) \Psi(\ov x)\, d\s(\ov x) =\int_{\Sig} f(\ov x) \Psi(\ov x)\, d\s(\ov x),
\ee
  where  $f,V\in L^1_{loc}(\Sig)$ and $u V\in L^1_{loc}(\Sig)$.\\
  \begin{theorem}\label{th:nonloca-surf1}
Let $s,\g\in (0,1)$, $N\geq 1$ and  $\Sigma$ be a $C^{1,\g}$-hypersurface  of $\R^{N+1}$ as above satisfying   \eqref{eq:integ-hypersurface}.
Let $f,V\in L^p (\Sig)$, for some $p>\frac{N}{2s}$ and  $u\in H^s_{loc}(\Sigma )\cap {L_s(\Sig)}$ satisfy 
\eqref{eq:-u-weak-non-flat}.  Then the following estimates hold.
\begin{itemize}
\item[$(i)$] If $2s\leq 1$, then 
$$
\|u\|_{C^{2s-N/p}(B_\varrho\cap\Sig) }\leq C (\|u\|_{L^2(B_{2\varrho}\cap\Sig)}+ \|u\|_{{L_s(\Sig)}}+\|f\|_{L^p(\Sig)}).
$$
\item[$(ii)$] If  $2s-1>\max(\frac{N}{p},\g)$, then 
$$
\|u\|_{C^{1,\min( 2s-\frac{N}{p}-1 , \g)}(B_\varrho\cap\Sig)}\leq C (\|u\|_{L^2(B_{2\varrho}\cap\Sig )}+ \|u\|_{{L_s(\Sig)}}+\|f\|_{L^p(\Sig)}),
$$
%
Here $C,\varrho>0$ are constants only depending on $N,s,\g,p $,$\|V\|_{L^p(\Sig)}$,  $ \|1\|_{{L_s(\Sig)}}$ and the bound of the local geometry of $\Sigma$ near $0$.  
\end{itemize}
\end{theorem}

  In the case of higher order regularity, we obtain the
\begin{theorem}\label{th:nonloca-surf2}
Let $s,\a,\g\in (0,1)$, $N\geq 1$ and  $\Sigma$ be a $C^{1,\g}$-hypersurface  of $\R^{N+1}$ as above  satisfying   \eqref{eq:integ-hypersurface}.
Let $f,V\in C^{0,\a} (\Sig)$  and  $u\in H^s_{loc}(\Sigma )\cap {L_s(\Sig)}$ satisfy 
\eqref{eq:-u-weak-non-flat}. 
\begin{itemize}
\item[$(i)$] If $2s>1$ and  $\g\geq \a+2s-1$, then 
$$
\|u\|_{C^{1,2s-1+\a}(B_\varrho\cap\Sig) }\leq C (\|u\|_{L^2(B_{2\varrho}\cap \Sig)}+ \|u\|_{{L_s(\Sig)}}+\|f\|_{C^{0,\a}(\Sig)}).
$$
\item[$(ii)$] If $2s+\a<1$ and $\g\geq \a$, then 
$$
\|u\|_{C^{0,2s+\a}(B_\varrho\cap\Sig) }\leq C (\|u\|_{L^2(B_{2\varrho}\cap \Sig)}+ \|u\|_{{L_s(\Sig)}}+\|f\|_{C^{0,\a}(\Sig)}).
$$
\item[$(iii)$] If $2s=1$ and $\g>\a$, then 
$$
\|u\|_{C^{1,\a}(B_\varrho\cap\Sig) }\leq C (\|u\|_{L^2(B_{2\varrho}\cap \Sig)}+ \|u\|_{{L_s(\Sig)}}+\|f\|_{C^{0,\a}(\Sig)}).
$$
Here $C,\varrho>0$ are constants only depending on $N,s,\g,\a $,$\|V\|_{C^{0,\a}(\Sig)}$,  $ \|1\|_{{L_s(\Sig)}}$ and the bound of the local geometry of $\Sigma$ near $0$.  
\end{itemize}
\end{theorem}
Here, by the bound of the local geometry of $\Sigma$ near $0$, we mean the $C^{1,\g}$ norm of a local parameterization of $\Sigma$ flattening $B_{\varrho_0}\cap\Sig$, for some $\varrho_0>0$.  If $\Sig$ is of class $C^{m+1,\g}$ and $f,V\in C^{m,\a}_{loc}(\Sig)$, then under the same  assumptions on $\g$ in Theorem \ref{th:nonloca-surf2}, we have the  estimates of $C^{m+2s+\a}$-norm of $u$ as long as $2s+\a\not\in \N$, thanks to Theorem \ref{th:Schauder-0-intro}.\\
   Theorem \ref{th:nonloca-surf1} and  Theorem \ref{th:nonloca-surf2} are consequences of Theorem \ref{th:main-th10} and Theorem \ref{th:Schauder-0-intro}, respectively,   after using a coordinate system that locally   flattens $ \Sig$.  \\

For $2s>1$,  Theorem \ref{th:nonloca-surf1} and  Theorem \ref{th:nonloca-surf2}  provide regularity estimates for solutions to some nonlocal equation driven by  the linearized nonlocal mean curvature operator (i.e. the nonlocal or fractional Jacobi operator) of a   set $E$ with constant nonlocal mean curvature (not necessarily bounded).    Indeed,  consider $\Sig:=\de E$  a $C^2$-hypersurface of $\R^{N+1}$ with constant nonlocal mean curvature such that  $0\in \de E$ and   $\|1\|_{{L_s(\Sig)}}<\infty $. See e.g. \cite{Davila2014B,FFMMM},    the second   variation of the fractional perimeter yields the bilinear form  $\cD_{\Sig}: H^s(\Sig)\times H^s(\Sig)\to \R$, given by 
$$
\cD_{\Sig}(u,v):=\frac{1}{2} \int_{\Sig}\int_{\Sig}\frac{(u(\ov x)-u(\ov y))(v(\ov x)-v(\ov y))}{|\ov x-\ov y|^{N+2s}}\, d\s(\ov y) d\s(\ov x) -\frac{1}{2}  \int_{\Sig}  V_{\Sig} (\ov x)u(\ov x)v(\ov x)d\s(\ov x),
$$
where, letting $\nu_\Sig$ be the unit exterior normal vector field of $\Sig:=\de E$,  
$$
 V_{\Sig} (\ov x):=\frac{1}{2} \int_{\Sig}\frac{| \nu_\Sig(\ov x)-\nu_\Sig(\ov y)|^2 }{|\ov x-\ov y|^{N+2s}}\, d\s(\ov y).
$$ 
One then defines the \textit{fractional Jacobi operator} as 
$$
\cJ_{\Sig}  :=\mbL_{\Sig} -V_{\Sig}  ,
$$
 where, for $u\in C^{1, 2s-1+\a}_{loc}(\Sigma)\cap {L_s(\Sig)}$,  
$$
 \mbL_{\Sig} u(\ov x):=PV\int_{\Sig}\frac{u(\ov x)-u(\ov y)}{|\ov x-\ov y|^{N+2s}}\, d\s(\ov y) .
$$
The \textit{fractional Jacobi fields} are   solutions to $\cJ_\Sig u=0$, and they play an important role in the study of stability of constant nonlocal mean curvature surfaces  or fractional area estimates of such surfaces.\\
We observe that if $\Sig$ is a $C^{1,\g}$-hypersurface for some $\g>s$, then $V_\Sig\in C^\g_{loc}(\Sig) $. Moreover we may consider a  weak solutions $u\in H^s_{loc}(\O)\cap{L_s(\Sig)}$ to the equation $\cJ_{\Sig} u=f$ on open subsets $\O$ of $\Sig$, in the sense of   \eqref{eq:-u-weak-non-flat}. Hence Theorem \ref{th:nonloca-surf1} and  Theorem \ref{th:nonloca-surf2} can be used to obtain regularity estimates of $u$. When $\Sig=S^{N-1}$, then Theorem \ref{th:nonloca-surf2}$(i)$ was proved in \cite{CFW-2017}, using the regularity theory of the fractional Laplacian and the Fredholm theory. Recall that   besides    the nonlocal minimal surfaces,  there exist several  nontrivial hypersurfaces  with  nonzero  constant nonlocal mean curvature, see e.g. the survey paper \cite{Fall-CNMC}.

\subsection{Anisotropic $C^{m,\a}$-nonlocal operators}\label{ss:General-nonlocaloperator}
As mentioned earlier, in many situations, nonlocal equations provide a wider framework than their local counterpart,  since   ${\cA}_K$ may have anisotropic regularity in its variables.   Namely, the spatial variable $x$,  the singular variable $r$ and the angular variable might have different qualitative properties. This    affects the local behavior of the solutions. First note that the class of operators $\cL_K$ falls in the class of nonlocal operators generated by a  L\'evy measure $\nu_x$.   In particular,   the map $z\mapsto K(x, x+z) $ is the density of a  L\'evy measure $\nu_x$ and thus does not necessarily posses any regularity. If the   L\'evy measure is symmetric and stable, then see \cite{RS2},  $\nu_x(r E)=r^{N-1}dr a(E)$ for $E\subset S^{N-1}$.    Under fairly general assumptions on the spectral measure $a$ on $S^{N-1}$ (not depending on  $x$), optimal interior and boundary regularity were proved by Ros-Oton and Serra in \cite{RS2}. The papers \cite{KRS,KS,DK} obtained also regularity estimates provided $a$ is absolutely continuous with respect to the Lebesgue measure on $S^{N-1}$ only on an open set of positive measure.\\
To capture this  possible anisotropic regularity of $\cA_K$ in its variables, we introduce a new class of fractional order nonlocal operators which are much larger than the class of $C^{m,\a}$-nonlocal operators introduced above. \\
In the following, for $\d>0$, we define 
\be\label{eq:def-Q-d-Q-infty}
Q_\d:=B_\d \times [0,\d)\qquad\textrm{ and }  \qquad Q_\infty:=\R^N \times [0,\infty).
\ee
We define the space $C^{m,\a}(Q_\d)\times L^\infty(S^{N-1} )$ by the set of functions $A\in L^\infty(Q_\d\times S^{N-1})$ such that, for every $\th\in S^{N-1}$,  the map $(x,r)\mapsto A(x,r,\th)$ belongs to $C^{m,\a}(Q_\d)$ and 
\be\label{eq:A-Cm12}
\|A\|_{C^{m,\a}(Q_\d)\times L^\infty(S^{N-1} )}:= \sup_{\th\in S^{N-1}}\|A(\cdot,\cdot,\th)\|_{C^{m,\a}(Q_\d)}<\infty  .
\ee
For  $\t\geq  0$,   the space  $ \cC^0_{\t}(Q_\d)\times L^\infty(S^{N-1} )$ is given by the  the set of function   $A\in L^\infty(Q_\d\times S^{N-1})$ such that 
$$
\| A\|_{ L^{\infty}_{\t}(Q_\d)\times L^\infty(S^{N-1} )}:=\sup_{\th\in S^{N-1}} \sup_{x\in B_\d, r\in (0,\d)}  \frac{| A(x,r,\th)| }{r^{\t}}<\infty 
$$
and
$$
[ A]_{ \cC^0_{\t}(Q_\d)\times L^\infty(S^{N-1} )}:= \sup_{  \th\in S^{N-1} }  \sup_{x\not=y\in B_\d, r\in (0,\d)}    \frac{| A(x,r,\th)  -   A(y,r,\th) |}{\min(r, |x-y|)^\t} <\infty .
$$
The space   $\cC^{m}_{\t}(Q_\d) \times L^\infty(S^{N-1} )$  is defined as  the set of functions $A\in C^{m,0}(Q_\d)\times L^\infty(S^{N-1} )$ such that 
\be \label{eq:A-Cm12-tau}
\| A\|_{ \cC^{m}_{\t}(Q_\d)\times L^\infty(S^{N-1} )}:= \sup_{\g\in \N^N, |\g|\leq m}  \|\de_x^\g  A\|_{L^{\infty}_{\t}(Q_\d) \times L^\infty(S^{N-1} )}+ \sup_{\g\in \N^N, |\g|\leq m} [\de_x^\g A ]_{\cC^0_{\t}(Q_\d)\times L^\infty(S^{N-1} ) }<\infty.
\ee
This   section is concerned with  optimal  H\"older estimates for nonlocal equation driven by the operator $\cL_K$ with coefficient $\cA_K$ in the spaces defined above.   
\begin{definition}\label{def:Kernel-not-reg-the}
  Let $\a\in   [0,1)$,  $\t\in[ 0,1]$, $m\in \N$ and $\k>0$. For $\d\in (0,\infty]$,  we define   $\ti \scrK_\t^s(\k,m+\a,Q_\d)$ by the set of kernels $K: \R^N\times \R^N\to [-\infty,+\infty] $ satisfying  \eqref{eq:Kernel-satisf} and 
$$
\begin{aligned}
%
%
%
& (iii) \|{\cA}_{K}  \|_{C^{m,\a}(Q_\d)\times L^\infty(S^{N-1} ) }+  \|\cA_{o,K}  \|_{\cC^{m}_{\t}(Q_\d) \times L^\infty(S^{N-1} ) }\leq \frac{1}{\k}  ,\\
&(iv) \cA_{o,K}(x,0,\th)=0\qquad\textrm{ for all $(x,\th)\in B_\d\times S^{N-1}$,}  
\end{aligned}
$$
where
\be\label{eq:def-ti-l-oK}
 {\cA}_{o,K}(x,r,\th) :=\frac{1}{2}\{ {\cA}_{K}(x,r,\th) - {\cA}_{K}(x,r,-\th) \}
\ee 
and   $\cA_K(\cdot,\cdot,\th)$ is  a continuous extension of $(x,r)\mapsto r^{N+2s}K(x,x+r\th)$ on $Q_\d $ for all $\th\in S^{N-1}$.
\end{definition}
The simple model case for the class of operators in Definition \ref{def:Kernel-not-reg-the}  is the anisotropic fractional Laplace operator, with kernel $K(x,y)=a((x-y)/|x-y|)|x-y|^{-N-2s}$ and $a\in L^\infty(S^{N-1})$ is even but not  necessarily continuous. In this case,  $\cA_K(x,r,\th)=a(\th)$, so that $K\in\ti   \scrK_\t^s(\k,m,Q_\d)$ for all $m\in \N$. As an example, a prototype energy functional can be an anisotropic integral energy functional,  generalizing   \eqref{eq:cJs}, given by
\be \label{eq:cJs-Gen}
\cJ_{s}(u):= (1-s) \int_{\R^{2N}\setminus (\R^N\setminus \O)^2} F\left( \frac{x-y}{|x-y|} ,\frac{u(x)-u(y)}{|x-y|} \right)  |x-y|^{-N-2s+2}\, dxdy ,
\ee
where    $F: S^{N-1}\times \R\to \R$  satisfies $\k\leq \de^2_z F(\th,z)\leq \frac{1}{\k}$.   The results in the present section provide smoothness of a critical point  $u\in H^s(\R^N)$ to $\cJ_s$ defined in  \eqref{eq:cJs-Gen}, provided $z\mapsto F(\cdot, z)$ is smooth. Indeed,   as above, the difference quotient $\frac{u(\cdot+h)-u(\cdot)}{|h|}$ solves an equation like \eqref{eq:Main-problem}, with $K\in \ti \scrK_{\a}^s(\k,m+\a,Q_\d)$, provided $u\in C ^{m+1,\a}(\O)$.
\\
We observe  that $\scrK^s(\k,m+\a,Q_\d) \subset  \ti \scrK_\t^s(\k,m+\a,Q_\d)$ for all $\t\leq \a$ and since $\cA_{o,K}(x,0,\th)=0$, we have that $\ti \scrK^s_0(\k,m+\a,Q_\d) = \ti \scrK_\a^s(\k,m+\a,Q_\d)$.  Moreover, we have the following interesting property on the set  $\ti \scrK_\t^s(\k,m+\a,Q_\infty)$ concerning scaling and translations. Indeed,   for $K\in \ti \scrK_\t^s(\k,m+\a,Q_\infty)$, $\rho\in (0,1)$ and $z\in \R^N$, letting  $K_{z,\rho}(x,y):=\rho^{N+2s}K(\rho x+z, \rho y+z) $,  we then have  that $\cA_{K_{z,\rho}}(x,r,\th)=\cA_{K}(\rho x+z,\rho r,\th)$ and thus $K_{z,\rho}\in \ti \scrK_\t^s(\k,m+\a,Q_\infty)$.   


The kernels in $\ti  \scrK_\t^s(\k,m+\a, Q_{\d})$ yield, in many cases,  similar regularity estimates as those in $ \scrK^s(\k,m+\a, Q_{\d})$, stated above,  provided some global regularity/behavior  of  $u$  is a priori known. \\
Our first main result in this section is the following. 
 \begin{theorem}\label{th:main-th1}
Let $s\in (0,1)$, $N\geq 1$, $\k>0$  and $\a\in (0,1)$.   Let $K\in  \ti \scrK_0^s(\k,\a, Q_2)$,  $u\in H^s(B_2)\cap {L_s(\R^N)}$ and $V, f\in L^p(B_2)$, for some $p>N/(2s)$, satisfy
 $$
 \cL_K u+ Vu = f\qquad \textrm{in $B_2$}.
 $$
 \begin{itemize}
 \item[$(i)$] If $2s\leq 1$, then there exists $C=C(s,N,\k, \a,p, \|V\|_{L^p(B_2)})>0$ such that  
$$
\|u\|_{C^{0,2s-\frac{N}{p}}(B_1)}\leq C(\|u\|_{L^2(B_2)}+\|u\|_{{L_s(\R^N)}}+ \|f\|_{L^p(B_2)}).
$$
 \item[$(ii)$] If $2s-1>\max(\frac{N}{p},\a)$, then   there exists $C=C(s,N,\k, \a,p, \|V\|_{L^p(B_2)})>0$ such that
$$
\|u\|_{C^{1,\min (2s-\frac{N}{p}-1,\a)}(B_1)}\leq C(\|u\|_{L^2(B_2)}+\|u\|_{{L_s(\R^N)}}+ \|f\|_{L^p(B_2)}).
$$
 \end{itemize}
 \end{theorem}
%
Our next result is concerned with $C^{m+2s+\a}$  Schauder estimates.   
  \begin{theorem}\label{th:Schauder-0}
Let $N\geq 1$ and  $s\in (0,1)$. Let $\k>0$,  $\a\in (0,1)$ and $m\in \N$. Let $K\in \ti  \scrK_{\b}^s(\k,m+ \a, Q_2)$,  with $\b=\min (\a+(2s-1)_+,1)$. Let   $u\in H^s(B_2)\cap L_s(\R^N)$ and $f\in C^{m,\a}(B_2)$ such that 
 $$
 \cL_K u= f\qquad \textrm{in $B_2$}.
 $$
 \begin{itemize}
 \item[$(i)$] If  $u\in   C^{m,\a}(\R^N)$ and  $2s+\a<1$, then
$$
\|u\|_{C^{m,2s+\a}(B_1)}\leq C(\|u\|_{C^{m,\a}(\R^N)}+ \|f\|_{C^{m,\a}(B_2)}).
$$
 \item[$(ii)$] If  $u\in   C^{m,\a}(\R^N)$, $2s\not=1$ and  $1<2s+\a<2$, then
$$
\|u\|_{C^{m+1, 2s+\a-1}(B_1)}\leq C(\|u\|_{C^{m,\a}(\R^N)}+ \|f\|_{C^{m,\a}(B_2)}).
$$
 \item[$(iii)$] If  $u\in   C^{m,\a}(\R^N)$,  $2<2s+\a$ and $K \in  \ti \scrK_0^{s}(\k,m+ 2s-1+ \a, Q_2)$, then
$$
\|u\|_{C^{m+2, 2s+\a-2}(B_1)}\leq C(\|u\|_{C^{m,\a}(\R^N)}+ \|f\|_{C^{m,\a}(B_2)}).
$$
\item[$(iv)$] If  $u\in   C^{m,\a}(\R^N)$,  $2s=1$ and $K \in  \ti \scrK_\t^{s}(\k,m+ \a, Q_2)$, for some $\t>\a$, then
$$
\|u\|_{C^{m+1,\a}(B_1)}\leq C(\|u\|_{C^{m,\a}(\R^N)}+ \|f\|_{C^{m,\a}(B_2)}).
$$
 \end{itemize}
If  moreover    $\|{\cA}_K\|_{C^{m,\a}(Q_2 \times S^{N-1})}\leq \frac{1}{\k}$,  then we can replace $ \|u\|_{C^{m,\a}(\R^N)}$ with $ \|u\|_{L^\infty(\R^N)} $.
Here $C=C(N,s,\k,\a, m,\t).$
 \end{theorem}
 We point out the remarkable     differences  between the last assertion in Theorem \ref{th:Schauder-0} and the results in  Theorem \ref{th:Schauder-0-intro}. Indeed, in  the former, $\cA_K$ is only required to be in $C^{m,\a} (Q_2\times S^{N-1})$, when   $2s+\a<2$, instead of $C^{m,\a+(2s-1)_+} (Q_2 \times S^{N-1})$ which was assumed in the latter.  Moreover, Theorem \ref{th:Schauder-0}-$(iv)$, for $s=1/2, $ provides the optimal estimate which covers the case   $\cL_K=\Ds_a$,  the anisotropic fractional Laplacian  i.e. when $K(x,y)=a((x-y)/|x-y|)|x-y|^{-N-2s}$,   while Theorem \ref{th:Schauder-0-intro} does not if $a$ is not smooth enough.   In fact the results in Theorem \ref{th:Schauder-0}   were known for  the anisotropic fractional Laplacian when $a$ is a measure on the unit sphere $S^{N-1}$, see Ros-Oton and Serra \cite{RS2}  and when $a\in C^\infty(S^{N-1})$, see Grubb \cite{ Grubb2}. \\
 Interior regularity and  Harnack inequality for linear and fully nonlinear nonlocal equations have been intensively  investigated in  last decades by many authors, see e.g. \cite{FK,BL,Kassm,Barrios,CS1,CS2,CS3,BC,KM1,Jin-Xiong,Kriventsov,SS,Serra-OK, {Silv-1},DK,Ab-L} and the references therein. \\
 
 Next, we observe  that  Theorem \ref{th:main-th10} and \ref{th:Schauder-0-intro} are immediate consequences of Theorem \ref{th:main-th1} and \ref{th:Schauder-0}, respectively. The proof of  Theorem \ref{th:main-th1} and \ref{th:Schauder-0} uses a blow up analysis and compactness method for weak and classical solutions, partly inspired by \cite{Serra-OK} and \cite{Fall-Reg}, see also \cite{Serra,RS2,RS4,Ros-Real,Ros-Real-2} for  translation invariant problems. Indeed, we use a fine scaling argument to  balance, in an optimal manner,  
the norm of the right hand side and the homogeneity of the equation. The scaling parameter is chosen so that the limit of the  rescaled solution, after subtracting a polynomial,   satisfies an equation for which all solutions with such  growth are explicitly known,  thanks to a Liouville type theorem. To obtain H\"older, gradient and second order derivative estimates,  the subtracted polynomial are, respectively given by the projection, with respect to the $L^2(B_r)$ scalar product, of the weak solution $u$ on constant functions, affine functions and second order polynomials. More precisely,  our primary goal is to    show the  growth estimates (or Taylor expansion in $L^2$-sense)  
$$
\|u-P_r\|_{L^2(B_r)}\leq C r^{\frac{N}{2}+\textrm{deg}(P_r)+\g},
$$
where $P_r$ is a suitable  polynomial and the parameter $\g\in (0,1)$ is determined by the regularity of the entries $V,f$ and $\cA_K$ the coefficient of the operator. The above expansion  leads to $u\in C^{m,\g }$, with $m=\textrm{deg}(P_r)$.\\
To carry over the blow up argument and to use compact Sobolev embedding or the  Arzel\'{a}-Ascoli theorem, after subtracting polynomials,  rescaling and normalization,  it  is  necessary to  derive a priori H\"older  estimate  for functions $v$ solving  the  more general equation 
\be \label{eq:intro-100}
\cL_Kv +\cL_{K'} U=F \qquad\textrm{ in $\O$.}
\ee 
Actually, in the counter part of \eqref{eq:intro-100} in the local case reads as 
\be \label{eq:intro-10000}
-\textrm{div}(A(x)\n v)+ \textrm{div} U=F,
\ee
for some potential $U$. The study of \eqref{eq:intro-100} is typically essential for the proof of  Theorem \ref{th:main-th1}-$(ii)$, where   $U=p_r$  is a first order polynomial and $\cL_{K'}$ is a non-translation invariant operator.  Recall here that obtaining gradient estimates for solutions to equations involving  divergence operators is more subtle than those involving operators in non-divergence form, since the  latter annihilate affine functions, while the  former  do not and so the study of \eqref{eq:intro-10000} becomes useful. 
 The same difficulty is of course faced here since we are dealing with  non-translation invariant variational solutions. \\
The core of the paper, from which we derive all the results,  is Proposition \ref{prop:bound-Kato-abstract} below, where we   prove H\"older continuity of the solutions to   \eqref{eq:intro-100}, under mild regularity assumptions on  $K,K', U$ and $F$, and we believe that the argument of proof and the result itself could be of independent interest. \\
We finally remark that the Schauder estimates in the present paper remains stable as $s\to 1$, provided we replace the kernel $K$, with $(1-s)K$ and $\a$ is such that $2<2s_0+\a$, for some $s_0\in (0,1)$.\\

 The paper is organized as follows. In Section  \ref{s:NotPrem}, we collect some preliminary result and notations. Section \ref{s:AprioriEstim} contains the   regularity estimates for solutions to \eqref{eq:intro-100}.   Now Theorem \ref{th:main-th1}-$(ii)$ is proved in Section \ref{s:GradEstim} and  Theorem \ref{th:Schauder-0} in Section \ref{s:Shaud}. Finally the proof of the main results are gathered in Section \ref{s:proofMainResults}.  
  
 \section{Notations and  preliminary results}\label{s:NotPrem}
\subsection{Notations}
In this paper, the ball centred at  $z\in\R^N$ with radius $r>0$ is denoted by $B(z,r)$ and $B_r:=B_r(0)$. 
Here and in the following, we let $\vp_1 \in C^\infty_c(B_2)$ such that  $\vp_1 \equiv 1$ on $B_{1}$ and $0\leq \vp_1\leq  1$ on $\R^N$.  We put  $\vp_R(x):=\vp(x/R)$.  For $b\in L^\infty(S^{N-1})$,   we define   $
\mu_b(x,y)= |x-y|^{-N-2s} b\left(\frac{x-y}{|x-y|} \right).$\\
Given $\s>0$, we define the space 
$$
L_\s(\R^N):=\left\{u\in L^1_{loc}(\R^N)\,:\, \|u\|_{L_\s(\R^N)}:= \int_{\R^N}{|u(x)|} (1+|x|^{N+2\s})^{-1}\,dx  <\infty  \right\}.
$$
Throughout this paper, for the   seminorm of  the fractional Sobolev spaces, we adopt the notation
$$
[u]_{H^s(\O)}:=\left(\int_{\O\times\O} {|u(x)-u(y)|^2}\mu_1(x,y)\, dxdy\right)^{1/2}.
$$
We will, sometimes use the notation 
$$
[u]_{H^s_{K}(\O)}:=\left(\int_{\O\times\O} {|u(x)-u(y)|^2}|K(x,y)|\, dxdy\right)^{1/2},
$$
for a function $K:\O\times \O\to [-\infty,+\infty]$. For  the   H\"older and Lipschitz seminorm, we write
$$
[u]_{C^{0,\a}(\O)}:=\sup_{x\not=y\in\O} \frac{|u(x)-u(y)|}{|x-y|^\a},
$$ 
for $\a\in (0,1]$. If there is no ambiguity, when   $\a\in (0,1)$, we will write $[u]_{C^{\a}(\O)}$ instead of $[u]_{C^{0,\a}(\O)}$. 
If $m\in \N$ and $\a\in (0,1)$, the H\"older space $\|u\|_{C^{m,\a}(\O)}$  is given by the set of functions in $C^m(\O)$ such that 
$$
\|u\|_{C^{m+\a}(\O)}:=\|u\|_{C^{m,\a}(\O)}=  \sup_{\g\in \N^N, |\g|\leq  m}   \| \de^\g  u\|_{ L^{\infty}(\O)}+ \sup_{\g\in \N^N, |\g|= m}  \|\de^\g u\|_{ C^{\a}(\O)}<\infty  .
$$

Letting $u\in L^1_{loc}(\R^N)$, the mean value of $u$ in $ B_r(z)$ is denoted by 
$$
u_{B_r(z)}=(u)_{B_r(z)}:=\frac{1}{|B_r|}\int_{B_r(z)}u(x)\, dx. 
$$
%
%
 For $\a\in [0,1]$, $h\in \R^N\setminus \{0\}$ and $f\in C^{0,\a}_{loc}(\R^N)$, we define
\be \label{eq:def-f-h-alph}
f_{h,\a}(x):=\frac{f(x+h)-f(x)}{|h|^\a}.
\ee

\subsection{Preliminary results}
We gather in this paragraph some results which we will frequently use in the following of the paper.
%
%
%
Let  
 $K: \R^N\times \R^N\to [0,\infty]$  satisfy the following properties:
 \be \label{eq:K-Kernel-satisf}
 \begin{aligned}
(i)\,& K(x,y)=K(y,x) \qquad\textrm{ for all $x,y\in\R^N$, }\\
(ii)\,& \k  \mu_1(x,y)\leq K(x,y)\leq \frac{1}{\k} \mu_1(x,y) \qquad\textrm{ for all $x, y\in\R^N$.}
 %
 \end{aligned}
 \ee
For $\a'\geq 0$, we let $K':\R^N\times \R^N\to [-\infty,+\infty]$  satisfy
\be \label{eq:K'-Kernel-satisf} 
 \begin{aligned}
(i)\,& K'(x,y)=K'(y,x) \qquad\textrm{ for all $x, y\in\R^N$, }\\
 %
 %
(ii)\,& |K'(x,y)|\leq \frac{1}{\k} (|x|+|y|+1)^{\a'} \mu_1(x,y) \qquad\textrm{for all $x, y\in \R^N$}.
 \end{aligned}
 \ee
%
%
 %
 %
 %
 %
%
  Let     $U\in H^s_{loc}(\O)\cap L_{(\a'+2s)/2}(\R^N)$ and  $f\in L^1_{loc}(\R^N)$. We say that $u\in H^s_{loc}(\O)\cap {L_s(\R^N)}$ is a (weak) solution to
\be\label{eq:cL_K-eq-V-f-K'}
\cL_{K}  u+ \cL_{K'}U=  f \qquad\textrm{ in $  \O$,}
\ee
if,     for every $\psi\in C^\infty_c(\O)$, 
\begin{align*}
\int_{\R^{2N}}(u(x)-u(y))(\psi(x)-\psi(y))K(x,y)\,dxdy&+ \int_{\R^{2N}}(U(x)-U(y))(\psi(x)-\psi(y))K'(x,y)\,dxdy\\
&=  \int_{\R^N} f(x)\psi(x)\, dx.
\end{align*}
We note that each of the  terms in the above identity is finite. 
For $\b\in [0,2s)$,  we define the Morrey space $\cM_\b$ by    the set of functions $f \in L^1_{loc} (\R^N)\ $ such that 
$$
 \|f\|_{\cM_\b}:= \sup_{\stackrel{x\in \R^N}{r\in (0,1)}} r^{\b-N}\int_{  B_r(x)} |f(y)| \, dy<\infty,
$$
with  $\cM_0:=L^\infty(\R^N)$, and  we note that $ \|f\|_{\cM_{N/p}}\leq C(N,p) \|f\|_{L^p(\R^N)}$. We have the following  coercivity property, see  \cite{Fall-Reg}, 
\be\label{eq:coerciv}
\||f|^{1/2}v\|_{L^2(\R^N)}^2\leq C(N,s,\b) \|f\|_{\cM_\b}\|v\|_{H^s(\R^N)}^2 \qquad\textrm{for all $v\in H^s(\R^N)$.}
\ee
We prove our a priori estimates for   right hand in $\cM_\b$. Recall that $\cM_{N/p}$ contains strictly $\|f\|_{L^p(\R^N)} $. \\
  The following energy estimate can be seen as a  nonlocal Caccioppoli  inequality.
\begin{lemma} \label{lem:from-caciopp-ok}
Let $N\geq 1$,  $s\in (0,1)$ and $\k>0$. We consider  $K$ satisfying  \eqref{eq:K-Kernel-satisf} and $K'$ satisfying \eqref{eq:K'-Kernel-satisf}, for some $\a'\geq 0$.
Let $v\in H^s (\R^N)  $  and $U\in H^s_{loc}(\R^N)\cap L_{(\a'+2s)/2}(\R^N)$ and  $f\in\cM_\b$ satisfy
\be\label{eq:Dsv-eq-V-f}
\cL_{K}  v+ \cL_{K'}U =  f \qquad\textrm{ in $  B_{2R}$.}
\ee
Then    for every $\e>0$,  there exist $\ov C=\ov C(s,N,\k,R)$ and  $C=C(\e, s,N, \k,R)$ such that   
\begin{align*}
\left \{\k-  \e \ov C   \|f\|_{\cM_\b}    \right\} &\int_{\R^N\times \R^N}(v(x)-v(y))^2 \vp_R^2(y)  \mu_1(x,y)\,dxdy\\
&\leq   C( \|f\|_{ \cM_\b}  +1)  \|v \|_{L^2(\R^N)}^2 + C \|f\|_{\cM_\b}   \|\vp_R\|^2_{H^s(\R^N)} \nonumber\\
%
%
 &\quad + C[U]_{H^s_{K'}(B_{4R})} ^2 +C\int_{\R^N}  \vp_R^2(y)|v(y)|   \left( \int_{ \R^N\setminus B_{4R}}|U(x)-U(y)||K'(x,y)|\,dx\right)dy    . 
\end{align*}
\end{lemma}
\begin{proof}
Applying  \cite[Lemma 9.1]{Fall-Reg},  we  get
\begin{align}\label{eq:from-cacciop-to-coerciv}
 (\k-\e)\int_{\R^{2N}}(v(x)-v(y))^2\vp^2_R(y)& \mu_1(x,y)\,dxdy\leq        \int_{\R^N} |f(x)| |v(x)| \vp_R^2(x)\, dx    \nonumber\\
&+ C  \int_{\R^{2N}}(\vp_R(x)-\vp_R(y))^2v^2(y) \mu_1(x,y)\,dxdy \nonumber\\
& + \int_{\R^{2N}}|U(x)-U(y)||\vp_R^2(x) v(x)-\vp_R^2(y)v(y)||K'(x,y)|\,dxdy . 
\end{align}
We now estimate
\begin{align} \label{eq:from-cacciop-to-coerciv1}
 \int_{\R^N\times \R^N}|U(x)-U(y)||\vp_R^2(x)& v(x)-\vp_R^2(y)v(y)||K'(x,y)|\,dxdy   \nonumber\\
& = \int_{B_{4R}\times B_{4R}}|U(x)-U(y)||\vp_R^2(x) v(x)-\vp_R^2(y)v(y)||K'(x,y)|\,dxdy  \nonumber\\
 &+ 2\int_{\R^N}  \vp_R^2(y)|v(y)|   \left( \int_{ \R^N\setminus B_{4R}}|U(x)-U(y)||K'(x,y)|\,dx\right)dy \nonumber\\
 &\leq \e/\k (2(4R)+1)^{\a'}  [\vp_R^2 v]_{H^s(B_{4R})}^2 + C[U]_{H^s_{K'}(B_{4R})} ^2  \nonumber\\
 & +C\int_{\R^N}  \vp_R^2(y)|v(y)|   \left( \int_{ \R^N\setminus B_{4R}}|U(x)-U(y)||K'(x,y)|\,dx\right)dy.
\end{align}
We recall that 
\be \label{eq:Op-vp2}
\int_{\R^N}(\vp_1(x)-\vp_1(y))^2\mu_1(x,y)\,dy\leq C(N,s)(1+|x|^{-N-2s}) \qquad \textrm{ for every $x\in \R^N$.}
\ee
Therefore
\begin{align*}
[\vp_R^2 v]_{H^s(B_{4R})}^2&\leq 2\int_{\R^{2N}}(v(x)-v(y))^2\vp^4_R(y) \mu_1(x,y)\,dxdy+ 2 \int_{\R^{2N}}(\vp_R^2(x)-\vp_R^2(y))^2v^2(y) \mu_1(x,y)\,dxdy\\
&\leq 2\int_{\R^{2N}}(v(x)-v(y))^2\vp^2_R(y) \mu_1(x,y)\,dxdy+    \|v \|^2_{L^2(\R^N)}  .
\end{align*}
 Using this in \eqref{eq:from-cacciop-to-coerciv1}, we get
\begin{align} \label{eq:from-cacciop-to-coerciv00}
& \int_{\R^N\times \R^N}|U(x)-U(y)||\vp_R^2(x) v(x)-\vp_R^2(y)v(y)||K'(x,y)|\,dxdy   \nonumber\\
& \leq \e \ov C \int_{\R^{2N}}(v(x)-v(y))^2\vp^2_R(y) \mu_1(x,y)\,dxdy+   C \|v \|^2_{L^2(\R^N)}   \nonumber\\
&+ C[U]_{H^s_{K'}(B_{4R})} ^2 +C\int_{\R^N}  \vp_R^2(y)|v(y)|   \left( \int_{ \R^N\setminus B_{4R}}|U(x)-U(y)||K'(x,y)|\,dx\right)dy.
\end{align} 
Next, from \eqref{eq:coerciv},  Young's inequality and \eqref{eq:Op-vp2}, we deduce that    
\begin{align*}
  \int_{\R^N} |  f(x)| |\vp_R(x)|^2|v(x)| \,dx  
& \leq     \e  \ov C   \|f\|_{\cM_\b}     \int_{\R^{2N}}(v(x)-v(y))^2\vp^2_R(y) \mu_1(x,y)\,dxdy\\
&     +   C   \|f\|_{\cM_\b}     \|v \|^2_{L^2(\R^N)}  + C    \|f\|_{\cM_\b}   \|\vp_R\|^2_{H^s(\R^N)}.
\end{align*}
Using this and \eqref{eq:from-cacciop-to-coerciv00}  in \eqref{eq:from-cacciop-to-coerciv}, we get the result.
\end{proof}
We state the  following result.
 \begin{lemma} \label{lem:convergence-very-weak}
Let $N\geq 1$,  $s\in (0,1)$ and $\k>0$. We consider  $K$ satisfying  \eqref{eq:K-Kernel-satisf} and $K'$ satisfying \eqref{eq:K'-Kernel-satisf}, for some $\a'\geq 0$.
Let $v\in H^s (\R^N)  $  and $U\in H^s_{loc}(\R^N)\cap L_{(\a'+2s)/2}(\R^N)$ and  $f\in\cM_\b$ satisfy
$$
\cL_{K}  v+ \cL_{K'}U =  f \qquad\textrm{ in $  B_{2R}$.}
$$
Then there exists $C=C(N,s,\k,\a',R)$ such that  for every $\psi\in C^\infty_c(B_{R}  )$, we have 
\begin{align} 
& \left|\int_{\R^{2N}}(v(x)-v(y))(\psi(x)-\psi(y))K(x,y)\, dxdy \right| \leq    C   \|f\|_{\cM_\b}   \left(  1 + \|\psi\|_{H^s(\R^N)}^2     \right) \label{eq:express1}\\
&+ C [U]_{H^s_{K'}(B_{4R})} [\psi]_{H^s(B_{4R})} +   C\int_{\R^N} | \psi(y)|   \left( \int_{ \R^N\setminus B_{4R}}|U(x)-U(y)||K'(x,y)|\,dx\right)dy. \label{eq:express2}
\end{align}
%
\end{lemma}
\begin{proof}
Using the weak formulation of the equation and  \eqref{eq:coerciv}, we get the  expression on the left hand side in \eqref{eq:express1}. Now expression \eqref{eq:express2} appears after decomposing the domain of integration and using H\"older's inequality as in the beginning of the proof of Lemma \ref{lem:from-caciopp-ok}.
\end{proof} 

We close this section with the following  result.
\begin{lemma}\label{lem:estim-G_Ku}
Let   $K$ satisfy \eqref{eq:Kernel-satisf}$(i)$-$(ii)$. 
Let $v\in H^s_{loc}(B_{2R})\cap {L_s(\R^N)} $  and $f\in L^1_{loc}(\R^N)$ satisfy
$$
\cL_{K}  v= f \qquad\textrm{ in $  B_{2R} $,} 
$$
for  some $R>0$. We let $v_R:=\vp_R v$. Then 
\be\label{eq:v-cat-vpR} 
\cL_{K} v_R  =  f+     G_{K,v,R} \qquad\textrm{in $B_{R/2} $},
\ee
where
\be \label{eq:def-GKvR}
G_{K,v,R}(x)=  \int_{\R^N} v(y)  (\vp_R(x)-\vp_R(y)) K(x,y) \, dy.
\ee
Moreover,  $G_{K,v,R}$ satisfies the following properties.
\begin{itemize}
\item[$(i)$]   There exists   $C=C(N,s,  R)$ such that 
$$
\|G_{K,v,R}\|_{L^\infty(B_{R/2})}\leq C\sup_{x\in B_{R/2}}\int_{|y|\geq R}|v(y)| |K(x,y)|\, dy.
$$
\item[$(ii)$] If $v\in C^{k+\a}(\R^N)$ and $ \|{\cA}_K \|_{C^{k,\a}(Q_{\infty} )\times L^\infty(S^{N-1})}\leq c_0$, then there exists  $C=C(N,s,\a,c_0, R,k)$ such that 
$$
\|G_{K,v,R}\|_{C^{k+\a}(B_{R/2})}\leq C \|v\|_{C^{k+\a}(\R^N)}.
$$
\item[$(iii)$] If   $\|{\cA}_K\|_{C^{k+\a}( Q_\infty\times S^{N-1})}\leq c_0$, then   there exists  $C=C(N,s,\a,c_0, R,k)$ such that
$$
\|G_{K,v,R}\|_{C^{k+\a}(B_{R/4})}\leq C \|v\|_{{L_s(\R^N)}}  .
$$
\end{itemize}
\end{lemma}
\begin{proof}
For \eqref{eq:v-cat-vpR}, see \cite[Lemma 9.2]{Fall-Reg}. Statement $(i)$ follows  easily, thanks to the definition of $\vp_R$.
 To prove $(ii)$, we write
\begin{align*}
G_{K,v,R}(x)=  \int_{S^{N-1}}\int_0^\infty v(x+r\th)(1-\vp_R(x+r\th))  {\cA}_{K}(x,r,\th) r^{-1-2s}\, dr d\th.
\end{align*}
Since $1-\vp_R(x+ r\th )=0 $ for all  $x\in B_{R/2}$, $r\in (0,R/2)$ and $\th\in S^{N-1}$, then $(ii)$ follows.\\
To prove $(iii)$, we note that
%
\begin{align*}
G_{K,v,R}(x)&:=   \int_{\R^N}  v(y)  (\vp_R(x)-\vp_R(y)) K(x,y) \, dy\\
&=  \int_{|y|\geq R} v(y)  (1-\vp_R(y)) {\cA}_{K}(x,|x-y|,(x-y)/|x-y|)|x-y|^{-N-2s} \, dy.
\end{align*}
We recall that (see e.g. \cite{FW})
for every $x_1,x_2,y\in\R^N$, $\varrho>0$ and $\a\in (0,1)$, 
$$
||x_1-y|^{-\varrho}- |x_2-y|^{-\varrho} |\leq C(\a,\varrho)|x_1-x_2|^{\a}\{
|x_1-y|^{-(\varrho+\a)}+ |x_2-y|^{-(\varrho+\a)}  \}.
$$
Hence for all $x_1,x_2\in B_{R/2}$, $y\in \R^N\setminus B_R$, $\varrho\geq N+2s$ and $\a\in (0,1)$,  we get 
$$
||x_1-y|^{-\varrho}- |x_2-y|^{-\varrho} |\leq C(\a,\varrho,R)|x_1-x_2|^{\a} |y|^{-N-2s}.
$$
Using this and   the Leibniz formula for higher order derivatives of the product of functions,  we get  $(iii)$.
\end{proof}

\section{A priori    estimates} \label{s:AprioriEstim}

%
In this section, we prove a priori estimates for solutions to \eqref{eq:cL_K-eq-V-f-K'},  provided $\cL_K$ is close to the translation invariant operator  $\cL_{\mu_a}$,  
with $a: S^{N-1}\to \R$ satisfies
\be\label{eq:def-a-anisotropi}
a(-\theta)=a(\theta) \qquad\textrm{ and } \qquad\k \leq a(\theta)\leq   \frac{1}{\k}\ \qquad\textrm{ for all $\theta\in S^{N-1}$}.
\ee 
We now recall two results from \cite{Fall-Reg} that will be needed in the following of the paper.
\begin{lemma} 
\label{lem:Ds-a-n--to-La}
Let $b\in L^\infty(S^{N-1})$. Suppose that there exists a sequence of functions $(a_n)_n$ satisfying \eqref{eq:def-a-anisotropi} and such that  $a_n\stackrel{*}{\rightharpoonup} b$ in $L^\infty(S^{N-1})$.    Let $\l_n: \R^N\times \R^N\to [0,\k^{-1}] $, with $\l_n\to 0 $ pointwise on $\R^N\times \R^N$. Let  $(K_n)_n$ be a sequence of  symmetric kernels satisfying 
$$
|K_n(x,y)-\mu_{a_n}(x,y)|\leq \l_n(x,y)\mu_1(x,y)  \qquad \textrm{  for all  $x\not=y\in \R^N $ and for all $n\in \N$. }
$$
If   $(v_n)_n$ is  a bounded  sequence in ${L_s(\R^N)}\cap H^s_{loc}(\R^N)$ such that $v_n\to v$ in ${L_s(\R^N)}$,   then 
$$
\int_{\R^{N}}v(x)\cL_{\mu_b}\psi(x)\,dx=\frac{1}{2} \lim_{n\to\infty}\int_{\R^{2N}}(v_n(x )-v_n(y))(\psi(x)-\psi(y)) {K_n(x,y)}\,dxdy    \quad\textrm{ for all  $\psi\in C^\infty_c(\R^N)$.}
$$
\end{lemma}
\begin{lemma} \label{lem:Liouville}
Let $b\in L^\infty(S^{N-1})$. Suppose that there exists a sequence of functions $(a_n)_n$ satisfying \eqref{eq:def-a-anisotropi} and such that  $a_n\stackrel{*}{\rightharpoonup} b$ in $L^\infty(S^{N-1})$. Consider $u\in H^s_{loc}(\R^N)$  satisfying 
\begin{align*}
\begin{cases}
\cL_{\mu_b} u=0  \qquad \textrm{ in $\R^N$},  \vspace{3mm}\\
\|u\|_{L^2(B_R)}^2 \leq  R^{N+2\g} \qquad \textrm{ for  some $\g<2s$ and for every $R\geq 1$}.
\end{cases}
\end{align*}
Then $u$ is an affine function.
  \end{lemma}

  \subsection{A priori    estimates and consequences}
We now state the main result of  the present section.
 \begin{proposition} \label{prop:bound-Kato-abstract}
Let  $s\in (0,1)$, $\b\in [0,2s)$,  $\s\in(s,1]$ and $\k>0$. Let $\a'\geq 0$, with  $\a'+\s\in (0,2s)$.  Pick $$\g\in (0,1)\cap(0,\s]\cap (0,2s-\b].$$
 Then there exist $\e_0>0$  and $C  >0$ such that if
\begin{itemize}
\item $a$ satisfies \eqref{eq:def-a-anisotropi},  $K_a$ satisfies \eqref{eq:K-Kernel-satisf} with
$$
|K_a-\mu_a|<\e_0\mu_1(x,y) \qquad\textrm{ on $B_2\times B_2\setminus \{x=y\}$}, 
$$
\item $K'$ satisfies \eqref{eq:K'-Kernel-satisf},   

\item $ f\in \cM_\b$,  $ g\in H^s(\R^N)$, $U\in C^{0,\s}_{loc}(\R^N)\cap L_{(\a'+2s)/2}(\R^N)$  are such that  
$$
 \cL_{K_a}  g+\cL_{K'} U=f \qquad\textrm{ in $B_2$,} 
$$
\end{itemize}  
 then     
$$
\sup_{r>0} r^{-2\g-N}  \|g-g_{B_r}\|_{L^2(B_r)}^2 \leq  C    (  \|g\|_{  L^2(\R^N) }+[ U]_{C^{0,\s}(\R^N) } +  \|f\|_{ \cM_\b } )^2.
$$
\end{proposition}
 \begin{proof}
Assume that  the assertion in the proposition does not hold, then   for every   $n\in \N$,  there exist:
\begin{itemize}
\item $a_n$ and $K_{a_n}$ satisfying  \eqref{eq:def-a-anisotropi} and  \eqref{eq:K-Kernel-satisf} respectively, with 
\be \label{eq:K-a-n-mu-an}
|K_{a_n}-\mu_{a_n}|<\frac{1}{n}  \mu_1(x,y) \qquad\textrm{ on $B_2\times B_2 \setminus \{x=y\}$}, 
\ee
\item $K'_n$ satisfying \eqref{eq:K'-Kernel-satisf},  $ f_n\in \cM_\b$, $U_n\in C^{0,\s}_{loc}(\R^N)\cap L_{(\a'+2s)/2}(\R^N)$ and  $g_n \in H^s(\R^N)$,      with   $\|g_n\|_{L^2(\R^N)}+\|f_n\|_{\cM_\b}+ \| U_n\|_{C^{\s}(\R^N) } \leq 1 $,
 \be
 \cL_{K_{a_n}}  g_n+\cL_{K'_ n} U_n =f_n \qquad\textrm{ in $B_2$}, 
\ee
\end{itemize}  
with the property  that    
$$
\sup_{r>0} r^{-N-2\g }  \|g_n-(g_n)_{B_{ r}}\|_{L^2(B_{r})}^2 > n. 
$$
Consequently,  there exists $\ov r_n>0$ such that 
\be\label{eq:ov-r-n-un-larger-n-eps}
 \ov r^{-N-2\g }_n   \|g_n-(g_n)_{B_{\ov r_n}}\|_{L^2(B_{\ov r_n})}^2 > n/2  .
\ee
We consider the (well defined, because $\|g_n\|_{ L^2(\R^N)} \leq 1$) nonincreasing function $\Theta_n: (0,\infty)\to [0,\infty)$ given by
$$
\Theta_n(\ov r)=\sup_{ r\geq \ov r }    r^{-N-2\g}   \|g_n-(g_n)_{B_{ r}}\|_{L^2(B_{ r})}^2  .
$$
Obviously, for $n\geq 2$, by \eqref{eq:ov-r-n-un-larger-n-eps},   
\be\label{eq:The-n-geq-n}
  \Theta_n(\ov r_n)> n/2\geq 1 .
\ee
Hence, provided $n\geq 2$,  there exists $r_n\in [\ov r_n,\infty)$ such that 
\begin{align*}
\Theta_n( r_n)&\geq   r^{-N-2\g }_n   \|g_n - (g_n)_{ B_{ r_n}   } \|_{L^2(B_{ r_n})}^2  \geq \Theta_n(\ov r_n)-1/2\geq (1-1/2)\Theta_n(\ov r_n)\geq \frac{1}{2}\Theta_n( r_n),
\end{align*}
where we used the monotonicity of $\Theta_n$ for the last inequality, while  the first inequality  comes from the definition of $\Theta_n$.
In particular, thanks to  \eqref{eq:The-n-geq-n},  $\Theta_n( r_n)\geq  n/4$. Now  since $ \|g_n \|_{L^2(\R^N)} \leq 1$,  we have that $  r^{-N-2\g }_n  \geq n/8$, so that $r_n\to 0$ as $n\to \infty$.  
 We now   define the   sequence  of functions
$$
 {w}_n(x)=  \Theta_n(r_n)^{-1/2} r_n^{-\g}   \left\{g_n(r_n x)   - \frac{1}{|B_1|} \int_{B_1} g_n(r_n x) \, dx  \right\}, 
$$
which,    satisfies
\be \label{eq:w-n-nonzero}
 \|w_n\|_{L^2(B_{1})}^2 \geq \frac{1}{2}, \qquad\int_{B_1}w_n(x)\,dx =0 \qquad\textrm{ for every $n\geq2$.}
\ee
%
 %
 Using that, for every $r>0$,  $    \|g_n-(g_n)_{B_{ r}}\|_{L^2(B_{ r})}^2\leq r^{N+2\g}\Theta_n(r)$  and the monotonicity of $\Theta_n$, by   \cite[Lemma 3.1]{Fall-Reg}, we find that 
\be\label{eq:groht-w-n-abs}
 \|w_n\|_{L^2(B_{R})}^2 \leq  C   R^{N+2\g}  \qquad\textrm{ for every $R\geq 1$ and $n \geq 2$,}
\ee
for some constant $C=C(N,\g)>0$.\\
We define
$$
\ov K_n(x,y)=r_n^{N+2s}K_{a_n}(r_nx ,r_ny ), \qquad \ov K_n'(x,y)=r_n^{N+2s}K_{n}'(r_nx ,r_ny )
$$
 and
$$
\ov U_n(x)=   U_n (r_n x), \qquad  \ov f_n(x)= r_n^{2s} f_n (r_n x).
$$
It is plain that 
\be  \label{eq:w_n-solves}
\cL_{\ov K_{n}} w_n   +r_n ^{-\g}  \Theta_n(r_n)^{-1/2} \cL_{\ov K_n'}\ov U_n= r^{-\g}_n \Theta_n(r_n)^{-1/2}  \ov f_n  \qquad\textrm{ in $B_{2/{r_n}} $.}
\ee
We fix $M>1$ and let $n\geq 2$ large, so that $1<M<\frac{1}{8r_n}$.
Therefore, letting $w_{n,M}:=\vp_{4M} w_n\in H^s(\R^N)$, we   apply Lemma \ref{lem:estim-G_Ku}$(i)$ to   get
\be \label{eq:eq-for-w-n-cut-Cacciopp}
\cL_{\ov K_{n}}w_{n,M} +  r_n ^{-\g}  \Theta_n(r_n)^{-1/2} \cL_{\ov K_n'}\ov U_n= r^{-\g}  \Theta_n(r_n)^{-1/2} \ov f_n+G_{K_n,w_n,M}   \qquad\textrm{ in $B_{2M}$,}
\ee
with $\|G_{\ov K_n,w_n,M}\|_{L^\infty(B_{M/2})}\leq C \|w_n\|_{{L_s(\R^N)}}\leq C$,  by \eqref{eq:groht-w-n-abs}.    We also note that 
\be \label{eq:estim-3}
\|\ov f_n\|_{\cM_\b}  \leq   r_n^{2s-\b}.
\ee
  Clearly $\ov K_n $ satisfies  \eqref{eq:Kernel-satisf}.
   Applying  Lemma \ref{lem:from-caciopp-ok} to the equation \eqref{eq:eq-for-w-n-cut-Cacciopp} and using \eqref{eq:groht-w-n-abs} together with \eqref{eq:estim-3}, we find a constant $\ov C $ such that    for every $\e>0$,  there exists    $C$ satisfying    
\begin{align}
&\left \{\k-  \e \ov C  \Theta_n(r_n)^{-1/2}  r_n^{2s-\b-\g}     \right\} \int_{B_{M/8}\times B_{M/8}}(w_{n,M}(x)-w_{n,M}(y))^2   \mu_1(x,y)\,dxdy \nonumber\\
&\leq   C( \Theta_n(r_n)^{-1/2} r_n^{2s-\b-\g} + 1)   +   C  r_n ^{-\g}  \Theta_n(r_n)^{-1/2} [\ov U_n]_{H^s_{\ov K'_n}(B_{4M})} ^2  \nonumber\\
 &\quad  +C  r_n ^{-\g}  \Theta_n(r_n)^{-1/2} \int_{\R^N}  \vp_M^2(y)|w_n(y)|   \left( \int_{ \R^N\setminus B_{4M}}|\ov U_n(x)-\ov U_n(y)||\ov K'_n(x,y)|\,dx\right)dy    .  \label{eq:estim-Ccciopp--111}
\end{align}   
%
We observe that
$$
|\ov K_n'(x,y)|\leq \frac{1}{\k} (r_n |x|+r_n|y|+1)^{\a'}   \mu_1(x,y).
$$
From this and the fact that $[\ov U_n]_{C^{\s}(\R^N)}\leq r_n^\s$ , we  have the following estimate:
\begin{align}\label{eq:estim-1}
 [\ov U_n]_{H^s_{\ov K'_n}(B_{4M})} ^2& =\int_{B_{4M}\times B_{4M}}(\ov U_n(x)-\ov U_n(y))^2\ov K_n'(x,y)\, dxdy \nonumber\\
 &\leq C(M)   \int_{B_{4M}\times B_{4M}}(\ov U_n(x)-\ov U_n(y))^2 \mu_1(x,y)\, dxdy \nonumber\\
 &  \leq C(M) r_n^{2\s}   \int_{B_{4M}\times B_{4M}}|x-y|^{-N-2s+2\s}\, dxdy \nonumber\\
&\leq C(M) r_n^{2\s},
\end{align}
because $\s>s$. 
In addition, since $\a'+\s<2s$, we get
\begin{align} \label{eq:estim-2}
\sup_{y\in B_M} \int_{ \R^N\setminus B_{4M}}|\ov U_n(x)-\ov U_n(y)| |\ov K_n'(x,y)|\,dx  &\leq   C(M)  r_n^{\s}     \int_{|x|\geq 2M}(1+|x|^{\a'}) |x|^{-N-2s+\s} dx \nonumber\\
& \leq C(M)   r_n^{\s} .
\end{align}
Now using \eqref{eq:estim-1} and  \eqref{eq:estim-2} in   \eqref{eq:estim-Ccciopp--111} and the fact that $\g\leq \min (2s-\b,\s)$, we find that
\begin{align*}
&\left \{\k-  \e \ov C  \Theta_n(r_n)^{-1/2}       \right\} [w_n]_{H^s(B_{M/8})}^2 \leq   C( \Theta_n(r_n)^{-1/2}   + 1)     .
\end{align*}   
%
Therefore, since $ \Theta_n(r_n)^{-1}\leq 1$, then   provided $\e$ is small  enough, by  \eqref{eq:groht-w-n-abs},  we deduce that $w_n$ is bounded in $H^s_{loc}(\R^N)$. Hence by Sobolev embedding and  \eqref{eq:groht-w-n-abs}, there exists $w\in H^s_{loc}(\R^N)\cap {L_s(\R^N)}$ such that,  up to a subsequence,  $w_n\to w$    in $L^2_{loc}(\R^N)\cap {L_s(\R^N)}$.     
Moreover, by \eqref{eq:w-n-nonzero}    we deduce that 
\be\label{eq:w-nonzero-eps-pp}
  \|w\|_{L^ 2(B_1)}^2  \geq \frac{1}{2} \qquad\textrm{ and } \qquad w_{B_1}= 0.  
\ee
In addition by \eqref{eq:groht-w-n-abs}, we have 
\be \label{eq:w-W-growth-ok}
  \|w\|_{L^2(B_{R})}^2  \leq  C  R^{N+2\g}  \qquad\textrm{ for every $R\geq 1$.}
\ee
Now applying Lemma \ref{lem:convergence-very-weak} to the equation \eqref{eq:w_n-solves} and using \eqref{eq:estim-1} together with  \eqref{eq:estim-2}, we get 
\begin{align*} 
& \left|\int_{\R^{2N}}(w_n(x)-w_n(y))(\psi(x)-\psi(y))\ov K_n(x,y)\, dxdy \right|\leq C  \Theta_n(r_n)^{-1/2} \qquad\textrm{ for all $\psi\in C^\infty_c(B_M)$.}
\end{align*}
Since $|\ov K_n-\mu_{a_n}|\leq \frac{\mu_1(x,y) }{n}$ almost everywhere in $B_{1/r_n}\times B_{1/r_n}$ and $\Theta_n(r_n)\to\infty$ as $n\to\infty$, we can apply Lemma \ref{lem:Ds-a-n--to-La} to deduce that 
  $\cL_{\mu_b}  w= 0 \quad\textrm{ in $\R^N$},$ 
 where $b$ is the weak-star limit of $a_n$ (which satisfies  \eqref{eq:def-a-anisotropi} for all $n\in \N$).
In  view of \eqref{eq:w-W-growth-ok},  by    Lemma \ref{lem:Liouville}, we deduce that   $w$ is  equivalent to a  constant function, since $\g<1$. This is clearly in contradiction   with       \eqref{eq:w-nonzero-eps-pp}. 
\end{proof} 
 As a consequence,  we get the following result.
 \begin{corollary}\label{cor:C-gamm-near-flat}
Let  $s\in (0,1)$, $N\geq 1$,  $\b\in [0,2s)$,  $\s\in(s,1]$, $\a'\geq 0$, with  $\a'+\s\in (0,2s)$, and $\k>0$. 
Let $a$ satisfy \eqref{eq:def-a-anisotropi}.
Consider $K$ and    $K'$ satisfying \eqref{eq:K-Kernel-satisf} and \eqref{eq:K'-Kernel-satisf}, respectively.   
 Let     $g\in H^s(B_2)\cap {L_s(\R^N)}$, $U\in C^{0,\s}_{loc}(\R^N)\cap  L_{(\a'+2s)/2}(\R^N)$ and $f\in \cM_\b$ satisfy 
 \be
 \cL_K  g+\cL_{K'} U =f \qquad\textrm{ in $B_2$.} 
\ee 
Then, for  $\g\in (0,1)\cap (0, 2s-\b]\cap(0,\s]$, there exist $\e_0, C>0$, such that if  $|K-\mu_a|<\e_0 \mu_1(x,y) $ in $B_2\times B_2\setminus \{x=y\}$,  we have   
$$
\|g\|_{C^{0,\g}(B_{1/4})}\leq C (\|g\|_{L^2(B_2)}+\|g\|_{{L_s(\R^N)}}+ [U ]_{C^{0,\s}(\R^N)}+ \|f\|_{\cM_\b}),
$$
with $C,\e_0>0$ depend only on $s,N,\b,\s,\a',\k$  and $\g$.
 \end{corollary}
 \begin{proof}
 Without loss of generality, we may assume that 
\be\label{eq:estim-tif-zmm}
 \|g\|_{L^2(B_2)}+\|g\|_{{L_s(\R^N)}}+ [U ]_{C^{0,\s}(\R^N)}+ \|f\|_{\cM_\b} \leq1.
\ee
 Let $z\in B_{1/2}$ and define $g_z:=g(x+z)$, $K_z(x,y)=K(x+z,y+z)$, $K_z'(x,y)=K(x+z,y+z)$,  $f_z(x)=f(x+z)$, $f_z(x)=f(x+z)$, and  $U_z(x)=U(x+z)$. We then have 
$$
 \cL_{K_z}  g_z+\cL_{K'_z} U_z =f_z \qquad\textrm{ in $B_1$.} 
$$
On the other hand by Lemma \ref{lem:estim-G_Ku}, 
  \be\label{eq:K-zvp-1gz}
 \cL_{K_z}  (\vp_1 g_z)+\cL_{K'_z} U_z  (\vp_1 g_z) =\ti f_z \qquad\textrm{ in $B_{1/2}$,} 
\ee
for some function $\ti f_z  $ 	satisfying
\be\label{eq:estim-tif-z}
 \|\ti f_z \|_{\cM_\b}\leq \| f_z \|_{\cM_\b} + C \| g_z \|_{{L_s(\R^N)}} \leq  C,
\ee
where we used \eqref{eq:estim-tif-zmm} for the last inequality.
By \eqref{eq:K-zvp-1gz} and Proposition \ref{prop:bound-Kato-abstract}, there exist $\e_0, C>0$, only depending on $s,N,\b,\s,\a',\k,\s$  and $\g$,  such that if $|K_z-\mu_a|<\e_0$ in $B_2\times B_2\setminus \{x=y\}$,  we get 
$$
\|g_z-(g_z)_{B_r}\|_{L^2(B_r)}= \|g-(g)_{B_r(z)}\|_{L^2(B_r(z))} \leq C r^{N/2+\g} \qquad\textrm{for every $r>0.$}
$$
It then follows, from \cite[ Lemma 3.1]{Fall-Reg}, that
$$ 
 \|g-g(z)\|_{L^2(B_r(z))} \leq C r^{N/2+\g}  \qquad\textrm{for every $z\in B_1$ and  $r\in (0,1).$}
 $$
 This implies that $\|g\|_{C^{\g}(B_{1/4})}\leq C$.   The proof is thus finished.
 \end{proof}
By  scaling and covering, we have the
  \begin{corollary}\label{cor:Holder-cont-coeff}
Let  $s\in (0,1)$, $N\geq 1$,   $\b\in [0,2s)$, $\s\in(s,1]$, $\a'\geq 0$, with  $\a'+\s\in (0,2s)$,  $\k>0$ and $\g\in (0,1)\cap(0,\s]\cap (0,2s-\b]$. 
Consider $K\in \ti \scrK_0^s(\k,0, Q_\infty)$ and  $K'$ satisfies \eqref{eq:K'-Kernel-satisf}.   
 Let  $g\in H^s(B_2)\cap {L_s(\R^N)}$, $U\in C^{0,\s}_{loc}(\R^N)\cap   L_{(\a'+2s)/2}(\R^N)$ and $f\in \cM_\b $ satisfy 
 \be\label{eq:clK-g-corrol}
 \cL_K  g+\cL_{K'} U =f \qquad\textrm{ in $B_2$.} 
\ee 
Then there exists   $C>0$,  only  depending   on $N,s,\a',\b, \k$ and $\g$, such that 
$$
\|g\|_{C^{0,\g}(B_1)}\leq C (\|g\|_{L^2(B_2)}+\|g\|_{{L_s(\R^N)}}+ [U ]_{C^{0,\s}(\R^N)}+  \|f\|_{\cM_\b}).
$$
 \end{corollary}
 \begin{proof}
Pick $x_0\in B_{3/2}$.   By   the continuity of $\cA_K(\cdot,\cdot,\th)$ (uniformly with respect to $\th$),  for every $\e>0$ there exists $\d=\d_{x_0,\e}\in (0,1/100)$ such that,  for every $x\in B_{4\d}(x_0)$,  $ r\in (0, 4\d)$ and $\th\in S^{N-1}$, we have 
%
\begin{align*}
 \left| K(x,x+r\th )-{\cA}_{K}(x_0,0,\th)  r^{-N-2s} \right |\leq \e r^{-N-2s} .
\end{align*}
Therefore,   for every $x\in B_{4\d}(x_0)$ and   $0<|z|<4\d$,
\begin{align*}
 \left| K(x,x+z)- {\cA}_{K}(x_0,0,z/|z|)   |z|^{-N-2s} \right |\leq  \e   |z|^{-N-2s} 
\end{align*}
and thus,   for every $x,y\in B_{2\d}(x_0)$, with $x\not=y$,
\begin{align}\label{eq:K-close-to-mu-a}
 \left| K(x,y)-  \mu_{ a}(x,y)  \right |\leq  \e   \mu_{ 1}(x,y),
\end{align}
where $  a(\th):=   {\cA}_{K}(x_0,0,\th)  $. By Definition \ref{def:Kernel-not-reg-the},  ${\cA}_{o,K}(x_0,0,\th)=0$ and thus  $  a$  satisfies \eqref{eq:a-satisf-elliptic}.
We now let  $K_{\d}(x,y)=\d^{N+2s}K( \d x+ x_0,\d y+ x_0)$ and  $K_{\d}'(x,y)=\d^{N+2s}K'( \d x+ x_0,\d y+ x_0)$, which satisfy \eqref{eq:K-Kernel-satisf} and \eqref{eq:K'-Kernel-satisf}, respectively. 

For $x\in B_{2}$, we define $g_\d(x)=g(\d x+ x_0)$, $U_\d(x)=U(\d x+ x_0)$ and  $f_\d(x)=\d^{2s}f(\d x+ x_0)$.  Since $\d\in (0,1/16)$,  by a change of variable in \eqref{eq:clK-g-corrol}, we get
\be \label{eq:u-delta-satisf-final}
\cL_{K_{\d}} g_\d+\cL_{K'_\d} U_\d = f_\d \qquad\textrm{ in $B_{8}$}.
\ee
On the other hand  \eqref{eq:K-close-to-mu-a} becomes
\begin{align*}
 \left| K_\d(x,y)-  \mu_{ a}(x,y)  \right |\leq  \e   \mu_{ 1}(x,y) \qquad\textrm{ for $x\not=y\in B_2$.}
\end{align*}
From this and \eqref{eq:u-delta-satisf-final}, then  provided $\e>0$ small,    by  Corollary \ref{cor:C-gamm-near-flat} and a change of variable,   we get
$$
\|g\|_{C^{\a}(B_{\d_{x_0,\e}}(x_0))}   \leq   C(x_0)  \left(    \|  g\|_{L^2(B_{2})}+   \|g\|_{{L_s(\R^N)}} +  \|f\|_{\cM_\b}+[U]_{C^{0,\s}(\R^N)} \right) ,
$$
where $C(x_0)$ is a constant, only depending on $N,s,c_0,\d_{x_0},\k,\t,\a,\a',\s,\g$ and $ x_0$.  Next, we cover $\ov B_{1}$ with  a finite number of balls $B_{\frac{1}{2}\d_{x_i,\e}}(x_i)\subset\subset B_{3/2}$, for $i=1,\dots,n$, with $x_i\in \ov B_{1}$.   It then follows that  
$$
\|g\|_{C^{\a}(B_1)} \leq   C'  \left(    \|  g\|_{L^2(B_{2})}+   \|g\|_{{L_s(\R^N)}} +  \|f\|_{\cM_\b}+[U]_{C^{0,\s}(\R^N)} \right) ,
$$
\end{proof}
We have  the following generalization.
  \begin{corollary} \label{cor:Holder-many-K'}
Let  $s\in (0,1)$, $\b\in [0,2s)$,   $\s_i\in(s,1]$. Let   $\k>0$ and 
$$
\g\in (0,1)\cap(0,\min_{1\leq i\leq \ell }\s_i]\cap (0,2s-\b].
$$
Consider $K\in\ti  \scrK_0^s(\k,0,Q_\infty)$ and   $K'_i$ satisfying \eqref{eq:K'-Kernel-satisf}, for $i=1,\dots,\ell$, and for some, $\a_i'\geq 0$, with  $\a_i'+\s_i\in (0,2s)$.
 Let  $g\in H^s (B_2)\cap L_s(\R^N)$, $U_i\in C^{0,\s_i}_{loc}(\R^N)\cap  L_{(\a'+2s)/2}(\R^N)$ and $ f\in \cM_\b  $ satisfy 
 \be
 \cL_K  g+\sum_{i=1}^\ell \cL_{K_i'} U_i  =f \qquad\textrm{ in $B_2$.} 
\ee 
Then,   there exists $C>0$, only  depending  on $s,N,\b, \a_i,\s_i , \k,\ell$ and $\g$, such that 
$$
\|g\|_{C^{0,\g}(B_1)}\leq C ( \|  g\|_{L^2(B_{2})}+   \|g\|_{{L_s(\R^N)}}+\sum_{i=1}^\ell   [U_i ]_{C^{0,\s_i}(\R^N)}+ \|f\|_{\cM_\b}).
$$
 \end{corollary}

\section{Gradient estimates} \label{s:GradEstim}
In this section, we consider the fractional parameter $s\in (1/2,1)$ and we prove H\"older estimates of $\n u$.
For $g\in L^2_{loc}(\R^N)$ and     $r>0$,  we define
\be \label{eq:def:Pru-x}
 \textbf{P}_{r,g}(x)=g_{B_r}+ T^{r,g}\cdot x= g_{B_r}+ \sum_{i=1}^N T^{r,g}_i x_i,
\ee
where
\be \label{eq:def-Tui}
T^{r,g}_i=\frac{\la g, x_i\ra_{L^2(B_r)}}{\|x_i\|_{L ^2(B_r)}^2}.
\ee
Note that  $\textbf{P}_{r,g} $ is the  $L^2(B_r)$-projection of $g$ on the space of affine functions.\\
%
In view of Corollary  \ref{cor:Holder-cont-coeff}, we know that the  solutions $u$ to $\cL_K u=f$ in $B_2$ are of class $C^{1-\varrho}(B_1)$ for every small $\varrho>0$, provided $K\in \ti  \scrK_0^s(\k,0, Q_\infty)$   and  $f\in \cM_\b$, with   $\b\in [0,2s)$. In particular $|T^{r,u-u(0)}|\leq C  r^{-\varrho}$.  The   result  below improves this to H\"older regularity estimates of the gradient of $u$ when $2s-\b>1$ and $K\in \ti  \scrK_0^s(\k,\a, Q_\infty)$, with  $\a>0$. 
%
 \begin{proposition} \label{prop:bound-wth-corrector}
Let $N\geq 1$,   $s\in (1/2,1)$, $\k,c_0>0$. Consider $\a\in (0,2s-1)$,  $\b\in [0,2s)$,  $\varrho\in [0,1)$ such that 
$$
 \g:= \min (1-\varrho+\a, 2s-\b) >1.
$$
 Then there exists  $C=C(N,s,\a,\b,\k,c_0,\varrho)  >0$ such that if:
\begin{itemize}
\item  $K\in \ti  \scrK_0^s(\k,\a, Q_\infty)$,  
\item   $g\in H^s(B_2)\cap L^{\infty}(\R^N)$  and $f\in  \cM_\b$ with 
$$
\|g\|_{L^\infty(\R^N)}+  \|f\|_{\cM_\b}\leq 1, 
$$
$$
 | T^{r,g}| \leq c_0 r^{-\varrho}\qquad\textrm{ for all $r>0$}
$$
  are such that  
$$
 \cL_K  g =f \qquad\textrm{ in $B_2$}, 
$$
\end{itemize}  
 then     we have 
$$
 \sup_{r>0} r^{-\g}     \left\| g-  \textrm{\em \textbf{P}}_{r,g}\right\|_{L^\infty(B_r)}\leq  C  .
$$
\end{proposition}
\begin{proof}
Suppose on the contrary that the assertion in the proposition does not hold. Then as in  the proof of  Proposition \ref{prop:bound-Kato-abstract},  for all $n\geq 2$, there exist     
\begin{itemize}
\item $r_n>0$,    $K_n\in \ti  \scrK_0^s(\k,\a,Q_\infty)$,  
\item   $g_n\in H^s(B_2)\cap L^{\infty}(\R^N)$  and $f_n\in  \cM_\b$ satisfying 
\be \label{eq:g-n-T_n-satis}
\|g_n\|_{L^\infty(\R^N)}+   \|f_n\|_{\cM_\b}\leq 1, \qquad  | T^{r,g_n}| \leq c_0 r^{-\varrho} \quad\textrm{ for all $r\in (0,\infty)$,}
\ee
 \be\label{eq:g-nsolves-Grad-estim}
 \cL_{K_n}  g_n  =f_n \qquad\textrm{ in $B_2$},
\ee
\item a nonincreasing  function  $\Theta_n: (0,\infty)\to [0,\infty)$ satisfying
\be \label{eq:sup-r-ok-Morrey-hi-Int}
\Theta_n( r)\geq  r^{-\g}   \| g_n- \textbf{P}_{r, g_n} \|_{L^\infty(B_{r})}   \qquad\textrm{ for every $r\in (0,\infty)$ and $n\geq 2$,}
\ee
\end{itemize}  
with the properties  that      $r_n\to 0$ as $n\to\infty$ and 
$$
 r^{-\g}_n   \| g_n- \textbf{P}_{r_n,g_n} \|_{L^\infty(B_{r_n})} \geq \frac{1}{2}\Theta_n(r_n) \geq\frac{n}{4} .
$$
 %
%
 We   define  
$$
 {v}_n(x)=  \Theta_n(r_n)^{-1}  r^{-\g}_n  [g_n(r_n x) - \textbf{P}_{r_n,g_n}(r_n x) ], 
$$
so that 
\be \label{eq:E1}
 \|v_n\|_{L^\infty(B_{1})}\geq \frac{1}{2}.
\ee
In addition, by a change of variable, we get 
\be\label{eq:E2}
 \int_{B_1} v_n(x)  \, dx= \int_{B_1} v_n(x) x_i\, dx=0 \qquad\textrm{ for every $i\in \{1,\dots,N\}$.}
\ee
 Since  $\Theta_n$ is nonincreasing and $\g>1$ then see e.g. \cite{Fall-Reg,Serra-OK}, inequality \eqref{eq:sup-r-ok-Morrey-hi-Int} always implies that 
\be\label{eq:groht-w-n-Morrey-HI-Grad}
 \|v_n\|_{L^\infty(B_{R})}\leq  C   R^{\g}  \qquad\textrm{ for every $R\geq 1$,}
\ee
for some constant $C=C(N,\g)$.\\
From \eqref{eq:g-nsolves-Grad-estim}, we deduce that 
\be \label{eq:eq-K-n-v-n}
 \cL_{\ov K_n} v_n+  \Theta_n(r_n)^{-1}  r^{1-\g}_n  T^{r_n,g_n}\cdot \cL_{\ov K_n} x  =\ov f_n \qquad\textrm{ in $B_{2/r_n}$},
\ee
where  $\ov K_n(x,y):=r_n^{N+2s}K_n(r_n x,r_n y)$ and $\ov f_n(x)=r_n^{2s} f_n(r_n x)$. 
Then, since $\cA_{\ov K_n}(x,r,\th)= {\cA}_{K_n}(r_n x,r_n r,\th)$ and ${\cA}_{ K_n}\in C^{0,\a}(Q_\infty)\times L^\infty(S^{N-1} )$, we get
$$
| {\cA}_{\ov K_n}(x,r,\th)- {\cA}_{ K_n}(0,0,\th)|\leq C \min( r_n^\a (|x|+r)^\a ,1)\qquad\textrm{ for all $x\in \R^N$, $\th\in S^{N-1}$ and $r>0$.}
$$
Moreover, recalling Definition \ref{def:Kernel-not-reg-the},  we have  ${\cA}_{o, K_n}(0,0,\th)=0 $ for all $\th\in S^{N-1}$.
Letting $a_n(\th):= {\cA}_{\ov K_n}(0,0,\th)$ and $ \ov K_n'(x,y):=r_n^{-\a}(\ov K_n(x,y)-\mu_{a_n}(x,y))$,  we immediately see that 
\be\label{eq:oK_n-close-to mua_n}
| \ov K_n(x,y)-\mu_{a_n}(x,y)|\leq C\min( r_n^\a (|x|+|y|)^\a ,1) \mu_1(x,y) 
\ee
and 
\be\label{eq:K-nprimes-Grad-estim}
| \ov K_n'(x,y)|\leq C (|x|+|y|)^\a\mu_1(x,y).
\ee
Since $\cL_{\mu_{a_n} } x_i =0$  on $\R^N$, we can rewrite \eqref{eq:eq-K-n-v-n} as 
\be\label{eq:eq-K-n-v-n-ok}
 \cL_{\ov K_n} v_n+\Theta_n(r_n)^{-1}  \sum_{i=1}^N    \ov T_n^i  \cL_{\ov K'_n} x_i   = r_n^{-\g}  \Theta_n(r_n)^{-1}   \ov f_n  \qquad\textrm{ in $B_{2/r_n}$},
\ee
where (recall \eqref{eq:def:Pru-x}) $\ov T_n^i:=    r^{1+\a -\g }_n T^{r_n,g_n}_i$. Note that $x_i\in L_{(\a+2s)/2}(\R^N)$, provided $\a\in (0, 2s-1)$ and $[x_i]_{C^{0,1}(\R^N)}\leq 1$. Clearly by \eqref{eq:g-n-T_n-satis}, 
\be \label{eq:E10} 
\|\ov f_n\|_{\cM_\b}\leq r_n^{2s-\b} \qquad \textrm{ and } \qquad  |\ov T_n^i|\leq  c_0 r_n^{1+\a-\g-\varrho}\leq c_0 .
\ee
Since $\ov K_n\in \ti  \scrK_0^s(\k,0,Q_\infty)$ and $\Theta_n(r_n)\to \infty$ as $n\to\infty$,  applying Corollary \ref{cor:Holder-many-K'} to \eqref{eq:eq-K-n-v-n-ok} and using \eqref{eq:E10} together with \eqref{eq:K-nprimes-Grad-estim}, we find that $v_n$ is bounded in $C^{1-\d}_{loc}(\R^N)$, for all $\d\in (0,1)$. Hence, provided $\d$ is small,  there exists $v\in C^{s+\d}_{loc}(\R^N)$ such that, up to a subsequence,   $v_n\to v$ in $C^0_{loc}(\R^N)$.  
  Hence by \eqref{eq:groht-w-n-Morrey-HI-Grad}, up to a subsequence,  $v_n$ converges strongly,  in $  {L_s(\R^N)}$,  to   $v\in H^s_{loc}(\R^N)\cap  {L_s(\R^N)}$.   
Moreover, by \eqref{eq:E1},   we deduce that 
\be\label{eq:w-nonzero-eps}
  \|v\|_{L^ \infty(B_1)}  \geq \frac{1}{2} \qquad\textrm{ and }  \qquad \int_{B_1} v(x)  \, dx= \int_{B_1} v(x) x_i\, dx=0 \quad\textrm{ for every $i\in \{1,\dots,N\}$.} 
\ee
In addition, passing to the limit in  \eqref{eq:groht-w-n-Morrey-HI-Grad}, we have 
\be \label{eq:w-WW-zero-mean}
 \|v\|_{L^\infty(B_{R})}  \leq    R^{\g}  \qquad\textrm{ for every $R\geq 1$.}
\ee
We observe that $a_n$ satisfies \eqref{eq:def-a-anisotropi} for all $n$.  By \eqref{eq:oK_n-close-to mua_n}, Lemma \ref{lem:convergence-very-weak} and  Lemma \ref{lem:Ds-a-n--to-La}, we can pass to the limit in \eqref{eq:eq-K-n-v-n-ok}, to  get 
  $\cL_{\mu_b} v= 0 \quad\textrm{ in $\R^N$},$ 
 where $b$ is the weak-star limit of $a_n$.
Now, since $v\in H^s_{loc}(\R^N)$ and satisfies \eqref{eq:w-WW-zero-mean},  by   Lemma \ref{lem:Liouville} we deduce that    $v$ is    an affine function,  because $\g <2s$. This is clearly in contradiction   with       \eqref{eq:w-nonzero-eps}.

\end{proof}
A first consequence of the previous result is the 
%
\begin{corollary}\label{cor:Holder-reg-Global}
Let  $s\in (1/2,1)$, $\b\in[0,2s-1)$,   $N\geq1$ and  $\k>0$. Let   $\a\in (0,2s-1)$ and $\varrho\in [0,\a)$. 
Let   $K\in \ti  \scrK_0^s(\k,\a, Q_\infty)$,    $g\in H^s(B_2)\cap C^{0,1-\varrho}(\R^N) $  and $f\in  \cM_\b$  satisfy
 \be
 \cL_K  g =f \qquad\textrm{ in $B_2$}.
\ee
Then,  there exists $C>0$, only depending  on $s,N,\b,\a,\k,\varrho$, such that 
$$
\| g\|_{C^{1,\min( 1+\a-\varrho ,2s-\b) -1 }(B_{1})}\leq C ( \|g\|_{  C^{0,1-\varrho}(\R^N) }+   \|f\|_{  \cM_\b }).
$$
\end{corollary}
\begin{proof}
Put $A:=    \|g\|_{   C^{0,1-\varrho}(\R^N) } +  \|f\|_{\cM_\b}.$  We define  $ G_z(x):=g(x+z)-g(z)$, for $z\in B_{1}$.  Since $G_z(0)=0$,  we have    $|T^{r,G_z}|\leq C(\varrho,N) r^{-\varrho} [g]_{C^{0,1-\varrho}(\R^N)} \leq C A$, for $r>0$. Obviously, $\cL_{K_z} G_z=f(\cdot +z)$ in $B_1$, where $K_z(x,y)=K(x+z,y+z)$.
We then  apply Proposition \ref{prop:bound-wth-corrector}   to  get  a constant $C>0$, only depending  on $s,N,\b,\a,\k,\varrho$, such that
$$
 \sup_{r>0} r^{-\g}     \left\| G_z-  \textbf{P}_{r, G_z} \right\|_{L^\infty(B_r)}\leq  C  A,
$$
where $\g:=\min ( 1+\a-\varrho ,2s-\b) >1$.   
By a well known iteration argument (see e.g \cite{Serra-OK}), we find that 
$$
|g(x)- g(z)-    T(z)\cdot( x-z) |\leq C A |x-z|^{\g} \qquad\textrm{ for every $x,z\in B_{1/2}$,}
$$
for some $T$, satisfying $\|T\|_{L^\infty(B_1)}\leq C A$. Since $\g>1$,    then $\n u(z)= T(z)$. 
By   a classical extension theorem (see e.g. \cite{Stein}[Page 177], we deduce that 
  $u\in C^{1, \g-1}(\ov {B_{1/2}})$.    Moreover
$$
\|u \|_{C^{1,\g-1}(\ov {B_{1/2}})}\leq C A.
$$
\end{proof}
By a bootstrap argument, we have the following result.
 \begin{theorem}  \label{th:abs-res-Propo}
Let  $s\in (1/2,1)$, $\b\in [0,2s-1)$,  $\a\in (0,2s-1)$ and $\k>0$.  Let  $K\in  \ti \scrK_0^s(\k,\a,Q_\infty)$,     $u\in H^s(B_2)\cap {L_s(\R^N)}$ and $V,  f\in  \cM_\b$   such that  
$$
 \cL_K  u+ Vu  =f \qquad\textrm{ in $B_2$}.
$$
Then     
$$
 \|u\|_{C^{1,\min (\a, 2s-\b-1)}(B_{1})} \leq  C  (  \|u\|_{  L^2(B_2) }+ \|u\|_{{L_s(\R^N)}}+    \|f\|_{  \cM_\b}),
$$
where $C>0$ only depends on $s,N,\k,\a,\b$  and $\|V\|_{\cM_\b}$.
\end{theorem}
\begin{proof}
Since $2s-\b>1$, by  \cite{Fall-Reg}, for every $\varrho\in (0,1)$,    there exists $C=C(N,s,\b,\a,\|V\|_{  \cM_\b},\varrho)>0$ such that 
$$
 \|u\|_{C^{1-\varrho}(B_{1})} \leq  C  (  \|u\|_{  L^2(B_2) }+ \|u\|_{{L_s(\R^N)}}+    \|f\|_{  \cM_\b}).
$$
Using Lemma \ref{lem:estim-G_Ku}$(i)$,  we apply first Corollary \ref{cor:Holder-reg-Global}     to get 
$$
 \|\vp_{1/2} u\|_{C^{0,1}(B_{1/8})} \leq  C  (  \| \vp_{1/2}  u\|_{  L^2(\R^N) }+  \|\vp_{1/2} u\|_{C^{1-\varrho}(\R^N)} + \|  u\|_{{L_s(\R^N)}}+ \|\vp_{2} Vu \|_{  \cM_\b} +   \|f\|_{  \cM_\b}).
$$
Therefore, 
$$
 \|u\|_{C^{0,1}(B_{2^{-3}})} \leq  C  (  \|u\|_{  L^2(B_2) }+ \|u\|_{{L_s(\R^N)}}+    \|f\|_{  \cM_\b}).
$$
We then apply once more  Corollary \ref{cor:Holder-reg-Global} (with   $\varrho=0$) and use  Lemma \ref{lem:estim-G_Ku}$(i)$  to obtain
$$
 \|\vp_{2^{-4}} u\|_{C^{1, \min(2s-\b-1,\a)}(B_{2^{-8}})} \leq  C  (  \|\vp_{2^{-4}} u\|_{  L^2(\R^N) }+  \|\vp_{2^{-4}} u\|_{C^{0,1}(\R^N)} + \|u\|_{{L_s(\R^N)}}+ \|\vp_{2^{-4}}Vu \|_{  \cM_\b} +  \|f\|_{  \cM_\b}),
$$
so that 
$$
 \| u\|_{C^{1, \min(2s-\b-1,\a)}(B_{2^{-8}})} \leq  C  (  \|u\|_{  L^2(B_2) }  + \|u\|_{{L_s(\R^N)}}+  \|f\|_{  \cM_\b}),
$$
with $C$ as in the statement of the theorem.
 After a covering argument, we obtain the result.
\end{proof}

  \section{Schauder estimates} \label{s:Shaud}
Here and in the following, given $u\in C^{2s+\a}_{loc}(\R^N)$, with $\a>0$, we let
\be\label{eq:delta-eo}
\d^e u(x,r,\th):=\frac{1}{2}(2u(x)-u(x+r\th)-u(x-r\th)),\qquad \d^o u(x,r,\th):=\frac{1}{2}(u(x+r\th)-u(x-r\th) ).
\ee
For   $A\in C^{m,\a}(Q_\infty)\times L^\infty(S^{N-1} )$,  we define 
\be \label{eq:defEs}
\cE^s_{A,u}(x):=\int_{S^{N-1}}\int_0^\infty \d^e u(x,r,\th)A(x,r,\th) r^{-1-2s}\, dr d\th
\ee
and  for $B\in \cC^{m}_{\min(1,\a+(2s-1)_+)}(Q_\infty)\times L^\infty(S^{N-1} )$,  we define
\be \label{eq:defOs}
\cO^s_{B,u}(x):=\int_{S^{N-1}}\int_0^\infty \d^o u(x,r,\th)B(x,r,\th) r^{-1-2s}\, dr d\th.
\ee
We observe that,  using the symmetry of $K\in \ti \scrK^s_{\min(1,\a+(2s-1)_+)}(\k,\a,Q_\infty)$ and a change of variables, we get 
\begin{align}\label{eq:Operator-splitting}
&\frac{1}{2}\int_{\R^{2N} }(u(x)-u(y))(\psi(x)-\psi(y))K(x,y)\, dxdy =\int_{\R^N}\psi(x) \cE^s_{{\cA}_{e,K},u}(x) \, dx +\int_{\R^N}\psi(x)\cO^s_{{{\cA}_{o,K},u}}(x)   \, dx ,
\end{align}
 where 
$$
{\cA}_{e,K}(x,r,\th):=\frac{1}{2} ({\cA}_{K}(x,r,\th) +{\cA}_{K}(x,r,-\th)) , \qquad{\cA}_{o,K}(x,r,\th):=\frac{1}{2} ({\cA}_{K}(x,r,\th) -{\cA}_{K}(x,r,-\th)) .
$$

We have the following result which will be proved in Section \ref{s:Appendix}.
\begin{lemma}\label{lem:2sp-alph-estim-F_Keo}
Let $s\in (0,1)$ and $\a\in (0,1)$. Let $\k>0$,    $m\in \N$, $A\in C^{m+\a}(Q_\infty) \times L^\infty(S^{N-1} )$ and $B\in \cC^{m}_{\t}(Q_\infty) \times L^\infty(S^{N-1} )$, with $\t:=\min(\a+(2s-1)_+,1)$.
\begin{itemize}
\item 
 Let $u\in C^{2s+\a+m}(\R^N) $ and $2s\not=1$. If $2s+\a< 1$ or $1<2s+\a<2$,  then  
$$
\| \cE^s_{A,u} \|_{C^{m+\a}(\R^N)} \leq C  \|A\|_{ C^{m+\a}(Q_\infty)\times L^\infty(S^{N-1} )}\|u\|_{C^{2s+\a+m}(\R^N) }
$$
and 
$$
 \| \cO^s_{B,u} \|_{C^{m+\a}(\R^N)}\leq  C  \|B\|_{ \cC^{m}_{\t}(Q_\infty)\times L^\infty(S^{N-1} )} \|u\|_{C^{2s+\a+m}(\R^N) },
$$
with $C=C(N,s,\a,m) $.
\item   Let $u\in C^{1+\a+m+\e}(\R^N) $, for some $\e\in (0,1-\a)$.  If $2s=1$ and $B\in \cC^{m}_{\a+\e}(Q_\infty)\times L^\infty(S^{N-1} )$,  then  
$$
\| \cE^s_{A,u} \|_{C^{m+\a}(\R^N)} \leq C  \|A\|_{ C^{m,\a}(Q_\infty)\times L^\infty(S^{N-1} )}\|u\|_{C^{2s+\a+m+\e}(\R^N) }
$$
and 
$$
 \| \cO^s_{B,u} \|_{C^{m+\a}(\R^N)}\leq  C  \|B\|_{ \cC^{m}_{\a+\e}(Q_\infty)\times L^\infty(S^{N-1} )} \|u\|_{C^{2s+\a+m+\e}(\R^N) },
$$
with $C=C(N,\a,m,\e) $.
\item 
 Let $u\in C^{2s+\a+m}(\R^N) $ and $2s+\a>2$. If  $A,B\in C^{m+2s-1+\a}(Q_\infty)\times L^\infty(S^{N-1} )$ and $B\in \cC^{m+1}_{2s+\a-2}(Q_\infty)\times L^\infty(S^{N-1} )$ ,  then  
$$
\| \cE^s_{A,u} \|_{C^{m+\a}(\R^N)} \leq C  \|A\|_{ C^{m+2s-1+\a}(Q_\infty) \times L^\infty(S^{N-1} )}\|u\|_{C^{2s+\a+m}(\R^N) }
$$
and 
$$
 \| \cO^s_{B,u} \|_{C^{m+\a}(\R^N)}\leq  C \left( \|B\|_{ \cC^{m+1}_{2s+\a-2}(Q_\infty)\times L^\infty(S^{N-1} )}  +  \|B\|_{ C^{m+2s-1+\a}(Q_\infty)\times L^\infty(S^{N-1} )} \right)\|u\|_{C^{2s+\a+m}(\R^N) },
$$
with $C=C(N,s,\a,m) $.
\end{itemize}

\end{lemma}
We remark that under the assumptions on $A$ and $B$,  for $2s=1$, the first assertion of Lemma \ref{lem:2sp-alph-estim-F_Keo}  does not in general hold.  
\subsection{Schauder estimates}
The following result is intended to the $C^{2s+\a}$ regularity estimates, for $2s+\a\not\in \N$. To deal with the case $2s+\a>2$, we look for optimal growth estimate of the difference between $u$ a   second order polynomial  that is close to $u$ in the $L^2$-norm.\\
 For $g\in L^2(B_r)$ and  $i,j=1,\dots,N$,  we define
 $$
T_{ij}^{r,g}=\frac{1}{ \|y_iy_j\|_{L^2(B_r)}^2}\int_{B_r} y_iy_jg(y)\, dy
$$
and 
$$
\textbf{Q}_{r,g}(x)=\sum_{i,j=1}^N T_{ij}^{r,g} x_ix_j.
$$
We note that $\textbf{Q}_{r,g} $ is nothing but the $L^2(B_r)$-projection of $g$ on the  space of  homogeneous quadratic polynomials. We now state the main result of this section.
%
 \begin{proposition} \label{prop:bound-wth-corrector-nab}
Let $N\geq 1$,   $s\in (0,1)$, $\k>0$, $\a\in (0,1)$, $\b\in (0,\a)$ and $\e>0$.  Let  $K\in\ti \scrK_{\t}^s(\k, \a, Q_\infty)$, with $\t:=\min (\a+(2s-1)_+,1)$.  Let
  $f\in C^\a(\R^N)$ and      $g\in C^{2s+\b}(\R^N)$, for $2s+\b\not\in \N$,   
such that 
$$
 \cL_K  g =f \qquad\textrm{ in $B_2$}, 
$$
\begin{itemize}
\item[$(i)$] If $1<2s+\a<2$ and $2s\not=1$,      then       there exists $C=C(N,s,\k,\a,\b)>0$ such that
$$
 \sup_{r>0} r^{-(2s-1+\a-\b)}     [ \n g ]_{C^{\b}(B_r)}\leq  C (\|g\|_{C^{2s+\b}(\R^N)}+   [f]_{C^\a(\R^N)})  .
$$
\item[$(ii)$] If $2s+\a>2 >2s+\b\geq 1+\a$, $K\in\ti \scrK_{0}^s(\k, 2s-1+\a, Q_\infty)$ and $g(0)=|\n g(0)|=0$,      then       there exists $C=C(N,s,\k,\a,\b)>0$ such that
$$
 \sup_{r>0} r^{-(2s-1+\a-\b)}     [ \n g -\n \textrm{\em\textbf{Q}}_{r,g} ]_{C^{\b}(B_r)}\leq  C (\|g\|_{C^{2s+\b}(\R^N)}+   [f]_{C^\a(\R^N)})  .
$$
\item[$(iii)$] If $2s=1$ and $K\in \ti \scrK_{\a+\e}^s(\k, \a, Q_\infty)$,     then   there exists $C=C(N,s,\k,\a,\b,\e)>0$ such that     
$$
 \sup_{r>0} r^{-(2s-1+\a-\b)}     [ \n g ]_{C^{\b}(B_r)}\leq  C (\|g\|_{C^{2s+\b}(\R^N)}+   [f]_{C^\a(\R^N)})  .
$$
\item[$(iv)$] If $2s+\a<1$,  then     there exists $C=C(N,s,\k,\a,\b)>0$ such that   
$$
 \sup_{r>0} r^{-(\a-\b)}     [  g ]_{C^{2s+\b}(B_r)}\leq  C  (\|g\|_{C^{2s+\b}(\R^N)}+   [f]_{C^\a(\R^N)}).
$$
\end{itemize}

\end{proposition}
\begin{proof}
We  start with $(i)$.
%
Assume that the assertion in  $(i)$ does not hold, then arguing as in the proof of Proposition \ref{prop:bound-Kato-abstract},       we can find   sequences
\begin{itemize}

\item  $r_n> 0$,   $K_n\in \ti \scrK_{\a+2s-1}^s(\k, \a , Q_\infty)$    $g_n\in C^{2s+\b}(\R^N)$  and $f_n\in C^{\a}(\R^N)$ with 
$
\|g_n\|_{C^{2s+\b}(\R^N)}+   [f_n]_{C^\a(\R^N)}\leq 1, 
$
%
 \be\label{eq:g-nsolves-Schauder}
 \cL_{K_n}  g_n  =f_n \qquad\textrm{ in $B_2$},
\ee

\item $\Theta_n:(0,\infty)\to [0,\infty)$,  nonincreasing,  
\end{itemize} 
with the properties  that  $r_n\to 0$ as $n\to \infty$,     
\be \label{eq:sup-r-ok-Morrey-hi-Int-nab}
\Theta_n( r)\geq  r^{-(2s-1+\a-\b)}  [ \n g_n ]_{C^{\b}(B_{r})}   \qquad\textrm{ for every $r>0$ and $n\geq 2$}
\ee
and  
\be \label{eq:non-degn-u-14}
 r^{-(2s-1+\a-\b)}_n [ \n g_n ]_{C^{\b}(B_{r_n})}\geq \frac{1}{2}\Theta_n(r_n)\geq \frac{n}{4} .
\ee
We define 
$$
u_n(x):=\frac{1}{r_n^{2s+\a}\Theta_n(r_n)} g(r_n x). 
$$
 By  \eqref{eq:sup-r-ok-Morrey-hi-Int-nab}, for $R\geq 1$, we have 
\begin{align*}
[ \n u_n ]_{C^{\b}(B_{R})}= \frac{ r_n^{1+\b} }{r_n^{2s+\a}\Theta_n(r_n)} [  \n g_n ]_{C^{\b}(B_{r_n R})}\leq  R^{ 2s-1+\a-\b} \frac{ \Theta_n(Rr_n) }{ \Theta_n(r_n)} .
\end{align*} 
Hence by the  monotonicity of $\Theta_n$, we have 
\be\label{eq:un-semi-norm-groth} 
[ \n u_n ]_{C^{\b}(B_{R})}  \leq  R^{2s-1+\a-\b}  \qquad\textrm{ for all $R\geq 1$.}
\ee
 In addition, by \eqref{eq:non-degn-u-14}, we get 
$$
   [ \n u_n ]_{C^{\b}(B_{1})} \geq \frac{1}{2} .
$$
 This then implies that there exists $x_n, h_n\in B_{1}$, with $h_n\not=0$, such that 
\be \label{eq:non-deg-u-n}
  |h_n|^{-\b} | \n u_n(x_n+ h_n)-\n u_n(x_n)| \geq \frac{1}{4} .
 \ee
 We define the new sequence 
$$
 v_n(x):= \frac{ u_n(x_n+ |h_n| x)-  u_n(x_n)- |h_n|\n u_n(x_n)\cdot x}{  |h_n|^{1+\b} }.
$$
By construction,  we have   that
 \be \label{eq:v_n0-nv_n0}
v_n(0)=|\n v_n(0)|=0
\ee
and by \eqref{eq:non-deg-u-n}, 
\be\label{nonded-v-n}
|\n v_n (h_n/|h_n|)|  \geq \frac{1}{4} .
\ee
Moreover by \eqref{eq:un-semi-norm-groth}, for $R\geq 1$ and  $x, y\in B_R$, 
\begin{align*}
|\n v_n(x)-\n v_n (y)|&=   |h_n|^{-\b}  | \n u_n(x_n+ |h_n| x)-  \n u_n(x_n+ |h_n| y)  |\\
& \leq     |h_n|^{-\b}  |h_n|^{\b}|x-y|^{\b}  (1+R)^{2s-1+\a-\b}\leq 2^{2s-1+\a-\b} |x-y|^{\b}R^{2s-1+\a-\b}.
\end{align*}
Combining this with \eqref{eq:v_n0-nv_n0},  we get, for all $R\geq1$,
\be \label{eq:groht-w-n-Morrey-HI}
 \|\n v_n\|_{L^{\infty}(B_R)}\leq CR^{2s-1+\a}
\ee
and 
$$
 \| v_n\|_{C^{1,\b}(B_R)}\leq C R^{2s+\a},
$$
for some $C=C(s,\a,\b)$.
This latter  estimate implies that there exists $v\in C^{1,\b}_{loc}(\R^N)$ such that, up to a subsequence,  
\be\label{eq:lim-v-n}
v_n \to v \qquad\textrm{in $C^{1}_{loc}(\R^N)$}. 
\ee
Moreover by \eqref{nonded-v-n} and \eqref{eq:v_n0-nv_n0}, there exists $e\in S^{N-1}$ such that  
\be\label{eq:E1-nab00}
|\n v(e)|\geq \frac{1}{4} \qquad\textrm{ and } \qquad \n v(0)=0.
\ee
We shall show that   $\n v\equiv 0$ on $\R^N$, which leads to a contradiction. Indeed, given $h\in \R^N$,   we define $w_n(x)=(v_n)_{h,0}(x)=v_n(x+h)-v_n(x)$.  It follows from \eqref{eq:groht-w-n-Morrey-HI}  and the fundamental theorem of calculus that 
\be\label{eq:groht-w-n-Morrey-HI-nab}
 \| w_n\|_{L^\infty(B_{R})}\leq  C   R^{2s-1+\a}  \qquad\textrm{ for every $R\geq 1$,}
\ee
where here and in the following of the proof,  the letter  $C$ is a positive constant only depending on $h,\k,N,s,\b$ and $\a$.
We put $\rho_n:=r_n |h_n|$, $z_n:=r_n x_n$ and we define 
$$
\ov K_n(x,y)=\rho_n^{N+2s}  K_n( z_n+ \rho_n x,  z_n+ \rho_n  y) ,
$$
$$
\ov K_n'(x,y):=r_n^{-\a}|h_n|^{-\b}  \left[\ov K_n(x+h,y+h)-\ov K_n(x,y) \right],
$$
$$
U_n(x):= r_n^{-2s}|h_n|^{-1} g_n(z_n+ \rho_n x+ \rho_n  h )
$$
and 
$$
   \ov f_n(x):= r_n^{-\a-2s} |h_n|^{-\b-1}\rho_n^{2s} \left[  f_n(z_n+\rho_n x+\rho_n h )  - f_n(z_n+ \rho_n x  ) \right].
$$
For $h\in \R^N\setminus \{0\},$ we let $n$ large so that  $z_n+ h\in  B_{1/2r_n}$, by  changing variables and using \eqref{eq:g-nsolves-Schauder},  we then have that
$$
 \cL_{\ov K_n} w_n+\frac{1}{\Theta_n(r_n)} \cL_{\ov K'_n} U_n  = \frac{1}{\Theta_n(r_n)}\ov  f_n \qquad\textrm{ in $B_{1/2r_n}$.}
$$
Therefore by \eqref{eq:Operator-splitting}, 
\be \label{eq:eqw_n--ok}
 \cL_{\ov K_n} w_n  =F_n  \qquad\textrm{ in $B_{1/2r_n}$,}
\ee
where 
\be \label{eq:def-F_n-Shaud}
 F_n(x):=\frac{1}{\Theta_n(r_n)}\ov  f_n- \frac{1}{\Theta_n(r_n)}\cE^s_{{\cA}_{e,\ov K_n'},U_n} - \frac{1}{\Theta_n(r_n)}\cO^s_{{\cA}_{o,\ov K_n'},U_n}.
\ee
By a change of variable, we get
\begin{align*}
\cE^s_{{\cA}_{e,\ov K_n'},U_n}(x)
%
&=|h_n|^{2s-1}\int_{S^{N-1}}\int_0^\infty\d^e g_n(z_n+\rho_n x +\rho_n h,t,\th ) {\cA}_{e,\ov K_n'}(x,t/\rho_n,\th) t^{-1-2s}\, dt d\th
\end{align*}
and 
\begin{align*}
\cO^s_{{\cA}_{o,\ov K_n'},U_n}(x)
%
&=|h_n|^{2s-1}\int_{S^{N-1}}\int_0^\infty\d^o g_n(z_n+\rho_n x +\rho_n h,t,\th ) {\cA}_{o,\ov K_n'}(x,t/\rho_n,\th) t^{-1-2s}\, dt d\th.
\end{align*}
We recall  that
$$
{\cA}_{\ov K_n}(x,r,\th):= {\cA}_{ K_n}( z_n+\rho_n x,\rho_n r,\th)
$$
and 
\be \label{eq:will-saveme}
{\cA}_{\ov K_n'}(x,r,\th):=r_n^{-\a}|h_n|^{-\b}\{ {\cA}_{ K_n}(z_n+  \rho_n x+\rho_n h,\rho_n r,\th)- {\cA}_{ K_n}(z_n+  \rho_n x,\rho_n r,\th)   \}.
\ee
Since $K_n\in\ti \scrK_{\a+2s-1} ^s(\k,\a,Q_\infty)$ (recall Definition \ref{def:Kernel-not-reg-the}),  for all $x,y\in \R^N$, $\th \in S^{N-1}$ and $r>0$,  
\be\label{eq:assump-on-ti-l-Ko-inProp-n}
\begin{aligned}
&|  {\cA}_{o,K_n}(x,r,\th) |\leq \frac{1}{\k}\min (r,  1)^{2s-1 +\a }, \\
&|  {\cA}_{o,K_n}(x,r,\th)- {\cA}_{o,K_n}(y,r,\th)|\leq \frac{1}{\k} \min (r, |x -y|)^{2s-1 +\a }
\end{aligned}
\ee
and
$$
|  {\cA}_{K_n}(x,r,\th)- {\cA}_{K_n}(y,r,\th)|\leq \frac{1}{\k} |x-y|^{\a}.
$$
Therefore,
$$
 |{\cA}_{e,\ov K_n'}(x,r,\th)|\leq C  |h_n|^{\b}\leq C \qquad\textrm{ for all $x\in \R^N$, $r>0$ and $\th\in S^{N-1}$.}
$$
Consequently, since $\|g_n\|_{C^{2s+\b}}\leq1$ and recalling \eqref{eq:delta-eo}, we have that 
 $$|\d^e g_n(z_n+\rho_n x +\rho_n h,t,\th ) |\leq C\min (1,t^{2s+\b}) $$
  and thus 
\be \label{eq:estime-ovF--e}
\|\cE^s_{{\cA}_{e,\ov K_n'},U_n} \| _{L^\infty(\R^N)}\leq C.
\ee
Moreover by \eqref{eq:assump-on-ti-l-Ko-inProp-n}, for all $x\in \R^N$, $r>0$ and $\th\in S^{N-1}$, we have 
$$
 |{\cA}_{o,\ov K_n'}(x,r,\th)|\leq C  r_n^{-\a}|h_n|^{-\b} \min ( \rho_n ,  \rho_n r)^{2s-1+\a } .
 $$
Since $|h_n| \leq 1$ and $\|g_n\|_{C^{0,1}(\R^N)}\leq 1$ (recalling \eqref{eq:delta-eo}),  the above estimate implies that 
\begin{align*}
|\cO^s_{{\cA}_{o,\ov K_n'},U_n}(x)|&\leq Cr_n^{-\a}|h_n|^{-\b} \int_0^\infty\min (1,t) \min ( \rho_n,  t)^{2s-1 +\a } t^{-1-2s }\, dt \\
&\leq  Cr_n^{-\a}|h_n|^{-\b} \int_0^{\rho_n}t^{2s+\a } t^{-1-2s}\, dt\\
&+Cr_n^{-\a}|h_n|^{-\b} \rho_n^{2s-1+\a }  \int_{\rho_n}^1  t^{-2s }\, dt  +C r_n^{-\a}|h_n|^{-\b} \rho_n^{2s-1+\a} \int_1^\infty   t^{-1-2s}\, dt .
\end{align*}
Using that $2s>1$ and recalling  that $\rho_n=r_n |h_n|$,  we then conclude that 
\be \label{eq:estime-ovF--o}
\|\cO^s_{{\cA}_{o,\ov K_n'},U_n} \| _{L^\infty(\R^N)}\leq C.
\ee
Because  $[f_n]_{C^\a(\R^N)}\leq 1$,  it is plain  that
\be \label{eq:estime-ovf_n-U_n}
\|\ov f_n\|_{L^\infty(\R^N)}\leq C.
\ee
Recalling \eqref{eq:def-F_n-Shaud}, it follow from \eqref{eq:estime-ovF--e},  \eqref{eq:estime-ovF--o} and \eqref{eq:estime-ovf_n-U_n} that 
 \be \label{eq:estime-F_n-U_n}
\|F_n\|_{L^\infty(\R^N)}\leq \frac{C}{\Theta_n(r_n)}.
\ee
In view of \eqref{eq:lim-v-n} and \eqref{eq:groht-w-n-Morrey-HI-nab},      for  every $h\in \R^N$,   we have that
\be \label{eq:w_n-to-h-h-0}
 w_n=   (v_n)_{h,0}\to  v_{h,0}    \qquad \textrm{ in  $L_s(\R^N)\cap H^s_{loc}(\R^N)$}. 
\ee
Letting $a_n(\th):={\cA}_{\ov K_n}(z_n,0,\th)$, we have that 
$$
| {\cA}_{\ov K_n}(x,r,\th)- a_n(\th) |\leq C \min( \rho_n^{\a} (|x|+r)^{\a} ,1)\qquad\textrm{ for all $x\in \R^N$,  $r>0$ and $\th \in S^{N-1}$,}
$$
so that 
\be \label{eq:ov-not-far-mu-a-n}
|\ov K_n(x,y)-\mu_{a_n}(x,y) |\leq C  \min( \rho_n^{\a} (|x|+|x-y|)^{\a} ,1) \mu_1(x,y) \qquad \textrm{ for all $x\not=y\in \R^N$.}
\ee
 Moreover   $a_n$ satisfies \eqref{eq:def-a-anisotropi} for all $n$. 
Therefore in view of Lemma \ref{lem:Ds-a-n--to-La}, Lemma \ref{lem:convergence-very-weak},   \eqref{eq:w_n-to-h-h-0}, \eqref{eq:ov-not-far-mu-a-n} and \eqref{eq:estime-F_n-U_n},    passing to the limit in \eqref{eq:eqw_n--ok}, we   deduce that 
\be \label{eq:eq-Lb-v-h0}
\cL_{\mu_b}  v_{h,0}=0 \qquad\textrm{ in $\R^N$ },
\ee
were $b$ is the weak-star limit of $a_n $ in $L^\infty(S^{N-1})$. Furthermore by \eqref{eq:groht-w-n-Morrey-HI}, 
$$
\|v_{h,0}\|_{L^\infty(B_R)}\leq C R^{2s-1+\a}, 
$$
Thanks to \eqref{eq:eq-Lb-v-h0} and since  $2s-1+\a<1$, we can apply Lemma \ref{lem:Liouville} to get a constant $c=c(h,\a,\b,N,s,\k)$ such that $v_{h,0}(x)=v(x+h)-v(x)= c$  for all  $x,h\in \R^N$. Hence, since $\n v(0)=0$, we find that $\n v(h)=0$ for all $h\in \R^N$.     This contradicts the first inequality in \eqref{eq:E1-nab00}. The proof of $(i)$ is thus finished.\\

The proof of $(ii)$ is similar to the one of $(i)$, we therefore give a sketch below,  emphasizing the main differences. Indeed,  following  the proof, we put
$$
\ov u_n (x)=\frac{1}{r_n^{2s+\a} \Theta_n(r_n) }\left\{g_n(r_n x)- \textbf{Q}_{r_n,g_n}(r_n x) \right\},
$$
with $\Theta_n(r)$ is  a nonincreasing  function as above, with $g_n$ replaced with $g_n-\textbf{Q}_{r,g_n}$.
From the definition of $\textbf{Q}_{r,g_n}$, the monotonicity of $\Theta_n$  and    the fact that $2s-1+\a>1$, we then get  $\|\n \ov u_n\|_{C^\b(B_R)}\leq C R^{2s-1+\a}$  for all $R\geq 1$. On the other hand,  there are $x_n\in B_1$ and $h_n\in B_1\setminus\{0\}$ such that  $|\n \ov  u_n(x_n+ h_n)-\n  \ov u_n(x_n)| |h_n|^{-\b}\geq \frac{1}{4}$. Similarly as above, we   define 
$$
 \ov v_n(x):= \frac{ \ov u_n(x_n+ |h_n| x)-  \ov u_n(x_n)- |h_n|\n \ov u_n(x_n)\cdot x}{  |h_n|^{1+\b} },
$$
so that  $\|\n \ov v_n\|_{L^\infty(B_R)}\leq C R^{2s-1+\a}$  for all $R\geq 1$. Moreover, 
$\ov v_n$ is bounded in $C^{1,\b}_{loc}(\R^N)$, so that  $\ov v_n\to \ov v$ in $C^{1}_{loc}(\R^N)$. In addition,  
\be \label{eq:to-contradict-ov-v}
\n\ov v(0)=0 \qquad \textrm{ and } \qquad |\n \ov v(e)|\geq \frac{1}{4},\quad\textrm{ for some $e\in S^{N-1}$.}
\ee
and $\|\ov v(\cdot +h)- \ov v\|_{L^\infty(B_R)}\leq C R^{2s-1+\a}$ for all $R\geq 1$.
Next, letting, $\ov w_n(x)= (\ov v_n)_{h,0}(x)=\ov v_n(x+h)-\ov v_n(x)$ we find that
$$
 \cL_{\ov K_n} \ov w_n - \frac{2\rho_n^2}{ r_n^{2s+\a} |h_n|^{1+\b}\Theta_n(r_n)} \sum_{i,j=1}^N T_{ij}^{r_n,g_n}  h_j   \cL_{\ov K_n} x_i =F_n \qquad\textrm{ in $B_{1/2r_n}$,}
$$
where   $F_n$ is given by \eqref{eq:def-F_n-Shaud}.  
 By \eqref{eq:Operator-splitting} and the fact that $\cE^s_{{\cA}_{e,\ov K_n},x_i}\equiv 0$ on $\R^N$, we then get 
\be \label{eq:eq-ovw_n--ok}
 \cL_{\ov K_n} \ov w_n =F_n+ \frac{2\rho_n^2}{ r_n^{2s+\a} |h_n|^{1+\b}\Theta_n(r_n)} \sum_{i,j=1}^N T_{ij}^{r_n,g_n}  h_j  \cO^s_{{\cA}_{o,\ov K_n},x_i} \qquad\textrm{ in $B_{1/2r_n}$.}
\ee
We start by estimating  $F_n$ defined in \eqref{eq:def-F_n-Shaud}.  Thanks to \eqref{eq:will-saveme},  we have 
\begin{align*}
&{\cA}_{o,\ov K_n'}(x,r,\th)= r_n^{-\a}|h_n|^{-\b} \rho_n\int_0^1D_x {\cA}_{o, K_n}(z_n+  \rho_n x+\varrho \rho_n h,\rho_n r,\th)h d\varrho\\
&=  r_n^{-\a}|h_n|^{-\b} \rho_n r\int_0^1 \left[D_r {\cA}_{ o, K_n}(z_n+  \rho_n x+  \rho_n h,\varrho \rho_n r,\th)- D_r {\cA}_{ o, K_n}(z_n+  \rho_n x,\varrho \rho_n r,\th) \right] d\varrho.
\end{align*}
Recall that $ K_n\in \ti \scrK^s_0(\k, 1+ 2s-2+\a, Q_\infty)$ with  $1>2s-2+\a>0$, and so by definition, 
$$
  \|D_r \cA_{K_n}\|_{ C^{2s+\a-2}(Q_\infty)\times L^\infty(S^{N-1} )}+ \|D_x \cA_{o,K_n}\|_{L^\infty_{2s+\a-2}(Q_\infty)\times L^\infty(S^{N-1} )}\leq C. 
$$ 
It follows that 
$$
| \cA_{o,\ov K_n'}(x,r,\th) |\leq C r_n^{-\a}|h_n|^{-\b} \min ( \rho_n r \rho_n^{2s+\a-2}, \rho_n (\rho_n r)^{2s+\a-2}). 
$$
From this and the fact that $|\d^og_n(y,t,\th)|\leq t$,  we deduce that 
\begin{align*}
|\cO^s_{{\cA}_{o,\ov K_n'},U_n}(x)|& \leq Cr_n^{-2s-\a}|h_n|^{-1-\b} \int_0^\infty\rho_n r \min ( \rho_n r  \rho_n^{2s+\a-2}, \rho_n (\rho_n r)^{2s+\a-2})r^{-1-2s }\, dr \\
&= Cr_n^{-2s-\a}|h_n|^{-1-\b}\rho_n^{2s} \int_0^\infty  \min ( t  \rho_n^{2s+\a-2}, \rho_n t^{2s+\a-2})t^{-2s }\, dt  \\
&\leq C r_n^{-2s-\a}|h_n|^{-1-\b}\rho_n^{2s} \int_0^{\rho_n}  \rho_n^{2s+\a-2}    t^{1-2s }\, dt+ C r_n^{-2s-\a}|h_n|^{-1-\b}\rho_n^{2s} \int_{\rho_n}^\infty \rho_n    t^{\a-2} \, dt    .
\end{align*}
Therefore, $\|\cO^s_{{\cA}_{o,\ov K_n'},U_n} \| _{L^\infty(\R^N)}\leq C$. By combining this with  \eqref{eq:estime-ovF--e} and   \eqref{eq:estime-ovf_n-U_n},  we get $$\|F_n\|_{L^\infty(\R^N)}\leq \frac{C}{\Theta_n(r_n)}.$$
Next, we estimate the last term in the right hand in \eqref{eq:eq-ovw_n--ok}.   From the first inequality in \eqref{eq:assump-on-ti-l-Ko-inProp-n}, we get  
\be  \label{eq:estim-cOs-Schauder}
 |\cO^s_{{\cA}_{o,\ov K_n},x_i}  (x)|\leq C \int_0^\infty \min (\rho_n r, 1) r^{-2s}\, dr\leq C \rho_n^{2s-1}.
\ee
  Since $g_n(0)=|\n g_n(0)|=0$ and $\|g_n\|_{C^{2s+\b}(\R^N)}\leq 1$, we then have that $|  T_{ij}^{r_n,g_n} |\leq C r_n^{-2+2s+\b}$.  Hence, since $2s+\b\geq 1+\a$, by \eqref{eq:estim-cOs-Schauder}, the last term in the right hand in \eqref{eq:eq-ovw_n--ok} is   bounded by $\frac{C}{\Theta_n(r_n)}$. Since $\frac{1}{\Theta_n(r_n)}$ tends to zero as $n\to\infty$,      thanks to  Lemma \ref{lem:convergence-very-weak}, we can pass to the limit in  \eqref{eq:eq-ovw_n--ok}, to get   $\cL_{\mu_b} \ov v_{h,0}=0 $ in $\R^N$.  Hence, since $2s-1+\a<2s$,  Lemma \ref{lem:Liouville}  implies that, there exist $\ov c\in \R$ and $\ov d\in \R^N$, only depending on $h,\a,\b,N,s$ and $\k$,  such that $ \ov v_{h,0}(x)= \ov v(x+h)-\ov v(x)=\ov d\cdot x+\ov c $  for all  $x,h\in \R^N$. Now, since $\n \ov v(0)=0$, we find   $\n \ov v\equiv 0$ on $\R^N$. This  contradicts \eqref{eq:to-contradict-ov-v} and the proof of $(ii)$ is finished.\\

The proof of  $(iii)$  follows (verbatim) the same argument as the one of $(i)$. The  fact that $K_n\in \ti \scrK_{\a+\e}^{1/2}(\k,\a,Q_\infty)$ is   only needed to deduce  the uniform bound  $\|\cO^{1/2}_{{\cA}_{o,\ov K_n'},U_n}\|_{L^\infty(\R^N)}\leq C$ from the uniform estimate $
 |{\cA}_{o,\ov K_n'}(x,r,\th)|\leq C  r_n^{-\a}|h_n|^{-\a/2} \min ( \rho_n ,  \rho_n r)^{ \a+\e } .
 $\\

The proof of $(iv)$ does not differ much from the one of $(i)$. We skip the details.
\end{proof}
As a consequence of the previous result, we have  the following 
   \begin{theorem}\label{th:Schauder-000}
Let $s\in (0,1)$, $N\geq 1$, $\k >0$ and $\a\in (0,1)$.   Let $K\in \ti \scrK_{\t}^s(\k, \a, Q_\infty)$ with $\t=\min(\a+(2s-1)_+,1)$.
  Let $u\in H^s(B_2)\cap  C^\a(\R^N)$ and $f\in C^\a(B_2)$ satisfy
 $$
 \cL_K u= f\qquad \textrm{in $B_2$}.
 $$
\begin{itemize}
\item[$(i)$]  If $2>2s+\a>1$, $2s+\a\not=2$ and $2s\not=1 $,    then  there exists $C=C(N,s,\k,\a)>0$ such that
$$
\|u\|_{C^{2s+\a}(B_1)}\leq C(\|u\|_{C^\a(\R^N)}+ \|f\|_{C^\a(B_2)}).
$$
\item[$(ii)$]  If $2s+\a>2$ and  $K\in \ti \scrK_{0}^s(\k, 2s-1+\a, Q_\infty)$   then  there exists $C=C(N,s,\k,\a)>0$ such that
$$
\|u\|_{C^{2s+\a}(B_1)}\leq C(\|u\|_{C^\a(\R^N)}+ \|f\|_{C^\a(B_2)}).
$$
 \item[$(iii)$]   If $2s=1$ and  $K\in \ti  \scrK_{\a+\e}^s(\k, \a, Q_\infty)$, for some $\e>0$,  then  there exists $C=C(N,s,\k,\a,\e)>0$ such that
$$
\|u\|_{C^{1+\a}(B_1)}\leq C(\|u\|_{C^\a(\R^N)}+ \|f\|_{C^\a(B_2)}).
$$
\item[$(iv)$]     If $2s+\a <1$,    then  there exists $C=C(N,s,\k,\a)>0$ such that
$$
\|u\|_{C^{2s+\a}(B_1)}\leq C(\|u\|_{C^\a(\R^N)}+ \|f\|_{C^\a(B_2)}).
$$
\end{itemize}

 \end{theorem}

\begin{proof}
From \cite{Fall-Reg},  there exists   $\b\in (0,\a)$, only depending on $N,s,\k$ and $\a$ such that, if $2s+\b\not \in \N$,
\be \label{eq:u-2s-b-proof}
 \|u\|_{C^{2s+\b}(B_1)}\leq C (\|u\|_{C^{\a}(\R^N)}+   \|f\|_{C^{\a}{(B_2)}} ),
\ee
for some $C=C(N,s,\a,\b,\k)$.\\
\noindent
\textbf{Case 1: $1<2s+\a<2$.}
For $z\in B_{1}$, we define
\be \label{eq:defin-g_z}
 g_z(x)= u(x+z)-u(z)- \vp_4(x) \n u(z)\cdot x 
\ee
which satisfies $g_z(0)=|\n g_z(0)|=0$.
 We introduce the cut-off function $\vp_4$ only because  the functions $x\mapsto x_i$, for $i=1,\dots,N$, do not belong to ${L_s(\R^N)}$ when $2s=1$.
 By construction and   \eqref{eq:u-2s-b-proof},  we have
\be  \label{eq:estim-d-C2s-e_0}
 \|g_z\|_{C^{2s+\b}(\R^N)}\leq C \|u\|_{C^{2s+\b}(B_1)}\leq   C (\|u\|_{C^{\a}(\R^N)}+   \|f\|_{C^{\a}{(B_2)}} ).
\ee
 In addition, by Lemma \ref{lem:estim-G_Ku}$(iii)$,
 $$
 \cL_{K_z}  g_z = f(\cdot + z)  -\cL_{K_z} U \qquad\textrm{ in $B_{1}$},
$$
where $K_z(x,y)=K(x+z,y+z)$,  $U(x):= \vp_4(x) \n u(z)\cdot x$. From Lemma \ref{lem:estim-G_Ku}$(iii)$, we get
 $$
 \cL_{K_z}  (\vp_{1/2}g_z) =\ti  f -\cL_{K_z} U \qquad\textrm{ in $B_{1/8}$},
$$
 for some function $\ti f$, satisfying 
\be \label{eq:ti-Calpba}
\|\ti f\|_{C^{\a}(\R^N)}\leq C  (\|u\|_{C^{\a}(\R^N)}+   \|f\|_{C^{\a}{(B_2)}} ).
\ee
Using \eqref{eq:Operator-splitting}, we then have that
\be \label{eq:clK_zcatg-z}
 \cL_{K_z} (\vp_{1/2}g_z) =F \qquad\textrm{ in $B_{1/8}$,}
\ee
  where $F:=\ti f  -\cE^s_{{\cA}_{e,K}, U} -\cO^s_{{\cA}_{o,K}, U}$.
Because $K_z\in   \ti \scrK_{\a+2s-1}^s(\k,\a,Q_\infty)$ for $2s> 1$ and  $K_z\in \ti \scrK_{\a+\e}^s(\k,\a,Q_\infty)$ for $2s=1$ and   since $U\in C^2_c(\R^N)$, we   deduce from Lemma \ref{lem:2sp-alph-estim-F_Keo} that
$$
\|\cE^s_{{\cA}_{e,K_z}, U}\|_{C^\a(\R^N)}+ \|\cO^s_{{\cA}_{o,K_z}, U}\|_{C^\a(\R^N)} \leq C(N,s,\a) \|\n u\|_{L^\infty(B_1)}.
$$
This with \eqref{eq:ti-Calpba} and \eqref{eq:u-2s-b-proof} imply that 
\be\label{eq:estim-F-a-Shaud}
 \|F\|_{C^{\a}(\R^N)}\leq C  (\|u\|_{C^{\a}(\R^N)}+   \|f\|_{C^{\a}{(B_2)}} ).
\ee
Thanks to \eqref{eq:clK_zcatg-z}, applying Proposition  \ref{prop:bound-wth-corrector-nab}$(i)$ and $(iii)$ and using \eqref{eq:estim-F-a-Shaud}, provided $1<2s+\a<2$, we get
$$
\|\n (\vp_{1/2}g_z)\|_{L^\infty(B_r)}\leq C r^{2s-1+\a}\left( \|\vp_{1/2}g_z\|_{C^{2s+\b}(\R^N)}+  \|u\|_{C^{\a}(\R^N)}+   \|f\|_{C^{\a}{(B_2)}}\right).
$$
As a consequence,   by   \eqref{eq:estim-d-C2s-e_0},  for all $r\in (0,1/2)$ and all $z\in B_{1}$, 
$$
\|\n u- \n u(z)\|_{L^\infty(B_r(z))}\leq C r^{2s-1+\a}\left( \|u\|_{C^{\a}(\R^N)}+   \|f\|_{C^{\a}{(B_2)}} \right),
$$
for some $C=C(N,s,\a,\k)$.
We then conclude that 
$$
\|\n u\|_{C^{2s-1+\a}(B_{1/8})}\leq C \left( \|u\|_{C^{\a}(\R^N)}+   \|f\|_{C^{\a}{(B_2)}}\right).
$$
Therefore $(i)$ and $(iii)$   follow from    a covering and scaling argument.\\
\noindent
\textbf{Case 2: $2s+\a>2$.} We know from \textbf{Case 1} that  for all  $\b\in (0, 2-2s)$,    
\be \label{eq:u-2s-b-proof00}
 \|u\|_{C^{2s+\b}(B_1)}\leq C (\|u\|_{C^{\a}(\R^N)}+   \|f\|_{C^{\a}{(B_2)}} ).
\ee
We then consider the function $g_z$  defined in \eqref{eq:defin-g_z}. Hence,  thanks to \eqref{eq:clK_zcatg-z},  by Proposition  \ref{prop:bound-wth-corrector-nab}$(ii)$, \eqref{eq:u-2s-b-proof00} and  \eqref{eq:estim-F-a-Shaud}  we get   
\be \label{eq:eq-to-iter-C2}
\|\n (\vp_{1/2}g_z)-\n \textbf{Q}_{r, \vp_{1/2}g_z}\|_{L^\infty(B_r)}\leq C r^{2s-1+\a} (\|u\|_{C^{\a}(\R^N)}+   \|f\|_{C^{\a}{(B_2)}} ),
\ee
provided  $2s+\b\geq 1+\a$. 
 In view of \eqref{eq:eq-to-iter-C2}, we can use an iteration argument to obtain, for all $r\in (0,1/2)$, 
$$
\|\n u-\n u(z)-M_z (\cdot-z)\|_{L^\infty(B_r(z))}\leq C r^{2s-1+\a} (\|u\|_{C^{\a}(\R^N)}+   \|f\|_{C^{\a}{(B_2)}} ).
$$
for some $(N\times N)$-matrix $M_z$ satisfying $|M_z|\leq C(\|u\|_{C^{\a}(\R^N)}+   \|f\|_{C^{\a}{(B_2)}} )$. Since $2s-1+\a>1$, we deduce that $M_z = D^2 u(z) $.   
Using now an  extension theorem, see e.g.  \cite{Stein}[Page 177],  we conclude that 
$$
\|\n u\|_{C^{1,2s-2+\a}(B_{1/8})}\leq C \left( \|u\|_{C^{\a}(\R^N)}+   \|f\|_{C^{\a}{(B_2)}}\right).
$$
We thus get  $(iii)$ after a covering and scaling argument.\\
\noindent
\textbf{Case 3: $2s+\a<1$.} Here,  we argue as in \textbf{Case 1},   by applying Proposition  \ref{prop:bound-wth-corrector-nab}$(iii)$ to  the function $  g_z(x)= \vp_{1/2}(x)\{u(x+z)-u(z)\} $. We skip the details. 
\end{proof}


By an induction argument, we have the following result.
\begin{theorem}\label{th:C-2s-k-al-reg-nonreg-theta}
Let $N\geq 1$,   $s\in (0,1)$ and $\a\in (0,1)$. Let  $\k>0$,  $k\in\N$ and   $K\in\ti \scrK_{\t} ^s (\k,k+ \a , Q_\infty)$, with $\t=\min(1,\a+(2s-1)_+)$.
Let  $u\in H^s(B_2)\cap C^{k+\a}(\R^N)$   and $f\in C^{k+\a}(\R^N)$  
  such that  
$$
 \cL_K  u =f \qquad\textrm{ in $B_2$}. 
$$
\begin{itemize}
\item[$(i)$] If $2s\not=1$ and  $2s+\a<2$, then there exists $C=C(N,s,k,\k,\a)$  such that 
 $$
 \|u\|_{C^{k+2s+\a}(B_1)}\leq C (\|u\|_{C^{k+\a}(\R^N)}+   \|f\|_{C^{k+\a}(\R^N)} ).
 $$
 \item[$(ii)$]  If $2s=1$ and $K\in\ti \scrK_{\a+\e}^{1/2}(\k,k+ \a , Q_\infty)$, for some $\e>0$,   then   there exists $C=C(N,k,\k,\a,\e)$  such that 
 $$
 \|u\|_{C^{k+1+\a}(B_1)}\leq C (\|u\|_{C^{k+\a}(\R^N)}+   \|f\|_{C^{k+\a}(\R^N)} ).
 $$
  \item[$(iii)$]  If $2s+\a>2$ and  $K\in \ti \scrK_{0}^s(\k, k+2s-1+\a, Q_\infty)$,   then  there exists $C=C(N,s,\k,\a)>0$ such that
$$
\|u\|_{C^{k+2s+\a}(B_1)}\leq C(\|u\|_{C^\a(\R^N)}+ \|f\|_{C^\a(B_2)}).
$$
\end{itemize}

\end{theorem}
\begin{proof}
We will prove $(i)$ and $(ii)$, since the one of $(iii)$ will follow the same arguments.
The case $k=0$, that $u\in C^{2s+\a}(B_1)$, is proved  in Thoerem \ref{th:Schauder-000}. We prove the statement first for $k=1$. 
By Lemma \ref{lem:estim-G_Ku}$(ii)$, we have that  $u^1:=\vp_{1} u\in C^{2s+\a}(\R^N)\cap C^{1,\a}(\R^N)$ and 
$$
 \cL_K  u^1=f^1 \qquad\textrm{ in $B_{2^{-2}}$,}
$$
for some function $f^1$ satisfying
\be\label{eq:estim--nab-f--1-C-a}
\| f^1\|_{C^{1+\a}(B_{2^{-2}})}\leq C (\|u\|_{C^{1+\a}(\R^N)}+ \| f\|_{C^{1+\a}(\R^N)}) .
\ee
Let $h \in B_{2^{-4}}$, with $h\not=0$. Then (recalling \eqref{eq:def-f-h-alph})
$$
 \cL_K   u^1_{h,1}-\cL_{K_{h,1}} u^1(\cdot+h) =  f_{h,1}^1 \qquad\textrm{ in $B_{2^{-4}}$}. 
$$
We note that $K_{h,1} $ satisfies \eqref{eq:K'-Kernel-satisf}, with $\a'=0$. 
By Corollary \ref{cor:Holder-cont-coeff}, we obtain $u_{h,1}^1$ is uniformly bounded in $C^\s(B_{2^{-5}})$, for some $\s\in (s,2s)$. Therefore $\n u^1\in C^{0,\s}(B_{2^{-5}}) \subset H^s( B_{2^{-5}})$. Using \eqref{eq:Operator-splitting}, we then have that
$$
 \cL_K   u^1_{h,1}-   \cE^s_{{\cA}_{e, K_{h,1}}, u^1_h}- \cO^s_{{\cA}_{o,K_{h,1}}, u^1_h}=  f_{h,1}^1 \qquad\textrm{ in $B_{2^{-4}}$}. 
$$
Letting $h\to 0$, we see that, for all $i_1\in\{1,\dots,N\}$, 
$$
 \cL_K  \de_{i_1} u^1=  \cE^s_{\de_{i_1}{\cA}_{e,K} , u^1}+ \cO^s_{\de_{i_1}{\cA}_{o,K} , u^1} +  \de_{i_1} f^1=:f^1_{i_1} \qquad\textrm{ in $B_{2^{-5}}$}. 
$$
By Lemma \ref{lem:2sp-alph-estim-F_Keo} and \eqref{eq:estim--nab-f--1-C-a}, if $2s\not=1$, the right hand-side in the above display belongs to $C^\a(\R^N)$  and satisfies
$$
\|f^1_{i_1} \|_{C^\a(B_{2^{-2}})}\leq C (\|u\|_{C^{1+\a}(\R^N)}+ \|f\|_{C^\a(B_2)}) .
$$
On the other hand if $2s=1$, then  Lemma \ref{lem:2sp-alph-estim-F_Keo} and \eqref{eq:estim--nab-f--1-C-a} yield
$$
\|f^1_{i_1} \|_{C^{\a-\d}(B_{2^{-2}})}\leq C (\|u\|_{C^{1+\a}(\R^N)}+ \|f\|_{C^\a(B_2)}) ,
$$
for all $\d\in (0,1-\a)$.
 It follows from Theorem \ref{th:Schauder-000} that if $2s\not=1$, then 
$$
\|\de_{i_1} u^1\|_{C^{2s+\a}(B_{2^{-6}})} \leq C (\|u\|_{C^{1+\a}(\R^N)}+ \|f\|_{C^\a(B_2)}) 
$$ 
and for $2s=1$, 
$$
\|\de_{i_1} u^1\|_{C^{2s+\a-\d}(B_{2^{-6}})} \leq C (\|u\|_{C^{1+\a}(\R^N)}+ \|f\|_{C^\a(B_2)}) .
$$ 
We now remove the $\d$ in the above estimate (for $2s=1$). Indeed,   we define $\ov u^1:=\vp_{2^{-7}} u^1 \in C^{2s+\a+(1-\d)}(\R^N)$ which, by  Lemma \ref{lem:estim-G_Ku}$(iii)$,  satisfies
$$
\cL_K \ov u^1=\ov f^1   \qquad\textrm{ in $B_{2^{-8}}$},
$$
with 
$$
\|\ov f^1\|_{C^\a(\R^N)}\leq C ( \|u\|_{C^\a(\R^N)} + \|f\|_{C^\a(\R^N)}  ).
$$
Therefore proceeding as above, we have 
\be \label{eq:cL_K-dui1ovu}
 \cL_K  \de_{i_1} \ov u^1=  \cE^s_{\de_{i_1}{\cA}_{e,K} , \ov u^1}+ \cO^s_{\de_{i_1}{\cA}_{o,K} , \ov u^1} +  \de_{i_1} \ov f^1=:\ov f^1_{i_1} \qquad\textrm{ in $B_{2^{-9}}$}. 
\ee
Since $K\in \ti \scrK_{\a+\e} ^{1/2}(\k,\a,Q_\infty)$ and $\ov u^1:=\vp_{2^{-7}} u^1 \in C^{2s+\a+1-\d}(\R^N)$,  Lemma \ref{lem:2sp-alph-estim-F_Keo} yields
$$
\| \cE^s_{\de_{i_1}{\cA}_{e,K} , \ov u^1}  \|_{C^{\a}(\R^N)}+\| \cO^s _{\de_{i_1}{\cA}_{o,K} , \ov u^1} \|_{C^{\a}(\R^N)}\leq C \|\ov u^1 \|_{C^{2s+\a+(1-\d)}(\R^N)}.
$$
Applying Theorem \ref{th:Schauder-000} to the equation \eqref{eq:cL_K-dui1ovu},  we then get
$$
\|\de_{i_1} \ov u^1\|_{C^{1+\a}(B_{2^{-10}})} \leq C (\| \ov u^1 \|_{C^{1+\a}(\R^N)}+ \|\ov f^1_{i_1}\|_{C^\a(B_2)} ) .
$$
The theorem is thus proved for $k=2.$\\

Let $k> 2$.
We now prove by induction that for every $(i_1,i_2,\dots, i_k)\in \{1,\dots,N \}^k$ there exist a constant $r_k$, only depending on $k$,  and a constant $ C_k>0$, only depending on $N,s,\k,\a$ and $k$,  such that
\be \label{eq:u-C2s-k-a-induction}
\|\de_{i_1i_2\dots i_k}^k u\|_{C^{2s+\a+k}(B_{r_k})} \leq C_k  (\|u\|_{C^{k+\a}(\R^N)}+ \|f\|_{C^{k+\a}(B_2)}).
\ee
We assume, as induction hypothesis that,   the result is true up order $k-1$. That is,  there exist $r_{k-1}, C_{k-1}>0$, as above,  such that
\be\label{eq:induc-hyp} 
\|  u\|_{C^{2s+k-1+\a}(B_{r_{k-1}})} \leq C_{k-1}(\|u\|_{C^{k-1+\a}(\R^N)}+ \|f\|_{C^{k-1+\a}(B_2)}).
\ee
 We then consider 
$$
u^k:=\vp_{r_{k-1}/2} u\in C^{2s+\a+k-1}(\R^N)\cap  C^{k,\a}(\R^N). 
$$
 By Lemma \ref{lem:estim-G_Ku}$(ii)$, we then have that 
$$
  \cL_K   u^k  =  f^k \qquad\textrm{ in $B_{r_{k-1}/4}$},
$$
for some function 
\be\label{eq:f-k-hold-k-a}
\|f^k \|_{C^{k+\a}(B_{r_{k-1}/4})}\leq C_k' (\|u\|_{C^{k+\a}(\R^N)}+ \|f\|_{C^{k+\a}(B_2)}) ,
\ee
where, unless otherwise stated,  $C_k'$ denotes a positive   constant,    only depending on $N,s,\k,\a$ and $k$.  
Proceeding as above, we can differentiate the equation $k$ times to deduce that  for all $(i_1,i_2,\dots, i_k)\in \{1,\dots,N \}^k$, 
\be \label{eq:cL-k-u-k-induction}
  \cL_K   \de_{i_1i_2\dots i_k}^k u^k  =  g^k+   \de_{i_1i_2\dots i_k}^k f^k \qquad\textrm{ in $B_{r_{k}'}$},
\ee
for constant $r_k'<r_k$, only depending on $r_k$ and $k$,  and for some function $g^k:=\sum_{j=1}^m c_j^e \cE^s_{a_j^e, v_j}+ \sum_{j=1}^m c_j^o \cO^s_{a_j^o, w_j}$ where $c_j^e,c_j^o$ are real numbers, $a_j^e$, $a_j^o$, $v_j$ and $w_j$ are respectively given by the partial derivatives in $x$ of ${\cA}_{e,K}$, ${\cA}_{o,K}$ up to order $k$ and  $v_j$ together with $w_j$  are given by partial derivatives of $u^k$ up to order $k-1$. Therefore, provided $2s\not=1$, by Lemma \ref{lem:2sp-alph-estim-F_Keo}, 
$$
\|g^k\|_{C^\a(\R^N)}\leq C_k'\|u^k\|_{C^{2s+\a+k-1}(\R^N)}\leq C ( \|u\|_{C^{k-1+\a}(\R^N)}+\|f^k\|_{C^{k,\a}(\R^N)}).
$$
 Now Theorem \ref{th:Schauder-000} implies that, for $2s\not=1$, 
$$
\| \de_{i_1i_2\dots i_k}^k u^k\|_{C^{2s+\a}(B_{r_k'/2})} \leq C_k' ( \|\de_{i_1i_2\dots i_{k-1}}^{k}  u^k\|_{C^{ \a}(\R^N)}+\|g^k\|_{C^{\a}(\R^N)}+\|f^k\|_{C^{k,\a}(\R^N)}).
$$
By \eqref{eq:f-k-hold-k-a},  we get \eqref{eq:u-C2s-k-a-induction} in the case $2s\not=1$. Therefore $(i)$   follows by a covering and scaling argument.\\

Now when $2s=1$, then we can argue similarly as above, noticing that, under the induction hypothesis \eqref{eq:induc-hyp},  by Lemma \ref{lem:2sp-alph-estim-F_Keo}, the function $g^k$   in the right hand side of   \eqref{eq:cL-k-u-k-induction}  belongs to $C^{\a-\d}(\R^N) $, for all $\d\in (0,\a)$. Hence Theorem \ref{th:Schauder-000} implies that, for all $\d\in (0,\a)$,
$$
\| \de_{i_1i_2\dots i_k}^k u^k\|_{C^{2s+\a-\d}(B_{r_k'/2})} \leq C_k'' ( \|\de_{i_1i_2\dots i_{k-1}}^{k}  u^k\|_{C^{ \a}(\R^N)}+\|g^k\|_{C^{\a-\d}(\R^N)}+\|f^k\|_{C^{k,\a}(\R^N)}),
$$
where, unless otherwise stated,  $C_k''$ denotes a positive   constant,    only depending on $N,\k,\a,\d,\e$ and $k$.
To remove the parameter $\d$, we consider 
$$
\ov u^k:=\vp_{r_{k}'/4} u\in C^{k+1+\a-\d}(\R^N) ,
$$
which satisfies
\be\label{eq:estim-ov-u-k}
\|\ov u^k\|_{ C^{k +\a+ (1-\d)}(\R^N)} \leq C_k'' ( \|u\|_{C^{k+\a}(\R^N)}+  \|f\|_{C^{k+\a}(\R^N)}).
\ee
 By Lemma \ref{lem:estim-G_Ku}$(ii)$, we then have that 
$$
  \cL_K   \ov u^k  =  \ov f^k \qquad\textrm{ in $B_{r_{k}'/8}$},
$$
for some function 
\be\label{eq:ov-f-k-hold-k-a}
\|\ov f^k \|_{C^{k+\a}(B_{r_k'/8})}\leq C_k'' (\|u\|_{C^{k+\a}(\R^N)}+ \|f\|_{C^{k+\a}(\R^N)}) .
\ee
As above, we then differentiate the equation $k$ times to deduce that  for all $(i_1,i_2,\dots, i_k)\in \{1,\dots,N \}^k$, 
\be \label{eq:cL-k-u-k-induction}
  \cL_K   \de_{i_1i_2\dots i_k}^k \ov u^k  =  \ov g^k+   \de_{i_1i_2\dots i_k}^k \ov f^k \qquad\textrm{ in $B_{r_{k}''}$},
\ee
for constant $r_k''$, only depending on $k$,  and for some function $\ov g^k(x):=\sum_{j=1}^m c_j^e \cE^{1/2}_{a_j^e, v_j}+ \sum_{j=1}^m c_j^o \cO^{1/2}_{a_j^o, w_j}$ where $c_j^e,c_j^o$ are real numbers and $a_j^e$, $a_j^o$  (resp. $v_j$ and $w_j$) are respectiveley given by the partial derivatives in $x$-variable of ${\cA}_{e,K}$, ${\cA}_{o,K}$ up to order $k$    (resp. $v_j$ together with $w_j$  are given by partial derivatives of $\ov u^k$ up to order $k-1$). Therefore by Lemma \ref{lem:2sp-alph-estim-F_Keo}, \eqref{eq:estim-ov-u-k},  and since $K\in \ti  \scrK_{\a+\e}^{1/2}(\k,k+\a,\R^N)$, we obtain
\be \label{eq:ov-g-k-hold-k-a}
\|\ov g^k\|_{C^\a(\R^N)}\leq C_k''\|\ov u^k\|_{C^{1+\a+k-\d}(\R^N)}\leq C_k''  ( \|u\|_{C^{k+\a}(\R^N)}+  \|f\|_{C^{k+\a}(\R^N)}) .
\ee
Applying Theorem \ref{th:Schauder-000}, we conclude that 
$$
\| \de_{i_1i_2\dots i_k}^k \ov u^k\|_{C^{1+\a}(B_{r_k''/2})} \leq C_k'' ( \|\de_{i_1i_2\dots i_{k-1}}^{k-1}  \ov u^k\|_{C^{ \a}(\R^N)}+\| \ov g^k\|_{C^ \a(\R^N)}+\| \ov f^k\|_{C^{k,\a}(\R^N)}).
$$
Hence, since $\ov u^k=u$ on $B_{r_k''/2}$,  by \eqref{eq:estim-ov-u-k},  \eqref{eq:ov-f-k-hold-k-a} and \eqref {eq:ov-g-k-hold-k-a} with then obtain  \eqref{eq:u-C2s-k-a-induction}. Now $(ii)$ follows by scaling and covering.
\end{proof}

 \section{Proof of the main results}\label{s:proofMainResults}
We start this section with  the following result  which shows how to pass from a nonlocal equation with kernels  in $\ti \scrK_{\t}^s(\k,m+\a,Q_{\d})$ to a nonlocal equation driven by kernels in $\ti \scrK^s_\t(\k,m+\a,Q_\infty ) $. 
 \begin{lemma}\label{lem:cut-reg-kernel}
 Let $K\in  \ti \scrK_{\t}^s(\k,m+\a,Q_{4R})$, for some  $\a\in [0,1)$, $\t\in [0,1]$, $m\in \N$ and $R>0$. Let $v\in H^s(B_{4R})\cap {L_s(\R^N)} $ and $f\in L^1_{loc}(B_{4R})$ satisfy
 $$
 \cL_K v= f\qquad \textrm{in $B_{4R}$.}
 $$
 Let 
 $$
 \ov K(x,y)=\vp_{2R}(x)\vp_{2R}(y)K(x,y)+(2-\vp_{R}(x)-\vp_{R}(y)) \mu_1(x,y).
$$
Then
\be \label{eq:new-eq-v}
 \cL_{\ov  K} v+ \ov Vv= f+\ov  f\qquad \textrm{in $B_{R/4}$,}
\ee
 where, for $x\in B_{R/4}$, 
 $$
 \ov V(x)=G_{1,K,2R}(x)-G_{1,\mu_1,R}(x), \qquad   \ov  f(x)=G_{v,K,2R}(x)-G_{v,\mu_1,R}(x),
 $$
and $G_{v,K,\rho }$ is given by \eqref{eq:def-GKvR}. In particular, $\ov K\in \ti \scrK_{\t}^s(\ov \k,m+\a,Q_{\infty})$, for some constant $\ov \k=\ov  \k(\k,\a,m,R,\t,s,N)$.
 \end{lemma}
 \begin{proof}
The proof of \eqref{eq:new-eq-v} is elementary, and we skip it. Next, we observe that 
$$
{\cA}_{\ov K}(x,r,\th)=\vp_{2R}(x)\vp_{2R}(x+r\th ){\cA}_{ K}(x,r,\th)+ (2-\vp_{R}(x)-\vp_{R}(x+r\th)).
$$
Recalling the definition of the cut-off function $\vp_R$ in the beginning of Section \ref{s:NotPrem}, we easily deduce that  
$$
\min( \k ,1 )  \leq {\cA}_{\ov K}(x,r,\th)\leq \max (1/k, 4) \qquad\textrm{ for all $x\in \R^N$, $r\geq 0$, $\th\in S^{N-1}$.}
$$
This in particular implies that  $\ov K$ satisfies \eqref{eq:K-Kernel-satisf}.
Moreover, it is also not difficult to check that 
$$
\| {\cA}_{\ov K}\|_{C^{m,\a}(Q_\infty) \times L^\infty(S^{N-1} )}+ \|{\cA}_{o,\ov K} \|_{ \cC^{m}_\t(Q_\infty) \times L^\infty(S^{N-1} )}\leq C( \k,m,\a,\t,R).
$$
%
 \end{proof} 

 \begin{proof}[Proof of Theorem \ref{th:main-th1} ]
 As mentioned in the first section, the case $2s\leq 1$ was already proven in \cite{Fall-Reg}. Now the  case $2s>1$ follows from Theorem  \ref{th:abs-res-Propo},  Lemma \ref{lem:cut-reg-kernel}, Lemma \ref{lem:estim-G_Ku} and the fact that $L^p(\R^N) \hookrightarrow \cM_{N/p}$.
 \end{proof}
 %
 %
 
 %
   \begin{proof}[Proof of Theorem \ref{th:Schauder-0}]
First applying  Theorem \ref{th:Schauder-000} and using  Lemma \ref{lem:cut-reg-kernel} together with Lemma \ref{lem:estim-G_Ku}, we get the   estimates.  
\end{proof}
 \begin{proof}[Proof of Theorem \ref{th:main-th10} ]
 It follows from  Theorem \ref{th:main-th1}.
 \end{proof}
   \begin{proof}[Proof of Theorem \ref{th:Schauder-0-intro}  ]
By Lemma \ref{lem:cut-reg-kernel}, we have  that 
\be \label{eq:new-eq-for-u}
 \cL_{\ov  K}  u =f+\ov f-\ov Vu \qquad\textrm{ in $B_{1/2}$},
\ee
with $\ov  K\in  \ti \scrK_{\t_s}^s(\ov \k,m+\a+(2s-1)_+, Q_\infty)$ with $\t_s:= \a+(2s-1)_+$ for $2s+\a<2$ and  $\t_s=0$ for $2s+\a>2$.
In addition, by  Lemma \ref{lem:estim-G_Ku}$(iii)$, we have
\be\label{eq:estimV-fin}
\|\ov V\|_{C^{m+\a}(B_{1/2})}\leq C 
\ee
and 
\be\label{eq:estimovf-fin}
\|\ov f\|_{C^{m+\a}(B_{1/2})}\leq C \|u\|_{{L_s(\R^N)}   }.
\ee
We consider first the case $2s\not=1$.
 Since $u$ satisfies \eqref{eq:new-eq-for-u}, applying Theorem \ref{th:C-2s-k-al-reg-nonreg-theta} and using  Lemma \ref{lem:estim-G_Ku}$(iii)$, we get 
  $$
 \|\vp _{1/2}u\|_{C^{2s+m+\a}(B_{2-{4}})}\leq C ( \|\vp _{1/2}u\|_{C^{m+\a}(\R^N)}+ \| u\|_{L^{\infty}(\R^N)}  + \|F\|_{C^{m+\a}(B_{1/2})} ),
 $$
 where $F:=f+\ov f-\ov Vu$.
Consequently, by \eqref{eq:estimV-fin} and \eqref{eq:estimovf-fin}
   $$
 \| u\|_{C^{2s+m+\a}(B_{2^{-4}})}\leq C ( \| u\|_{C^{m+\a}(B_1)}+ \| u\|_{L^{\infty}(\R^N)}  + \|f\|_{C^{m+\a}(B_2)} ).
 $$
 Using now adimentional H\"older norms and interpolation (see e.g. \cite{GT, Barrios}), we can absorb the $C^{m+\a}(B_1)$-norm of $u$ to deduce that 
    $$
 \| u\|_{C^{2s+m+\a}(B_{2^{-5}})}\leq C ( \| u\|_{L^{\infty}(\R^N)}  + \|f\|_{C^{m+\a}(B_2)} ).
 $$
If now $2s=1$, then since $\ov  K\in  \ti \scrK_{\a }^{1/2}(\ov \k,m+\a, Q_\infty)$ and  in view of Theorem \ref{th:C-2s-k-al-reg-nonreg-theta}, the same arguments as above yield
    $$
 \| u\|_{C^{1+m+\a-\e}(B_{2^{-5}})}\leq C ( \| u\|_{L^{\infty}(\R^N)}  + \|f\|_{C^{m+\a-\e}(B_2)} ),
 $$
 for all $\e\in (0,\a)$. Now by scaling and covering, we get the result.
 \end{proof}

\subsection{Proof of Theorem \ref{eq:thm-nmc-reg}}
The following fundamental lemma allows, in particular, to consider truncation of the nonlocal mean curvature kernel   $1_{B_r}(x) 1_{B_r}(y)\cK_u(x,y)$ without any assumption on $u$ in the exterior of $B_r$.
\begin{lemma}\label{lem:Lem-Gamma-u-nmc}
Let $u:\R^N\to \R$ be a measurable function  and  $\G^{u,R}: B_{R/2}\to \R$  be  given by
\begin{align}\label{eq:def-g-nmc}
\G^{u,R}(x)
&:= \int_{\R^N} (1-   1_{B_{R}}(y)) \frac{\cF_s(p_u({x},{y}))- \cF_s(-p_u({x},{y})) }{|{x}-{y}|^{N+2s-1}} \, dy.
\end{align}
If $u\in C^{k,\a}(B_{R/2})$, for $k\geq 1$ and $\a\in [0,1]$,  then, there exists a constant $C=C(N,s,k,\a,R)$ such that 
\be \label{eq:estimd-Gu-Holder-NMC}
\|\G^{u,R}\|_{C^{k,\a}(B_{R/2})}\leq C(1+\|u\|_{C^{k,\a}(B_{R/2})})^{2k}.
\ee
If $u\in C^{0,1}(B_{R/2})$  then, there exists a constant $C=C(N,s,R)$ such that 
\be \label{eq:Lip-impL-infty-NMC-fund-lem}
\|\G^{u,R}\|_{C^{0,1}(B_{R/2})}\leq C(1+\|u\|_{C^{0,1}(B_{R/2})}).
\ee
\end{lemma}
\begin{proof}
For simplicity, we assume that $R=2$, and to alleviate the notations, we put $\G^u:=\G^{u, R}$.
We first observe, from \eqref{eq:def-of-F},  that $\cF_s'(p)=-(1+p^2)^{(-N-2s)/2}$, so that for all $j\in \N$,
\be \label{eq:pjFi-nmc}
 |p|^j| \cF_s^{(j)}(p)|\leq C(N,s,j) \qquad\textrm{ for all $p\in \R$. }
\ee 
In particular, since $2s>1$, 
\be \label{eq:Gu-L-infty}
 \|\G^u\|_{L^\infty(B_{1})}\leq C(N,s). 
\ee
Next, for all $(x,y)\in  B_{1} \times  \R^N\setminus B_2$,  we have
\begin{align}\label{eq:deriv-p-u-x-y-nmc-00}
|\de_{x}^\mu p_u(x,y)|\leq C(k)( |u(y)| |y|^{-1}+\|u\|_{C^{k-1,1}(B_1)} ) \qquad\textrm{ for $\mu\in \N^N$ with $|\mu|\leq k$}.
\end{align} 
On the other hand, by writing $u(y)=(u(y)-u(z))+ u(z)$, we easily deduce that 
\be \label{eq:u-1-nmc}
|u(y)||y|^{-1}\leq C (|p_u(z,y)|+ \|u\|_{L^\infty(B_1)})  \qquad\textrm{for   all $z\in  B_{1}$ and  $y\in  \R^N\setminus B_2$.}
\ee
Using this in \eqref{eq:deriv-p-u-x-y-nmc-00}, we see that, for $\mu\in \N^N$ with $|\mu|\leq k$, 
\begin{align}\label{eq:deriv-p-u-x-y-nmc}
|\de_{x}^\mu p_u(x,y)|\leq C(k)( |p_u(z,y)|+\|u\|_{C^{k-1,1}(B_1)} ) \qquad\textrm{for   all $x,z\in  B_{1}$ and  $y\in  \R^N\setminus B_2$  .}
\end{align}
 By  the  {Fa\`{a} di Bruno formula} (see e.g. \cite{FaadeBruno-JW}), for $|\g|= k$ and    $(x,y)\in B_{1} \times  \R^N\setminus B_2$,  we get 
\be \label{eq:Faa-de-Bruno}
\de^\g_x \cF_s(p_u(x,y)) = \sum_{\Pi\in\scrP_k} \cF_s^{ (\left|\Pi\right|)}(p_u(x,y)) \prod_{\mu \in\Pi} \de_x^{\mu } p_u(x,y) ,
\ee
 where $\scrP_k$ denotes the set of all partitions of  $\left\{ 1,\dots, k \right\}$. 
 Hence, for   $x\in B_{1}$ and $y\in \R^N\setminus B_2$, by  \eqref{eq:deriv-p-u-x-y-nmc},     we have that  
 \begin{align*}
 |\de^\g_x \cF_s(p_u(x,y) )| 
 &\leq  C \sum_{\Pi\in\scrP_k}  2^{\Pi-1} \left(  |p_u(x,y)|^{\left|\Pi\right|} \left|\cF_s^{ (\left|\Pi\right|)}(p_u(x,y)) \right|+\|u\|_{C^{k-1,1}(B_1)}^{|\Pi|} \left|\cF_s^{ (\left|\Pi\right|)}(p_u(x,y))\right|\right)\\
 &\leq   C\left(1+ \|u\|_{C^{k-1,1}(B_1)}^{k}+ \sum_{\Pi\in\scrP_k}  2^{\Pi-1}  |p_u(x,y)|^{\left|\Pi\right|} |\cF_s^{ (\left|\Pi\right|)}(p_u(x,y))| \right).
 \end{align*}
From this and \eqref{eq:pjFi-nmc},   we deduce that, for all $\g\in \N^N$ with $|\g|=k$,  
\be \label{es:estim-integrn-NMC-fund-lem}
\sup_{(x,y)\in B_{1}\times  \R^N\setminus B_2 } |\de^\g_x \cF_s(p_u(x,y))|+\sup_{(x,y)\in B_{1}\times  \R^N\setminus B_2} |\de^\g_x \cF_s(-p_u(x,y))|\leq C(1+   \|u\|_{C^{k-1,1}(B_{1})})^k,
\ee
with $C=C(s,N,k)$.
Since $2s>1$,  from the above estimate,  \eqref{eq:Gu-L-infty}  and the dominated convergence theorem, we can differentiate under the integral sign in \eqref{eq:def-g-nmc} to deduce that
\be 
\|\G^u\|_{C^{k-1,1}(B_{1})}\leq C(1+\|u\|_{C^{k-1,1}(B_{1})})^k.
\ee
Moreover, to see  \eqref{eq:Lip-impL-infty-NMC-fund-lem}, we note that if $u\in C^{0,1}(B_{1})$,  then Rademarcher's theorem implies that $u$ is equivalent to a  differentiable function. Therefore    \eqref{es:estim-integrn-NMC-fund-lem} holds (with $k=1$) and replacing "$\sup$" with "essup". Now  by the dominated convergence theorem, we get \eqref{eq:Lip-impL-infty-NMC-fund-lem}.\\
Let us now fix    $x_1,x_2\in B_{1}$ and $y\in    \R^N\setminus B_2$. Direct computations yield 
$$
|\de_{x}^\mu p_u(x_1,y)-\de_{x}^\mu p_u(x_2,y)|\leq C |x_1-x_2|^\a( |u(y)| |y|^{-1}+\|u\|_{C^{k,\a}(B_1)} ) \qquad\textrm{ for $\mu\in \N^N$ with $|\mu|\leq k$}.
$$
Note that,    \eqref{eq:u-1-nmc} implies that 
$$
|u(y)||y|^{-1}\leq C \{\min (|p_u(x_1,y)| , |p_u(x_2,y)|)+  \|u\|_{L^\infty(B_1)}\} .
$$
Therefore,  for all   $\mu\in \N^N$ with $|\mu|\leq k$, we get 
\begin{align}\label{eq:deriv-p-u-x-y-nmc-Hold}
|\de_{x}^\mu p_u(x_1,y)-\de_{x}^\mu p_u(x_2,y)|\leq C |x_1-x_2|^\a\{ \min (|p_u(x_1,y)| , |p_u(x_2,y)|)+\|u\|_{C^{k,\a}(B_1)} \} 
\end{align} 
and, by \eqref{eq:deriv-p-u-x-y-nmc},  
\begin{align}\label{eq:deriv-p-u-x-y-nmc-Bound}
|\de_{x}^\mu p_u(x_1,y) |\leq C  \{ \min (|p_u(x_1,y)| , |p_u(x_2,y)|)+\|u\|_{C^{k}(B_1)} \} .
\end{align}
Next, we define   
$$
g_s\in C^\infty(\R_+, \R), \qquad 
g_s(r)= -r^{-(N+2s-1)/2},
$$
so that  $\cF_s'(p )=g_s(1+ p ^2)$ for all $p\in \R$.  Moreover,  for $r>0$,  
\be \label{eq:higher-deriv-g-j}
  g^{(j)}_s(r)=(-1)^{j+1}2^{-j} \prod_{i=0}^{j-1 } (N+2s-1 +2 i) r^{-\frac{N+2s-1+2j}{2}}.
\ee
From this and the generalized  chain rule for higher derivatives, we get 
 \begin{align}
\cF^{(j+1)}_s(p)  = &\sum_{(m_1,m_2)\in \cN_j}\t_j(m_1,m_2) p^{m_1} g_s^{(m_1+m_2)}(1+p^2),
\label{eq:Dk-K-s_1s_2-0-pp}
 \end{align}
 where  $\t_j(m_1,m_2)=\frac{j! 2^{m_1}}{m_1! m_2!}$ and $\cN_j:=\{(m_1,m_2)\in \N^2\,:\,m_1+ 2^{m_2}m_2=j\}$.
Hence, for all $p_1,p_2\in \R$, 
  \begin{align*}
|\cF^{(j+1)}_s(p_1)-&\cF^{(j+1)}_s(p_2)  |\leq  \sum_{(m_1,m_2)\in \cN_j}\t_j(m_1,m_2) |p^{m_1}_1-p_2^{m_1}| |g_s^{(m_1+m_2)}(1+p^2_1)|\\
& +  \sum_{(m_1,m_2)\in \cN_j }\t_j(m_1,m_2)| p_2^{m_1}| |p_1^2-p_2^2| \int_{0}^1| g_s^{(m_1+m_2+1)}(1+t p^2_1+(1-t)p_2^2) \, | dt.
 \end{align*}
 It then follows from, \eqref{eq:deriv-p-u-x-y-nmc-Hold} and  \eqref{eq:deriv-p-u-x-y-nmc-Bound}, that 
\begin{align}
&\left|\cF^{(j+1)}_s(p_u(x_1,y)) -\cF^{(j+1)}_s(p_u(x_2,y))\right| \nonumber\\
&\leq C |x_1-x_2|^\a  \sum_{(m_1,m_2)\in \cN_j} \t_j(m_1,m_2)  \frac{ \left(\min (|p_u(x_1,y)| , |p_u(x_2,y)| )+\|u\|_{C^{k,\a}(B_1)}  \right)^{ m_1}}{   (1+  \min (|p_u(x_1,y)| , |p_u(x_2,y)| )^2)^{\frac{N+2s-1+2(m_1+m_2)}{2}} }  \nonumber\\
&+C |x_1-x_2|^\a  \sum_{(m_1,m_2)\in \cN_j} \t_j(m_1,m_2)  \frac{ \left(\min (|p_u(x_1,y)| , |p_u(x_2,y)| )+\|u\|_{C^{k,\a}(B_1)}  \right)^{ m_1+2}}{   (1+  \min (|p_u(x_1,y)| , |p_u(x_2,y)| )^2)^{\frac{N+2s-1+2(m_1+m_2+1)}{2}} }  .
\label{eq:estim-cFs-jplus-1-nmc}
\end{align}
On the other hand, it is immediate,  from  \eqref{eq:higher-deriv-g-j} and \eqref{eq:Dk-K-s_1s_2-0-pp},   that 
\be \label{eq:pjFi-nmc-ok}
|\cF^{(j+1)}_s(p_u(x_2,y)) | \leq  \frac{C}{(1+  \min (|p_u(x_1,y)| , |p_u(x_2,y)| )^2)^{\frac{N+2s-1+2(j+1)}{2}}}.
\ee
Using \eqref{eq:deriv-p-u-x-y-nmc-Bound},   \eqref{eq:deriv-p-u-x-y-nmc-Hold} and an induction argument,   we get 
\begin{align}\label{eq:prod-p-u-CNMC}
&\left|\prod_{\mu \in\Pi} \de_x^{\mu } p_u(x_1,y)- \prod_{\mu \in\Pi} \de_x^{\mu } p_u(x_1,y) \right| \leq  C   |x_1-x_2|^\a\{ \min (|p_u(x_1,y)| , |p_u(x_2,y)| )+\|u\|_{C^{k,\a}(B_1)} \}^{|\Pi|}.
\end{align}
Moreover  \eqref{eq:deriv-p-u-x-y-nmc-Bound} yields
\be \label{eq:Prodp-x2-y-nmc}
\left|\prod_{\mu \in\Pi} \de_x^{\mu } p_u(x_2,y)  \right| \leq C (    \min (|p_u(x_1,y)| , |p_u(x_2,y)|)+\|u\|_{C^{k}(B_1)} )^{|\Pi|}.
\ee
We have, from \eqref{eq:Faa-de-Bruno},  that 
\begin{align} \label{eq:Faa-de-Bruno-fund-calcul-nmc}
&\left| \de^\g_x \cF_s(p_u(x_1,y))- \de^\g_x \cF_s(p_u(x_2,y)) \right| \leq \sum_{\Pi\in\scrP_k} \left|\cF_s^{ (\left|\Pi\right|)}(p_u(x_1,y))  - \cF_s^{ (\left|\Pi\right|)}(p_u(x_2,y)) \right|  \left|\prod_{\mu \in\Pi} \de_x^{\mu } p_u(x_1,y) \right|  \nonumber \\
&\hspace{3cm}+ \sum_{\Pi\in\scrP_k} \left|\cF_s^{ (\left|\Pi\right|)}(p_u(x_2,y)) \right| \left|\prod_{\mu \in\Pi} \de_x^{\mu } p_u(x_1,y)- \prod_{\mu \in\Pi} \de_x^{\mu } p_u(x_2,y) \right|.
\end{align}
Next, we observe that for $(m_1,m_2)\in \cN_{|\Pi|-1}$,  then  
$$
|\Pi|+ m_1-2(m_1+m_2)-(N+2s-1)< 0.
$$
Now from this,  \eqref{eq:estim-cFs-jplus-1-nmc},  \eqref{eq:pjFi-nmc-ok}, \eqref{eq:prod-p-u-CNMC}, \eqref{eq:Prodp-x2-y-nmc} and \eqref{eq:Faa-de-Bruno-fund-calcul-nmc}, we deduce that,  for all  $x_1,x_2\in B_{1}$ and $y\in    \R^N\setminus B_2$,  
\begin{align*}
\left| \de^\g_x \cF_s(p_u(x_1,y) )- \de^\g_x \cF_s(p_u(x_2,y) )\right| \leq  C |x_1-x_2|^\a (   1+\|u\|_{C^{k,\a}(B_1)} )^{2k}.
\end{align*}
Combining this with  \eqref{es:estim-integrn-NMC-fund-lem}, we get $\sup_{y\in \R^N\setminus B_2}\|\cF_s(p_u(\cdot,y) )\|_{C^{k,\a}(B_1)}\leq 
 C (   1+\|u\|_{C^{k,\a}(B_1)} )^{2k}.$ Since the same estimates remains valid when $p_u$ is replaced with $-p_u$, then \eqref{eq:estimd-Gu-Holder-NMC} follows.
\end{proof}
%
%
%
We will need the following elementary  result which follows from the fact that $\cF'_s$ is even on $\R$ and the fundamental theorem of calculus.
\begin{lemma}\label{lem:local-comparison-NMC-graphs} For all $a,b\in \R$, we have 
$$
[\cF_s(a)-\cF_s(b)]- [\cF_s(-a)-\cF_s(-b)]=2(a-b)\int_0^1 \cF'_s\left(b+\rho(a-b)\right) d\rho.
$$
\end{lemma}
We now  complete the
\begin{proof}[Proof of Theorem \ref{eq:thm-nmc-reg}] In view of \eqref{eq:decom-NMC-intro}, we have 
\be \label{eq:Ck_u-u-sol-NMC}
\cL_{\ti \cK_u} u= f-\G^u  \qquad\textrm{ in $B_{1/2}$,}
\ee
where  
$$
\ti\cK_u(x,y):=1_{B_1}(x) 1_{B_1}(y)\cK_u(x,y) \qquad\textrm{ for all $x\not=y\in \R^N$}
$$
and, for $x\in B_{1/2}$, 
\begin{align*} 
\G^u(x)
&:=  \int_{\R^N} (1-   1_{B_1}(y)) \frac{\cF_s(p_u({x},{y}))- \cF_s(-p_u({x},{y})) }{|{x}-{y}|^{N+2s-1}} \, dy.
\end{align*}
We recall from the fundamental theorem of calculus that 
\be \label{eq:lin-to-Quasilin}
 (u(x)-u(y)) \cK_u(x,y)=  [\cF_s(p_u({x},{y}))- \cF_s(-p_u({x},{y}))] {|{x}-{y}|^{-(N+2s-1)}}.
\ee
Let $h\in B_{1/4} $ with $h\not=0$.  Then, recalling the notation in \eqref{eq:def-f-h-alph},  by Lemma \ref{lem:local-comparison-NMC-graphs}, \eqref{eq:lin-to-Quasilin} and \eqref{eq:Ck_u-u-sol-NMC},  
$$
\cL_{K^u_{h}} u_{h,1}= f_{h,1}+\G^u_{h,1} \qquad\textrm{ in $B_{1/4}$},
$$
where
\be \label{eq:def-Kuh}
  K^u_{h}(x,y):= 1_{B_1}(x) 1_{B_1}(y)  \frac{1}{{|x-y|^{N+2s}}}  q^u_{h}(x,y)  
\ee 
and 
$$
q^u_{h}(x,y):=-2  \int_{0}^{1}\cF'_s\left(p_{u(\cdot+h)}(x,y)+\rho p_{u-u(\cdot+h)}(x,y)\right)\, d\rho.
$$
Since $\cF'_s$ is  even and $p_w(x,y)=-p_w(y,x)$, we see that $K^u_{h}(x,y)=K^u_{h}(y,x)$. Moreover
\be\label{eq:K-u-h-lower-bound}
 K^u_h(x,y)\geq C |x-y|^{-N-2s}\qquad\textrm{$x\not=y\in B_1 $, }
\ee
for some constant  $C>0$, only depending on $N,s$ and $  \| u\|_{C^{0,1} (B_1) }$.
Letting $v:=\vp_{1/8}u_{h,1}$  and using   Lemma  \ref{lem:estim-G_Ku}$(i)$,   we have that 
\be\label{eq:Kuh-v-e-NMC} 
\cL_{K^u_{h}} v= f_{h,1}+ \G^u_{h,1}+ G_h,   \qquad\textrm{ in $B_{2^{-4}}$},
\ee
with  $G_h:=G_{ K^u_{h},u_{h,1},1/4}$  satisfying (note that $K^u_h$ is  supported in $B_1\times B_1$ and $|q^u_{h}|\leq 2$)
\be \label{eq:estim-G-cnmc}
\|G_h \|_{L^{\infty}(B_{2^{-5}})}\leq C(N,s)\|u_{h,1}\|_{L^\infty(B_1)}\leq C\|\n u\|_{L^\infty(B_2)}.
\ee
 We would like to apply \cite[Theorem 2.4]{Cozzi} to get the $C^{\a_0}$ bound of $v$, but our kernel $K^u_{h}$, which is compactly supported might vanish at some  diagonal points $\{x=y\}$.  A way out to such difficulty, is to use the argument in \cite[Remark 2.1]{Fall-Reg} (see also Lemma \ref{lem:cut-reg-kernel}) by considering 
 $$\ov K^u_h(x,y)=K^u_h(x,y)+ (2-1_{B_{1/2}}(x)-1_{B_{1/2}}(y)) |x-y|^{-N-2s}.$$ We then  deduce, from  \eqref{eq:Kuh-v-e-NMC},   that 
\be\label{eq:Kuh-v-e-NMC-first} 
\cL_{\ov K^u_{h}} v+ \ov V v= f_{h,1}+ \G^u_{h,1}+ G_h+ \ov f,   \qquad\textrm{ in $B_{2^{-4}}$},
\ee 
for some functions $\|\ov V\|_{L^\infty(B_{ 2^{-4} })}\leq C(N,s)$ and $ \|\ov f\|_{L^\infty(B_{ 2^{-4} })}\leq C(N,s)\|v\|_{L^\infty(B_1)}$.  Since $K^u_h(x,y) $ satisfies \eqref{eq:K-u-h-lower-bound}, we  find that   $\ov K^u_{h}$ satisfies \eqref{eq:K-Kernel-satisf}.
Therefore, since $v\in H^s(\R^N)\cap L^\infty(\R^N)$,  by  \cite[Theorem 2.4]{Cozzi}, we have that
$$
\|v\|_{C^{0,\a_0}(B_{2^{-5}})} \leq C ( \|v\|_{L^\infty (\R^N) }  + \|f_{h,1} \|_{L^\infty(\R^N)}+  \|\G^u_{h,1} \|_{L^\infty(B_{2^{-4}})}+ \|G_h\|_{L^\infty  (B_{2^{-4}}   ) } ),
$$
for some $\a_0>0$ and $C>0$, only depending on $N,s$ and $  \| u\|_{C^{0,1} (B_1) }$. From  \eqref{eq:estim-G-cnmc} and the fact that $v=u_{h,1}$ on $B_{1/8}$, we  get
$$
\| u_{h,1}\|_{C^{0,\a_0}(B_{2^{-5}})} \leq C ( \|u_{h,1}\|_{L^\infty (B_1) }+  \|f_{h,1} \|_{L^\infty(\R^N)}+  \|\G^u_{h,1} \|_{L^\infty(B_{2^{-4}})} ).
$$
This and Lemma \ref{lem:Lem-Gamma-u-nmc} imply that
\be \label{eq:esimt-u-CNM1}
\|u\|_{C^{1,\a_0}(B_{2^{-6}})} \leq C (1+ \| u\|_{C^{0,1} (B_1) } + \|f \|_{C^{0,1}(\R^N)}),
\ee
which proves \eqref{eq:estimu-NMC-first}. \\
To obtain the gradient estimate of $v$ from Theorem \ref{th:main-th10}, we   check that $\cL_{K^u_h}$ is a  $C^{0,\a_0}$-nonlocal operator. To this scope,   for every  $w\in C^{0,1}(B_1)$, we define  $Z_w: B_{1/2}\times [0,1/2)\times S^{N-1}\to \R$ by 
$$
Z_w(x,r,\th): =-\int_0^1\n w(x+rt\th)\cdot \th dt  \qquad \textrm{ for $r\in [0, 1/2)$, $x\in B_{1/2}$ and $\th\in S^{N-1}$}.
$$
Clearly $Z_w$ is as smooth as $\n w$ and $Z_w(x,r,\th):=p_w(x,x+r\th)$ for $r>0$. We then define ${\cA}_{    K^u_{h}}: B_{1/4}\times [0,1/4)\times S^{N-1}\to \R$ by  
\begin{align}\label{eq:AK-CNMC}
{\cA}_{    K^u_{h}}(x,r,\th):=1_{B_2}(x) 1_{B_2}(x+r\th) \int_{0}^{1}\frac{2 d\rho}{\left(1+  (Z_{u(\cdot+h)}(x,r,\th)+\rho Z_{u-u(\cdot+h)}(x,r,\th))^2  \right)^{(N+2s)/2}} ,
\end{align}
which by \eqref{eq:def-Kuh}, satisfies    ${\cA}_{    K^u_{h}}(x,r,\th)= r^{N+2s}K^u_h(x,x+r\th)$ for all $(x,r,\th)\in B_{1/4}\times (0,1/4)\times S^{N-1}$.
Moreover, 
$$
{\cA}_{    K^u_{h}}(x,0,\th)-{\cA}_{    K^u_{h}}(x,0,-\th)=0 \qquad\textrm{ for all $(x,\th)\in B_{1/4}\times S^{N-1} $.}
$$
In addition from,   \eqref{eq:esimt-u-CNM1} together with  \eqref{eq:AK-CNMC}, we have  that 
$$
\|{\cA}_{    K^u_{h}}  \|_{C^{\a_0}(Q_{2^{-7}}\times S^{N-1}    )}\leq C,
$$
with $C$, only depending on $N,s,\|u\|_{ C^{0,1}(B_2) },\a_0$ and  $\|f \|_{C^{0,1}(B_2)}$.
We then conclude  that $   K^{u}_{h}\in \scrK^s(\k,\a_0, Q_{2^{-7}})$, for some $\k$, only depending on $N,s,\|u\|_{ C^{0,1}(B_2) },\a_0$ and $ \|f \|_{C^{0,1}(B_2)}$.
 Therefore applying Theorem \ref{th:main-th10}$(ii)$ to   \eqref{eq:Kuh-v-e-NMC}, we deduce that 
 $$
\| \n v\|_{C^{\min(2s-1-\e,\a_0 )}(B_{2^{-8}})} \leq C ( \|v\|_{L^\infty (B_1) }+ \|f_{h,1} \|_{L^\infty (B_2)}),
$$
for all $\e\in (0,2s-1)$ and $C$ a constant,  only depending on $N,s,\|u\|_{ C^{0,1}(B_2) },\a_0,\e$ and $ \|f \|_{C^{0,1}(B_2)}$. 
Hence, recalling that $v=u_{h,1}$ in $B_{1/8}$, we get  
$$
\| \n u\|_{C^{1, \a_1}(B_{2^{-9}})} \leq C ,  
$$
with $\a_1:=\min(2s-1-\e,\a_0 )$. Hence, for all $h\in B_{2^{-10}}$, we have   $    K^{u}_{h}\in \scrK^s(\k,1, Q_{2^{-10}})$, for some $\k$, only depending on $N,s,\|u\|_{ C^{0,1}(B_2) },\a_1$ and $ \|f \|_{C^{0,1}(B_2)}$.
We apply once more Theorem \ref{th:main-th10}$(ii)$ to   \eqref{eq:Kuh-v-e-NMC}, to   get
$
\| v\|_{C^{ 1,2s-1-\e }(B_{2^{-11}})}  \leq C,
$
so that
\be \label{eq:estim-CNMC11}
\| u\|_{C^{ 2,2s-1-\e }(B_{2^{-12}})}  \leq C.
\ee
This finishes the   proof of $(i)$ after a scaling and covering.\\

For $(ii)$,   we consider first the case $m=1$. Clearly \eqref{eq:estim-CNMC11} and \eqref{eq:AK-CNMC} imply that  $    K^{u}_{h}\in \scrK^s(\k,2s-1+\a, Q_{2^{-13}})$,  for all $h\in B_{2^{-13}}$ and $\a\in (0,1)$.  In particular,  by Lemma  \ref{lem:estim-G_Ku}$(iii)$, we have $\|G_h\|_{C^{0,\a}(B_{2^{-13}})}\leq C \|u_{h,1}\|_{L^\infty(B_1)}$.
Now by \eqref{eq:estim-CNMC11} and Lemma \ref{lem:Lem-Gamma-u-nmc},    for all $h\in B_{2^{-13}}$, we have  $ \|\G^u_{h,1}\|_{ C^{1,2s-1-\e}(B_{2^{-13}})}\leq C $. Therefore, 
  applying Theorem \ref{th:Schauder-0-intro} to the equation \eqref{eq:Kuh-v-e-NMC},  we get 
$
\| v\|_{C^{ 2s+\a }(B_{2^{-15}})}  \leq C,
$
provided $2s+\a\not\in \N$.
Hence 
$$
 \| u\|_{C^{ 1+2s+\a }(B_{2^{-16}})}  \leq C.
 $$
  If now $m\geq 2$, then the above estimate implies that $    K^{u}_{h}\in \scrK^s(\k,2s+\a, Q_{2^{-18}})$ for all $h\in B_{2^{-18}}$. Hence,   Lemma  \ref{lem:estim-G_Ku}$(iii)$ implies that $\|G_h\|_{C^{1,\a}(B_{2^{-18}})}\leq C \|u_{h,1}\|_{L^\infty(B_1)}$. On the other hand, by  Lemma \ref{lem:Lem-Gamma-u-nmc},   $  \G^u_{h,1}\in C^{2s+\a}(B_{2^{-16}})\subset  C^{1,\a}(B_{2^{-16}})$, because $2s>1$. It then follows, from  \eqref{eq:Kuh-v-e-NMC} and   Theorem \ref{th:Schauder-0-intro},    that $\| \de_{x_i} v\|_{C^{ 2s+\a }(B_{2^{-18}})}  \leq C $. This yields $\| \de_{x_i} u\|_{C^{ 1+2s+\a }(B_{2^{-19}})}  \leq C $, because $v=\vp_{1/8} u_{h,1}$.
 Now iterating the above argument, then   for all $k\in\{1,\dots,m\}$ and $i=\{1,\dots,N\}$, we can  find two constants   $r_k$, only depending on $k$, and a constant $ C_k>0$, only depending on $N,s,\|u\|_{ C^{0,1}(B_2) },k,\a$ and $ \|f \|_{C^{m,\a}(B_2)}$, such that 
 $$
 \| \de_{x_i}^k u_{h,1}\|_{C^{ 2s+\a }(B_{r_k})}  \leq C_k 
 $$ for all $h\in B_{r_k/2}$. A covering and scaling arguments yield $(iii)$.
\end{proof}
%
%
%
  \subsection{Proof of Theorem \ref{th:nonloca-surf1} and Theorem \ref{th:nonloca-surf2} }
Up to a change of coordinates and a scaling, we may assume that  a neighborhood of   $0\in \Sig$ is  parameterized by a $C^{1,\g}$-diffeomorphism $\Phi: B_2\to \Sig$, for some $\g\in (0,1)$,  satisfying $ \Phi(0)= 0$ and 
\be\label{eq:DPhi-close-identity}
|D \Phi(x)-id|\leq \frac{1}{2} \qquad\textrm{ for all $x\in B_2$.} 
\ee
  We consider the following open sets in $\Sig$ given by
 $$
   \cB_r:=\Phi(B_r) \qquad\textrm{ for $r\in (0,2]$}
   $$
   and  we define $\eta_r(\ov x)=\vp_r(\Phi^{-1}(\ov x))$. For $\Psi\in C^\infty_c(\cB_{1/2})$,  we then we have 
  \begin{align}
  \int_{ \cB_2} \int_{ \cB_2}  \frac{(u(\ov x)-u(\ov y))(\Psi(\ov x)-\Psi(\ov y)) }{|\ov x-\ov y|^{N+2s}}\eta_2(\ov x)\eta_2(\ov y) \, d\s(\ov x)d\s(\ov y)&+\int_{\Sig} V_1(\ov x) u(\ov x) \Psi(\ov x)\, d\s(\ov x)  \nonumber\\
  & =\int_{\Sig} f_1(\ov x) \Psi(\ov x)\, d\s(\ov x), \label{eq:reg-Manifold-fin}
  \end{align}
where
\be  \label{eq:def-V_11}
V_1(\ov x):=V(\ov x)+  \int_{\Sig} (1-\eta_2(\ov y)) |\ov x-\ov y|^{-N-2s}\, d\s(\ov y)
\ee
 and 
 \be  \label{eq:def-V_12}
  f_1(\ov x):=f(\ov x)+ \int_{\Sig} (1-\eta_2(\ov y))u(\ov y) |\ov x-\ov y|^{-N-2s}\,d\s(\ov y).
\ee 
We denote by $Jac_\Phi$   the Jacobian determinant of $\Phi $.
Let  $\psi(x)=\Psi(\Phi(x))$,  $v(x)=u(\Phi(x))$, $\ti V(x)=V_1(x)Jac_\Phi(x)$ and  $\ti f(x)=f_1(x)Jac_\Phi(x)$.  Then by the  changes of variables $\ov x=\Phi (x)$ and $\ov y=\Phi (y)$, in \eqref{eq:reg-Manifold-fin},  we get 
  \begin{align*}
\frac{1}{2}  \int_{ \R^N} \int_{ \R^N}  {(v(x)-v(y))(\psi(x)-\psi(y)) } K(x,y) \, dxdy+\int_{B_1} \ti V(x) u(x) \psi(x)\, dx =\int_{B_1} \ti f(x) \psi(x)\, dx,
  \end{align*}
where 
\be\label{eq:def-K-hyersurf}
K(x,y) =  {\vp_2(x) \vp_2(y) Jac_{\Phi}(x)Jac_{\Phi}(y) }  {|\Phi(x)-\Phi(y)|}  ^{-N-2s}.
\ee
We further consider $w=\vp_{1/4} v\in H^s(\R^N)$, so that  by Lemma \ref{lem:estim-G_Ku},  
\be\label{eq:w-satisf}
\cL_{K} w+ \ti V w= \ti f+ G \qquad\textrm{in $B_{1/16}$}, 
\ee
where 
\be\label{eq:def-G-nonloc-hypersurface} 
G(x)=\int_{B_2}(1-\vp_{1/4}(y))v(y)K(x,y)\, dy.
\ee
Next, we observe  that the function   ${\cA}_{K}:B_{1}\times [0,1]\times S^{N-1}\to \R^N$, given  by 
$$
{\cA}_{K}(x,r,\th)=\vp_2(x) \vp_2(x+r\th) Jac_{\Phi}(x)Jac_{\Phi}(x+r\th)  \left|\int_0^1D\Phi(x+r\th)  \th\,  dt\right| ^{- N-2s}
$$
is an extension of $(x,r,\th)\mapsto r^{N+2s}K(x,x+r\th)$ on $B_{1}\times [0,1]\times S^{N-1}$.   
%
%
Moreover, since $\Phi \in C^{1,\g}(B_2)$,  we see that 
\be \label{eq:K-hypersurface}
\begin{aligned}
& \|{\cA}_{K} \|_{C^\g(B_{1/2}\times[0,1/2]\times S^{N-1})}\leq \frac{1}{\k},\\\
&{\cA}_{K}(x,0,\th)={\cA}_{K}(x,0,-\th) \qquad\textrm{ for all $(x,\th)\in \R^N\times S^{N-1} $},
%
\end{aligned}
\ee
for some $\k>0$, only depending on  $N,s,\g$ and $\|\Phi\|_{C^{1,\g}(B_2)}$. Consequently by \eqref{eq:K-hypersurface}, \eqref{eq:DPhi-close-identity} and   \eqref{eq:def-K-hyersurf}, decreasing $\k$ if necessary, we see that      $K\in \scrK^s(\k,\g, Q_{1/2} )$.
  In addition, from \eqref{eq:def-V_11} and \eqref{eq:def-V_12}, we easily deduce that for  $p>1$,
\be\label{eq:Holder-entires-manifold}
\|\ti f\|_{L^p(  B_{1/16}   )}+ \|G\|_{L^p(  B_{1/16}   )}     \leq C (\|u\|_{{L_s(\Sig)}}+ \|f\|_{L^p(\cB_2)}) \qquad\textrm{ and } \qquad\|\ti V\|_{L^p(  B_{1/2}   )}   \leq C,
\ee
where  $C$ is a constant only depending on $N,s,p,\g,    \|\n \Phi\|_{C^{1,\g}(B_2)}, \|V\|_{L^p(\cB_2)}$ and $ \|1\|_{{L_s(\Sig)}}$.
%
\begin{proof}[Proof of Theorem \ref{th:nonloca-surf1} (completed)]
From the computations above, we have that $w=\vp_{1/2}u\circ\Phi \in H^s(\R^N)$ satisfies \eqref{eq:w-satisf} with $K\in \scrK^s(\k,\g, Q_{1/2} )$.   Thanks to \eqref{eq:Holder-entires-manifold},  we can apply Theorem \ref{th:main-th10}, to get the result.
\end{proof}
%

%
%
%
%
%
%
 \begin{proof}[Proof of Theorem \ref{th:nonloca-surf2} (completed)]
We know from Theorem \ref{th:nonloca-surf1} and the above argument  that  $w=\vp_{1/2}u\circ\Phi \in H^s(\R^N)\cap C^{\min (2s-\e,1)}(B_{1/4})$, for all $\e\in (0,2s)$ and solves  \eqref{eq:w-satisf} with  $K\in \scrK^s(\k,\g, Q_{1/2} )$. However, in view of \eqref{eq:def-V_11} and \eqref{eq:def-V_12},  we can use similar arguments as in the proof of Lemma \ref{lem:estim-G_Ku}$(iv)$ to deduce that  
\be \label{eq:es1-Hyp} 
\|\ti f\|_{C^\a(  B_{1/16}   )}  \leq C (\|u\|_{{L_s(\Sig)}}+ \|f\|_{C^\a(\cB_2)})
\ee
and, using also \eqref{eq:integ-hypersurface}, we get
\be  \label{eq:es2-Hyp} 
\|\ti V\|_{C^\a(  B_{1/16}   )}  \leq C (\| V\|_{C^\a(  \cB_{2}   )} +  \|1\|_{{L_s(\Sig)}} ) \leq C  ,
\ee
where here and below, the letter $C$ denotes a positive constant which may vary from line to line but only depends on  $N,s,\a,\g,  \|V\|_{C^\a(\cB_2)}, \|\n \Phi\|_{C^{1,\g}(B_2)}$ and $ \|1\|_{{L_s(\Sig)}}$.
Moreover, recalling \eqref{eq:def-G-nonloc-hypersurface},  applying Lemma \ref{lem:estim-G_Ku}$(iii)$, we have that  
\be \label{eq:es3-Hyp} 
\|G\|_{C^\a(  B_{1/16}   )}  \leq C  \|w\|_{L^1(B_2)} \leq C \|u\|_{{L_s(\Sig)}}.
\ee
In view of   \eqref{eq:w-satisf},   \eqref{eq:es1-Hyp},  \eqref{eq:es2-Hyp} and  \eqref{eq:es3-Hyp}, we can  apply Theorem \ref{th:Schauder-0-intro}-$(i)$ and use a bootstrap argument, to deduce  that 
\begin{align*}
\|w\|_{C^{2s+\a}(B_{r_0}) } &\leq C (\|w\|_{L^\infty(\R^N)}+   \|\ti f\|_{C^\a(B_2)}+ \|G\|_{C^\a(B_{1/16})})\\
&\leq C (\|u\|_{L^2(\cB_2)}+ \|u\|_{{L_s(\Sig)}} +\|f\|_{C^\a(\cB_2)}),
\end{align*}
for some $r_0,C>0$,   depends only on  $N,s,\a,\g,c_1,c_0,\|V\|_{C^\a(\cB_2)}, \|\n \Phi\|_{C^{1,\g}(B_2)}$ and $I_{s,\Sig}$.
The proof is thus completed by scaling, covering and a  change of variables.
 \end{proof}

\section{Appendix}\label{s:Appendix}
%
\begin{proof}[Proof of Lemma \ref{lem:2sp-alph-estim-F_Keo}]
\noindent
\textbf{Case $2s+\a<2$.}
For simplicity, recalling \eqref{eq:A-Cm12} and \eqref{eq:A-Cm12-tau}, we assume that 
$$
 \|A\|_{C^{k+2s+\a}(Q_\infty)\times L^\infty(S^{N-1} )}+  \|B\|_{\cC^{k}_{\t_s}(Q_\infty )\times L^\infty(S^{N-1} )} \leq 1,
$$
where  $\t_s:=\a+(2s-1)_+$ if $2s\not=1$ and  $\t_{1/2}:=\a + \e$ if $2s=1$. We also assume that 
$$
\|u\|_{C^{k+2s+\a+\e_s}(\R^N )}\leq 1,
$$
where $\e_s:=0$ if $2s\not=1$ and $\e_s:=\e$ if $2s=1$.\\
We consider the case $m=0$.
Since $\|u\|_{L^\infty(\R^N)}\leq 1$,   we have
\be\label{eq:Apend-1e}
|\d^e u(x,r,\th )| \leq C \min (1,r^{2s+\a}).
\ee
Here,   for $2s\geq 1$, we use the fact that $ 2 \d^e u(x,r)=r \int_0^1(\n u(x+t r \th )-\n u(x-t r \th ))\cdot \th \, dt $.
Moreover for $x_1,x_2\in \R^N$ and $r>0$, then for $2s+\a<1$,  we have 
\be\label{eq:Apend-2e}
|\d^e u(x_1,r,\th )-\d^e u(x_2,r,\th )| \leq C \min (r^{2s+\a}, |x_1-x_2|^{2s+\a})
\ee
 and if $2s\geq 1$, we have 
\be\label{eq:Apend-3e}
|\d^e u(x_1,r,\th )-\d^e u(x_2,r,\th )| \leq C \min (r^{2s+\a}, r |x_1-x_2|^{\t_s}).
\ee
On the other hand, for all $s\in (0,1)$,  
\be\label{eq:Apend-4o}
|\d^o u(x,r,\th) | \leq C \min (1,r  )^{\min(2s+\a,1)},
\ee
and 
\be\label{eq:Apend-5o}
|\d^o u(x_1,r,\th )-\d^o u(x_2,r,\th )| \leq C \min (r,  |x_1-x_2| )^{\min(2s+\a,1)}.
\ee
Using \eqref{eq:Apend-4o},  for  $s\in (0,1)$,  we estimate
\begin{align*}
|\cO^s_{B,u}(x)|&\leq C \int_0^\infty\min (r,1)^{\min (2s+\a,1)}    \min (r, 1)^{(2s-1)_++\a}    r^{-1-2s}\, dr\\
&\leq C\int_0^1r^{\min (2s+\a,1)}  r ^{(2s-1)_++\a}   r^{-1-2s}\, dr+  C   \int_1^\infty     r^{-1-2s}\, dr,
\end{align*}
so that, 
\be\label{eq:estimFoB-k-1}
\|\cO^s_{B,u}\| _{L^\infty(\R^N)}\leq C.
\ee
We consider next $\cE^s_{A,u}$. For all $x\in \R^N$ and for all $s\in (0,1)$, by \eqref{eq:Apend-1e},  we have
\begin{align*}
|\cE^s_{A,u}(x)|&\leq C \int_0^\infty\min( r^{2s+\a}  ,1)    r^{-1-2s}\, dr \leq C \int_0^1r^{\a-1}   \, dr+ C \int_1^\infty     r^{-1-2s}\, dr,
\end{align*}
yielding
\be\label{eq:estimFeB-k-1}
\|\cE^s_{A,u}\| _{L^\infty(\R^N)}\leq C.
\ee
  Let $x_1, x_2\in \R^N$ with    $|x_1-x_2|\leq 1$. Using \eqref{eq:Apend-5o}, for $s\in (0,1)$  we have
\begin{align*}
|\cO^s_{B,u}(x_1)&-\cO^s_{B,u}(x_2)|\leq C \int_0^\infty\min (r,|x_1-x_2|)^{\min (2s+\a,1)}    \min (r, 1)^{\t_s} r^{-1-2s}\, dr\\
&+C \int_0^\infty\min (r,1)^{\min (2s+\a,1)}    \min (r, |x_1-x_2|)^{\t_s}  r^{-1-2s}\, dr\\
&\leq C\int_0^{|x_1-x_2|}r^{\min (2s+\a,1)+\t_s}   r^{-1-2s}\, dr +C|x_1-x_2|^{\min (2s+\a,1)  }  \int_{|x_1-x_2|}^1 r^{\t_s -1-2s  } \, dr\\
&+C|x_1-x_2|^{\t_s   }  \int_{|x_1-x_2|}^1 r^{\min (2s+\a,1)-1-2s  } \, dr\\
&+C |x_1-x_2|^{\min (2s+\a,1)  }    \int_{1}^\infty     r^{-1-2s}\, dr+C |x_1-x_2|^{ \t_s }    \int_{1}^\infty     r^{-1-2s}\, dr\\
&\leq C |x_1-x_2|^{\a} .
\end{align*}
In the above estimate, it is used that $\t_{s}=\a+\e$, for $s=1/2$.
This together with \eqref{eq:estimFoB-k-1}    imply that $\|\cO^s_{B,u}\|_{C^{0,\a}(\R^N)}\leq C $, for all $s\in (0,1)$. \\
Now for $2s\geq  1$, let  $x_1 \not = x_2\in \R^N$ with   $|x_1-x_2|\leq 1$. Using \eqref{eq:Apend-3e} and \eqref{eq:Apend-1e}     we have 
\begin{align*}
&|\cE^s_{A,u}(x_1)-\cE^s_{A,u}(x_2)|\\
& \leq C \int_0^\infty \min (r^{ 2s+\a} , r|x_1-x_2|^{ \t_s} )    r^{-1-2s }\, dr  +C |x_1-x_2|^{\a}   \int_0^\infty \min(r ^{ 2s+\a},1 )    r^{-1-2s}\, dr\\
&\leq C \int_0^{|x_1-x_2|}   r^{\a-1}\, dr+C|x_1-x_2|^{\t_s}\int_{|x_1-x_2|}^\infty    r^{-2s}\, dr +C |x_1-x_2|^{\a}  \leq C |x_1-x_2|^\a.
\end{align*}
Hence using \eqref{eq:estimFeB-k-1},  for $2s\geq 1$,  we get $\|\cE^s_{A,u}\|_{C^{0,\a}(\R^N)}\leq C $.\\
We now consider the case $2s+\a<  1$.   For  $x_1, x_2\in \R^N$,   $|x_1-x_2|\leq 1$, by  \eqref{eq:Apend-2e},   we estimate
\begin{align*}
&|\cE^s_{A,u}(x_1)-\cE^s_{A,u}(x_2)|\\
&\leq  C\int_0^\infty\min (r,|x_1-x_2|)^{ 2s+\a}     r^{-1-2s}\, dr  +C|x_1-x_2|^{\a}   \int_0^\infty \min(r ^{ 2s+\a},1 )    r^{-1-2s}\, dr\\
&\leq C \int_0^{|x_1-x_2|}r^{-1+\a}    \, dr+C|x_1-x_2|^{2s+\a}\int_{|x_1-x_2|}^\infty    r^{-1-2s}\, dr+C |x_1-x_2|^{\a}   \leq C |x_1-x_2|^\a.
\end{align*}
We then conclude from this and \eqref{eq:estimFeB-k-1} that   $\|\cE^s_{A,u}\|_{C^{0,\a}(\R^N)}\leq C $, provided    $2s+\a<1$.\\
If  $m>1$, we can use the Leibniz formula  for the derivatives of the product of two functions. Note that for all $\g\in\N^N$ with $|\g|\leq m$, we have  that $\d^e \de^\g u$ (resp. $\d^o \de^\g u$) satisfies \eqref{eq:Apend-1e} and  \eqref{eq:Apend-2e} (resp. \eqref{eq:Apend-4o} and  \eqref{eq:Apend-5o}).\\

\noindent
\textbf{Case $2s+\a>2$}.  We first observe from the arguments in the previous case that 
\begin{align}\label{eq:L-infty-bound-cEcO-higher}
\|\cE_{A,u}^s\|_{L^{\infty}(\R^N)}\leq  C \|A\|_{C^{0}(Q_\infty) \times L^\infty(S^{N-1} )}  \|u\|_{C^{2s+\a}(\R^N)},\nonumber\\
  \|\cO_{B,u}^s\|_{L^{\infty}(\R^N)} \leq  C \|A\|_{\cC^{0}_{1}(Q_\infty)\times L^\infty(S^{N-1} )}  \|u\|_{C^{2s+\a}(\R^N)} .
\end{align}
  Since $B(y,0,\th) =0$, we  have 
$$
B(x_1,r,\th) - B(x_2,r,\th)= r\int_0^1  (D_rB(x_1,\varrho  r,\th) - D_r B(x_2,\varrho  r,\th))  \,  d\varrho.
$$
On the other hand
$$
B(x_1,r,\th) - B(x_2,r,\th)=  \int_0^1  D_x B(\varrho x_1+(1-\varrho) x_2,
 r,\th)\cdot (x_1-x_2)   \,  d\varrho.
$$
The above two estimates  yield
\begin{align}\label{eq:B-x1-x2-High}
 |B(x_1,r,\th) - B&(x_2,r,\th)|\leq   (  \|B\|_{C^{2s+\a-1}(Q_\infty)\times L^\infty(S^{N-1} )} \nonumber\\
 &+  \|B\|_{\cC^1_{2s+\a-2} (Q_\infty) \times L^\infty(S^{N-1} )}  ) \min (r|x_1-x_2|^{2s+\a-2},  r^{2s+\a-2} |x_1-x_2|  ).
\end{align}
In addition, we have 
\begin{align*}
 \d^o u(x_1,r,\th) - \d^o u(x_2,r,\th) = \int_0^1   D_x   \d^ou(\varrho x_1+(1-\varrho) x_2,
 r,\th)\cdot (x_1-x_2)    d\varrho, 
\end{align*}
  so that 
$$
| \d^o u(x_1,r,\th) - \d^o u(x_2,r,\th) ) | \leq  C \|u\|_{C^{2s+\a}(\R^N)}  \min (r, r^{2s+\a-2} |x_1-x_2|).
$$
Using this and \eqref{eq:B-x1-x2-High}, we find that, for all $x_1,x_2\in \R^N$, 
\begin{align}\label{eq:cOB-Holder-Higher}
|\cO_{B,u}^s(x_1)- \cO_{B,u}^s(x_2)| \le  C (  \|B\|_{C^{2s+\a-1}(Q_\infty)\times L^\infty(S^{N-1} )}+  \|B\|_{\cC^1_{2s+\a-2} (Q_\infty) \times L^\infty(S^{N-1} )}  ) |x_1-x_2|^\a.
\end{align}
Next, we write $2 \d^e u(x,r)=r \int_0^1(\n u(x+t r \th )-\n u(x-t r \th ))\cdot \th \, dt$ from which we deduce that
\begin{align*}
 \d^e u(x,r)&=  r^2 \int_0^1 t \int_0^1 D^2_x \d^o u(x, \varrho  t r,  \th ) [\th,\th]  \,  d\varrho dt
\end{align*}
and 
\begin{align*}
 \d^e u(x_1,r)-  \d^e u(x_2,r)&=r  \int_0^1 \int_0^1D^2_x \d^o u( \varrho x_1+ (1-\varrho) x_2,t r, \th )[x_1-x_2,\th] \, dt  d\varrho.
\end{align*}
By combining the above two estimates, we get 
$$
| \d^e u(x_1,r)- \d^e u(x_2,r)| \leq C  \|u\|_{C^{2s+\a}(\R^N)}  \min (r^{2s+\a}, r^{2s+\a-1} |x_1-x_2| ).
$$
Using now the above estimate and the fact  that $A\in {C^{\a}(Q_\infty) \times L^\infty(S^{N-1} )}$, we immediately deduce that  $[\cE_{A,u}^s]_{C^\a(\R^N)}\leq C \|A\|_{C^{\a}(Q_\infty)\times L^\infty(S^{N-1} )}  \|u\|_{C^{2s+\a}(\R^N)} $. From this,   \eqref{eq:L-infty-bound-cEcO-higher} and \eqref{eq:cOB-Holder-Higher}, we get the statement in the lemma for $m=0$ and $2s+\a>2$. 
In the general case that $m\geq 1$,  we can use the Leibniz formula  for the derivatives of the product of two functions and argue as above to get the desired estimates.
\end{proof}

\end{document}